\documentclass[10pt,bezier]{article}
\usepackage{amsmath,amssymb,amsfonts,graphicx,caption,subcaption}
\usepackage[colorlinks,linkcolor=red,citecolor=blue]{hyperref}

\newtheorem{prethm}{{\bf Theorem}}[section]
\newenvironment{thm}{\begin{prethm}{\hspace{-0.5
em}{\bf.}}}{\end{prethm}}
\newtheorem{prepro}{{\bf Theorem}}

\newtheorem{precor}[prethm]{{\bf Corollary}}
\newenvironment{cor}{\begin{precor}{\hspace{-0.5
em}{\bf.}}}{\end{precor}}
\newtheorem{preconj}[prethm]{{\bf Conjecture}}
\newenvironment{conj}{\begin{preconj}{\hspace{-0.5
em}{\bf.}}}{\end{preconj}}
\newtheorem{preremark}[prethm]{{\bf Remark}}
\newenvironment{remark}{\begin{preremark}\em{\hspace{-0.5
em}{\bf.}}}{\end{preremark}}
\newtheorem{prelem}[prethm]{{\bf Lemma}}
\newenvironment{lem}{\begin{prelem}{\hspace{-0.5
em}{\bf.}}}{\end{prelem}}
\newtheorem{preque}[prethm]{{\bf Question}}
\newenvironment{que}{\begin{preque}{\hspace{-0.5
em}{\bf.}}}{\end{preque}}

\newtheorem{preobserv}[prethm]{{\bf Observation}}
\newenvironment{observ}{\begin{preobserv}{\hspace{-0.5
em}{\bf.}}}{\end{preobserv}}

\newtheorem{predef}[prethm]{{\bf Definition}}

\newtheorem{preproposition}[prethm]{{\bf Proposition}}

\newtheorem{preproof}{{\bf Proof.}}
\newtheorem{preprooff}{{\bf Proof}}

\newenvironment{proof}[1]{\begin{preproof}{\rm
#1}\hfill{$\Box$}}{\end{preproof}}

\newtheorem{preproofs}{{\bf The second proof of }}

\newtheorem{preprooft}{{\bf Third proof of }}

\newtheorem{preproofF}{{\bf Proof of}}

\title{\bf\Large 
Equitable factorizations of highly edge-connected graphs: complete characterizations
}
\author{{\normalsize{\sc Morteza Hasanvand${}$} }\vspace{3mm}
\\{\footnotesize{${}$\it Department of Mathematical
 Sciences, Sharif
University of Technology, Tehran, Iran}}
{\footnotesize{}}\\{\footnotesize{ $\mathsf{morteza.hasanvand@alum.sharif.edu }$ }}}

\date{}

\begin{document}
\maketitle
\begin{abstract}{
In this paper, we show that every highly edge-connected graph $G$, under a necessary and sufficient degree condition, can be edge-decomposed into $k$ factors $G_1,\ldots, G_k$ such that for each vertex $v\in V(G_i)$ with $1\le i\le k$, $|d_{G_i}(v)-d_G(v)/k|<1$. 
This characterization covers graphs having at least $k-1$ vertices with degree not divisible by $k$. 
In addition, we investigate almost equitable factorizations in arbitrary edge-connected graphs.

Next, we establish a simpler criterion for the existence of factorizations $G_1,\ldots, G_k$ satisfying $d_{G_i}(v)\ge \lfloor d_G(v)/k\rfloor$ for all vertices $v$ (reps. $d_{G_i}(v)\le \lceil d_G(v)/k\rceil$).
As an application, we come up with a criterion to determine whether a highly edge-connected graph with $\delta(G)\ge \delta_1+\cdots+ \delta_m$ (resp. $\Delta(G)\le \Delta_1+\cdots+ \Delta_m$) can be edge-decomposed into factors $G_1,\ldots, G_m$ satisfying $\delta(G_i)\ge \delta_i$ (resp. $\Delta(G_i)\le \Delta_i$) for all $i$ with $1\le i \le m$, provided that $\delta_1+\cdots+ \delta_m$ is divisible by an odd number $p$ and $\delta_i\ge p-1\ge 2$ (resp. $\Delta_1+\cdots+ \Delta_m$ is divisible by $p$ and $\Delta_i\ge p-1\ge 2$). 
 
For graphs of even order, we replace an odd-edge-connectivity condition.
In particular, for the special case $m=2$, we refine the needed odd-edge-connectivity further by giving a sufficient odd-edge-connectivity condition for a graph $G$ to have a partial parity factor $F$ such that for each vertex $v$ with a given parity constraint, $| d_{F}(v)-\varepsilon d_G(v)|< 2$, and for all other vertices $v$, $| d_{F}(v)-\varepsilon d_G(v)|\le 1$, where $\varepsilon $ is a real number and $0< \varepsilon < 1$. 

Finally we introduce another application on the existence of almost even factorizations of odd-edge-connected graphs.
In particular, we prove that every odd-$k$-edge-connected graph $G$ can be edge-decomposed into $k$ factors $G_1,\ldots, G_k$ such that for each vertex $v\in V(G_i)$ with $1\le i\le k$, $|d_{G_i}(v)-d_G(v)/k|< 2$, $d_{G_1}(v)\stackrel{2}{\equiv} d_G(v)$, and $d_{G_i}(v)$ is even provided that $i\neq 1$.
} {\small
}
\end{abstract}
%
%
%
%
%
%
%
%
%
%
\section{Introduction}
In this article, graphs may have loops and multiple edges.
Let $G$ be a graph. The vertex set, the edge set, the minimum degree, and the maximum degree of $G$ are denoted by 
$V(G)$, $E(G)$, $\delta(G)$, and $\Delta(G)$, respectively.
We denote by $d_G(v)$ the degree of a vertex $v$ in the graph $G$, whether $G$ is directed or not.
Also the out-degree and in-degree of $v$ in a directed graph $G$ are denoted by $d_G^+(v)$ and $d_G^-(v)$.
An orientation of $G$ is said to be 
{\bf $p$-orientation}, if for each vertex $v$, $d_G^+(v)\stackrel{k}{\equiv}p(v) $,
 where $p:V(G)\rightarrow \mathbb{Z}_k$ is a mapping and $\mathbb{Z}_k$ is the cyclic group of order $k$.
Note that for any two integers $x$ and $y$, we write $x\stackrel{k}{\equiv}y$ when $x-y$ is divisible by $k$.
For an integer $x$, we denote by $[x]_k$ the unique integer in $\{0,\ldots, k-1\}$ congruent (modulo $k$) to $x$ (this is a little different from a similar definition in \cite{ModuloBounded}).
For any two disjoint vertex sets $A$ and $B$, 
we denote by $G[A,B]$ the bipartite factor of $G$ with the bipartition $(A,B)$.
Likewise, we denote by $d_G(A,B)$ the number of edges with one end in $A$ and the other one in $B$.
Also, we denote by $d_G(A)$ the number of edges of $G$ with exactly one end in $A$.
Note that $d_G(\{v\})$ denotes the number of non-loop edges incident with $v$,
while $d_G(v)$ denotes the degree of $v$.
We denote by $G[A]$ the induced subgraph of $G$ with the vertex set $A$ containing
precisely those edges of $G$ whose ends lie in $A$.
Th graph obtained from $G$ by contracting all vertices of $A$ into a single vertex is denoted by $G/A$
 (note that this operation does not create new loops).
The bipartite index {\bf $bi(G)$} of a graph $G$ is the smallest number of all $E(G) \setminus E(H)$ taken over all bipartite
factors $H$.
For a graph $G$, {\bf splitting a vertex $v$} means to replace $v$ by two (or more) new vertices $v_1$ and $v_2$ such every non-loop edges incident with $v$ becomes incident with either $v_1$ or $v_2$, and every loop incident with $v$ becomes incident with both of them.
Clearly, this operation is not necessarily unique whereas it keeps the other ends of non-loop edges unchanged.
Let $g$ and $f$ be two integer-valued functions on $V(G)$.
A {\bf parity $(g,f)$-factor} of $G$ refers to a spanning subgraph $F$ such that for each vertex $v$, $g(v)\le d_F(v)\le f(v)$, and
also $d_F(v)$, $g(v)$, and $f(v)$ have the same parity.
When $g=f$, this graph is called an {\bf $f$-factor} as well.
For a vertex set $V_0$, we say that a factor $F$ is {\bf $V_0$-partially parity $(g,f)$-factor} if for all $v\in V(G)$, $g(v)\le d_F(v)\le f(v)$, and for all $v\in V_0$, $d_F(v)$, $g(v)$, and $f(v)$ have the same parity.
An {\bf $f$-parity factor} is a spanning subgraph $F$ such that for each vertex $v$, 
$d_F(v)$ and $f(v)$ have the same parity.
Note that a parity $(g,f)$-factor is also an $f$-parity factor.
We say also that a factor $F$ is a {\bf $V_0$-partially $f$-parity factor}, if for all $v\in V_0$, 
$d_F(v)$ and $f(v)$ have the same parity.
An {\bf even graph} refers to a graph with all vertex-degrees even.
A graph $G$ is called 
{\bf $m$-tree-connected}, if it contains $m$ edge-disjoint spanning trees. 
Likewise, a graph $G$ is said to be
{\bf $(m,l)$-partition-connected}, if it can be edge-decomposed into an $m$-tree-connected factor and a factor $F$ having an orientation such that for each vertex $v$, $d^+_F(v)\ge l(v)$, where $l$ is a nonnegative integer-valued function on $V(G)$ and $d^+_F(v)$ denotes the out-degree of $v$ in $F$.
Note that an $(m_1+m_2, l_1+l_2)$-partition-connected graph can be edge-decomposed into an $(m_1, l_1)$-partition-connected factor and an $(m_2, l_2)$-partition-connected factor.
A graph $G$ is termed {\bf essentially $\lambda$-edge-connected},
 if all edges of any edge cut of size strictly less than $\lambda$ are incident with a common vertex.
A graph $G$ is called {\bf odd-$\lambda$-edge-connected}, if $d_G(A)\ge \lambda$
for every vertex set $A$ with $d_G(A)$ odd.
Note that $2\lambda$-edge-connected graphs are
odd-$(2\lambda+1)$-edge-connected.
For a vertex set $V_0$, we say that $G$ is {\bf $V_0$-partially $\lambda$-edge-connected}, if $d_G(A)\ge \lambda$, 
for all nonempty subsets $A$ of $V_0$ with $A\neq V(G)$.
In other words, $G/(V(G)\setminus V_0)$ is $\lambda$-edge-connected.
Similarly, we say that $G$ is {\bf $V_0$-partially $m$-tree-connected}, if $G/(V(G)\setminus V_0)$ is $m$-tree-connected.
Likewise, we say that $G$ is {\bf $V_0$-partially odd-$\lambda$-edge-connected}, if $d_G(A)\ge \lambda$, 
for all nonempty subsets $A$ of $V_0$ with $d_G(A)$ odd.
For a vertex set $Q$, we say that $G$ is {\bf $V_0$-partially odd-$(\lambda, Q)$-edge-connected}, if $d_G(A)\ge \lambda$, 
for all subsets $A$ of $V_0$ with $|A\cap Q|$ odd.
If $V_0=V(G)$, we say that $G$ is {\bf odd-$(\lambda, Q)$-edge-connected}.
A factorization $G_1,\ldots, G_k$ of $G$ is called 
a {\bf $k$-equitable factorization}, if for each vertex $v$, $|d_{G_i}(v)-d_G(v)/k|<1$, where $1\le i\le k$.
The {\bf $k$-core} of a graph $G$ refers to the induced subgraph of $G$ with the vertex set $\{v\in V(G):d_G(v)\stackrel{k}{\equiv}0\}$.
Throughout this article, all variables $m$ are nonnegative integers and all variables $k$ are positive integers.

In 1971 de Werra observed that bipartite graphs admit equitable factorizations.
\begin{thm}{\rm (\cite{Werra})}\label{thm:Werra}
{Every bipartite graph $G$ can be edge-decomposed into $k$ factors $G_1,\ldots, G_k$ such that for each $v\in V(G_i)$, $1\le i\le k$,
 $ | d_{G_i}(v)-d_G(v)/k|< 1$.
}\end{thm}

In 1994 Hilton and de Werra~\cite{Hilton-Werra} showed that simple graphs whose $k$-cores have no edge must admit $k$-equitable factorizations. 
In 2011 Zhang and Liu~\cite{Zhang-Liu} extended their result to a more general class of simple graphs and concluded the following theorem which was originally conjectured by Hilton (2008)~\cite{HILTON2008645}.
\begin{thm}{\rm (\cite{Zhang-Liu})}
{Let $G$ be a simple graph. If the $k$-core of $G$ forms a forest, then $G$ can be edge-decomposed into $k$ factors $G_1,\ldots, G_k$ such that for each $v\in V(G_i)$, $1\le i\le k$,
 $ | d_{G_i}(v)-d_G(v)/k|< 1$.
}\end{thm}

Recently, the present author (2022) gave a sufficient degree condition for the existence of equitable factorizations in edge-connected graphs by proving the following result.
This is a generalization of a result due to Thomassen (2019)~\cite{Thomassen-2020} who proved that every $(3k-3)$-edge-connected $kr$-regular graph of even order can be edge-decomposed into $r$-regular factors. 
\begin{thm}{\rm (\cite{EquitableFactorizations-2022})}\label{intro:thm:factorization:Z}
{Let $G$ be a graph satisfying $|E(G)|\stackrel{k}{\equiv}\sum_{v\in Z}d_G(v)$ for a vertex subset $Z\subseteq V(G)$.
If $G$ is $(3k-3)$-edge-connected, then it can be edge-decomposed into $k$ factors $G_1,\ldots, G_k$ such that for each $v\in V(G_i)$, $1\le i\le k$, $|d_{G_i}(v)-d_G(v)/k|< 1$.
}\end{thm}

In this paper, we characterize degree sequences modulo $k$ for the existence of equitable factorizations in highly edge-connected graphs during several steps. We begin with a simpler condition by proving the following theorem in Section~\ref{sec:Factorizations}.
 This result contains graphs having at least $k-1$ vertices whose degrees are not divisible by $k$, and
 together with Theorem~\ref{intro:thm:sum-min-x-k-x}, it is surprisingly enough to characterize degree sequences without equitable factorizations for small numbers $k\le 7$, see Table 1.
\begin{thm}\label{intro:thm:sum-min-x-k-x}
{Let $G$ be a graph satisfying $\sum_{v\in V(G)}\min \{[d_G(v)]_k, [k-d_G(v)]_k\}\ge k-1$. If $G$ is $10(k-1)$-edge-connected, then it can be edge-decomposed into $k$ factors $G_1,\ldots, G_k$ such that for each $v\in V(G_i)$, $1\le i\le k$, $ | d_{G_i}(v)-d_G(v)/k|< 1$. 
}\end{thm}

 In Section~\ref{sec:lower-upper-bound-factorizations}, we establish a simpler criterion for the existence of factorizations $G_1,\ldots, G_k$ satisfying $d_{G_i}(v)\ge \lfloor d_G(v)/k\rfloor$ for all vertices $v$ (reps. $d_{G_i}(v)\le \lceil d_G(v)/k\rceil$). In particular, we prove the following simpler version of Theorem~\ref{intro:thm:sum-min-x-k-x}.

\begin{thm}
{Let $G$ be a graph satisfying $\sum_{v\in V(G)}[d_G(v)]_k\ge k-1$ {\rm (}resp. $\sum_{v\in V(G)}[k-d_G(v)]_k\ge k-1${\rm)}. If $G$ is $3(k-1)$-edge-connected, then it can be edge-decomposed into $k$ factors $G_1,\ldots, G_k$ such that for each $v\in V(G_i)$, $1\le i\le k$,
$d_{G_i}(v)\ge \lfloor d_G(v)/k\rfloor$ {\rm (}reps. $d_{G_i}(v)\le \lceil d_G(v)/k\rceil${\rm)}.
}\end{thm}

Moreover, we prove the following two results for the existence of factorizations with bounded minimum (resp. maximum) degrees in graphs with larger minimum (resp. maximum) degree. These theorems generalize two recent results due to Thomassen (2019)~\cite[Theorems 4 and 5]{Thomassen-2020} who proved them for graphs of even order and odd vertex degrees, respectively.
For graphs of even order, we replace an odd-edge-connectivity condition.
In Section~\ref{sec:epsilon-parity-factor}, we refine further the special case $m=2$ to improve a recent result due to Qin and Wu~(2022) \cite{Qin-Wu-2022}.
\begin{thm}\label{intro:thm:delta}
{Let $k$ be a positive integer and let $\delta_1,\ldots,\delta_m$ be positive integers with $\delta_i\ge p-1\ge 2$ and $kp=\delta_1+\cdots+ \delta_m$, where $p$ is an odd integer. Let $G$ be a $(3k-3)$-edge-connected graph satisfying $\delta(G)\ge \delta_1+\cdots+ \delta_m$.
Then $G$ can be edge-decomposed into factors $G_1,\ldots, G_m$ satisfying $\delta(G_i)\ge \delta_i$ for all $i$ with $1\le i \le m$, 
if and only if one of the following conditions holds:
\begin{enumerate}{
\item [$\bullet$] 
$|V(G)|$ is even.
\item [$\bullet$] 
$\sum_{v\in V(G)}(d_G(v)-(\delta_1+\cdots+ \delta_m))\ge |\{i:\delta_i \text{ is odd}\}|-1$.
}\end{enumerate}
}\end{thm}

\begin{thm}\label{intro:thm:Delta}
{Let $k$ be a positive integer and let $\Delta_1,\ldots,\Delta_m$ be positive integers with $\Delta_i\ge p-1\ge 2$ and $kp=\Delta_1+\cdots+ \Delta_m$, where $p$ is an odd integer. Let $G$ be a $(3k-3)$-edge-connected graph satisfying $\Delta(G)\le \Delta_1+\cdots+ \Delta_m$.
Then $G$ can be edge-decomposed into factors $G_1,\ldots, G_m$ satisfying $\Delta(G_i)\le \Delta_i$ for all $i$ with $1\le i \le m$, 
if and only if one of the following conditions holds:
\begin{enumerate}{
\item [$\bullet$] 
$|V(G)|$ is even.
\item [$\bullet$] 
$\sum_{v\in V(G)}(\Delta_1+\cdots+ \Delta_m-d_G(v))\ge |\{i:\Delta_i \text{ is odd}\}|-1$.
}\end{enumerate}
}\end{thm}

In 1982 Hilton constructed the following parity version of Theorem~\ref{thm:Werra} 
on the existence of even factorizations with bounded degrees of even graphs.
This result is recently generalized in \cite[Theorems 4.2 and 4.4]{EquitableFactorizations-2022} in two ways
 for highly odd-edge-connected graphs.
\begin{thm}{\rm (\cite{Hilton1982})}\label{thm:Hilton}
{Every even graph $G$ can be edge-decomposed into $k$ even factors $G_1,\ldots, G_k$ such that for each $v\in V(G_i)$ with $1\le i\le k$, 
 $|d_{G_i}(v)-d_{G}(v)/k|<2.$
}\end{thm}

Moreover, we conjectured that Theorem~\ref{thm:Hilton} can be developed to the following version on even graphs. 
In particular, we proved the following conjecture for the special case $k=2$ and confirmed a weaker version of it by replacing the upper bound of $6$ on $|d_{G_i}(v)-\varepsilon_i d_G(v)|$.

\begin{conj}{\rm (\cite{EquitableFactorizations-2022})}
{Let $\varepsilon_1,\ldots,\varepsilon_k$ be $k$ nonnegative real numbers with $\varepsilon_1+\cdots +\varepsilon_k=1$.
If $G$ is an even graph, then it can be edge-decomposed into $k$ even factors $G_1,\ldots, G_k$ satisfying for each $v\in V(G)$, $|d_{G_i}(v)-\varepsilon_i d_G(v)|<2$, where $1\le i\le k$.
}\end{conj}

In Section~\ref{thm:almost even factorizations}, we turn our attention to the existence of almost even factorizations in arbitrary graphs and highly odd-edge-connected graphs. In particular, we prove the following conjecture for the 
case that almost all of $\varepsilon_i$ are the same number, except possibly $\varepsilon_1$, and confirm a weaker version of it by replacing the upper bound of $6$ on $|d_{G_i}(v)-\varepsilon_i d_G(v)|$.
\begin{conj}\label{intro:conj:new}
{Let $\varepsilon_1,\ldots,\varepsilon_k$ be $k$ nonnegative real numbers with $\varepsilon_1+\cdots +\varepsilon_k=1$.
If $G$ is a graph, then it can be edge-decomposed into $k$ factors $G_1,\ldots, G_k$ satisfying the following properties:
\begin{enumerate}{
\item [$\bullet$] For each $v\in V(G)$, $|d_{G_i}(v)-\varepsilon_i d_G(v)|<2$, where $1\le i\le k$.
\item [$\bullet$] For each $v\in V(G)$, there exists at most one index $j_v$ with $d_{G_{j_v}}(v)$ odd (in particular, there is not such index when $d_G(v)$ is even).
}\end{enumerate}
Furthermore, if $G$ is odd-$\lceil 1/\varepsilon_1 \rceil$-edge-connected, then $j_v=1$ for all odd-degree vertices $v$.
}\end{conj}

%
%
%
%
%
%
%
%
%

\section{Preliminary results: highly edge-connected splitting graphs}
Let $G$ be a graph with a given orientation $D$. We denote by $G_D$ be the graph obtained from $G$ and $D$ by splitting every vertex into two vertices $v^+$ and $v^-$ such that
 the incident edges of $v^+$ are directed away from $v$ in $G$ and 
the incident edges of $v^-$ are directed toward $v$ in $G$.
In this construction, every loop in $G$ incident with $v$ is transformed into an edge between $v^+$ and $v^-$.
Clearly, this graph is bipartite. In this section, we provides a sufficient edge-connectivity condition to guarantee the existence of an orientation $D$ with restricted out-degrees such that the bipartite graph $G_D$ is close to be highly tree-connected.
 For this purpose, we need to recall the following theorem from~\cite{ModuloBounded, Lovasz-Thomassen-Wu-Zhang-2013}.
\begin{thm}{\rm (\cite{ModuloBounded, Lovasz-Thomassen-Wu-Zhang-2013})}\label{thm:modulo:2k-2:3k-3}
{Let $G$ be a graph, let $k$ be a positive integer, and let $p:V(G)\rightarrow \mathbb{Z}_k$ be a mapping with $|E(G)| \stackrel{k}{\equiv} \sum_{v\in V(G)}p(v)$. If $(3k-3)$-edge-connected or $G$ is $(2k-2)$-tree-connected, then it admits a $p$-orientation modulo $k$.
}\end{thm}
We begin with the following conclusion that allows us to make the proof easier. 
\begin{cor}\label{cor:F_1-F_2}
{Let $G$ be a graph, let $k$ be a positive integer, and let $p:V(G)\rightarrow \mathbb{Z}_k$ a mapping with
$|E(G)| \stackrel{k}{\equiv} \sum_{v\in V(G)}p(v)$. 
Let $F_1$ and $F_2$ be two edge-disjoint factors of $G$ and let $U_1,U_2$ be a bipartition of $V(G)$. 
If $G\setminus E(F_1\cup F_2)$ is $(2k-2)$-tree-connected, 
then every vertex $v$ can be split into two vertices $v_1$ and $v_2$
such that the resulting graph $H$ satisfies the following properties:
\begin{enumerate}{
\item [$\bullet$]

For all $v\in V(G)$, $d_H({v_1})\stackrel{k}{\equiv}p(v)$.

\item [$\bullet$]

$H$ is a bipartite graph with the bipartition $(V_1,V_2)$ for which $V_i=\{v_i:v\in V(G)\}$.

\item [$\bullet$]

$H[S_i]$ contains the transformed edges of $F_i[U_1, U_2]$, where 
$S_i = \{v_1:v\in U_i\}\cup \{v_2:v\not\in U_i\}$.

\item [$\bullet$] 
The edges of $H[S_1, S_2]$ are the same transformed edges of $G[U_1]\cup G[U_2]$.

\item [$\bullet$]
For every edge $xy$ of $F_1\cup F_2$ with both ends in $U_i$, we can arbitrarily have $x_1y_2 \in E(H)$ or $x_2y_1 \in E(H)$.
}\end{enumerate}
}\end{cor}
\begin{proof}
{First we orient the edges of $F_i[U_1,U_2]$ from $U_i$ to $V(G)\setminus U_i$. 
Next we arbitrarily orient the edges of $F_i[U_1]\cup F_i[U_2]$.
For each $v\in U_i$, we define $p'(v)=p(v)-d^+_{F_i}(v)$ (modulo $k$).
Let $G'=G\setminus E(F_1\cup F_2)$.
It is easy to see that
$\sum_{v\in V(G)} p'(v) \stackrel{k}{\equiv} 
\sum_{v\in V_1} (p(v)-d_{F_1}(v))+\sum_{v\in V_2} (p(v)-d_{F_2}(v))
\stackrel{k}{\equiv}
|E(G)|-|E(F_1\cup F_2)|=|E(G')|$.
Thus by Theorem~\ref{thm:modulo:2k-2:3k-3}, the graph $G'$ has an orientation 
such that for each vertex $v$, $d_{G'}^+(v) \stackrel{k}{\equiv} p'(v)$. 
Consider the orientation of $G$ obtained from the orientations of $F_1$, $F_2$, and $G'$. 
Obviously, for all vertices $v\in U_i$, $d^+_G(v)=d^+_{G'}(v)+d^+_{F_i}(v) \stackrel{k}{\equiv}p(v)$.
First we split every vertex $v$ into two vertices $v_1$ and $v_2$ such that
 the incident edges of $v_1$ are directed away from $v$ in $G$ and 
the incident edges of $v_2$ are directed toward $v$ in $G$.
In this construction, every loop in $G$ incident with $v$ is transformed into an edge between $v_1$ and $v_2$.
 Call the resulting (loopless) bipartite graph $H$. Obviously, for all vertices $v$, 
 $d_H(v_1) =d^+_G(v) \stackrel{k}{\equiv}p(v)$.
 Since all edges of $F_i[U_1,U_2]$ are directed from $U_i$ to $V(G)\setminus U_i$,
the graph $H[S_i]$ contains the edges of $F_i[U_1,U_2]$.
For every edge $xy$ directed from $x$ to $y$ in $G[U_i]$ we must have $x_1y_2\in E(H[S_1, S_2])$.
Since we could arbitrarily specify the orientations of those edges in $F_i[U_1]\cup F_i[U_2]$, the last item is consequently derived.
Hence the result holds.
}\end{proof}
\begin{figure}[h]
 \centering
 \includegraphics[scale =1.3]{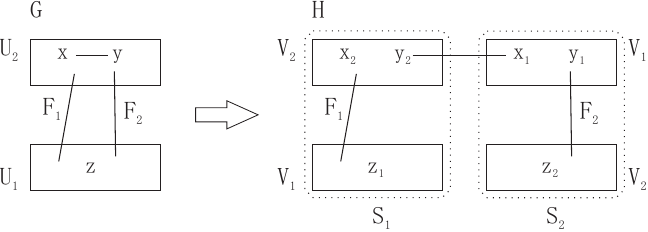}
 \caption{Splitting graph $G$ to make a modulo bipartite graph $H$.}
 \label{L-z0}
\end{figure}
\begin{lem}{\rm (\cite{B})}\label{lem:bipartite:partition-connected}
{Every $(2m, 2l)$-partition-connected graph $G$ has a bipartite $(m, l)$-partition-connected factor, where $l$ is a nonnegative integer-valued function on $V(G)$.
}\end{lem}
Before stating the main result, we are going to present a shorter proof but for graphs with higher partition-connectivity.
\begin{thm}\label{thm:split:tree-connected}
{Let $G$ be a graph, let $k$ be a positive integer, let $m_i$ be a nonnegative integer, and let $l_i$ be a nonnegative integer-valued function on $V(G)$, where $i=1,2$.
If $G$ is $(2k-2+2m_1+2m_2, 2l_1+2l_2)$-partition-connected, then
there exists a bipartition $(U_1, U_2)$ of $V(G)$ such that if $p:V(G)\rightarrow \mathbb{Z}_k$ is a mapping with
$|E(G)| \stackrel{k}{\equiv} \sum_{v\in V(G)}p(v)$, then every vertex $v$ can be split into two vertices $v_1$ and $v_2$
such that the resulting graph $H$ satisfies the following properties:
\begin{enumerate}{
\item [$\bullet$]

For all $v\in V(G)$, $d_H({v_1})\stackrel{k}{\equiv}p(v)$.

\item [$\bullet$]

$H$ is a bipartite graph with the bipartition $(V_1,V_2)$ for which $V_i=\{v_i:v\in V(G)\}$.

\item [$\bullet$] 
$H[S_i]$ is $(m_i, l_i)$-partition-connected, where $S_i = \{v_1:v\in U_i\}\cup \{v_2:v\not\in U_i\}$.

\item [$\bullet$]
The edges of $H[S_1, S_2]$ are the same transformed edges of $G[U_1]\cup G[U_2]$.
}\end{enumerate}
}\end{thm}
\begin{proof}
{Decompose $G$ into a $(2k-2)$-tree-connected factor $G_0$ and
a $(2m_1+2m_2, 2l_1+2l_2)$-partition-connected factors $G'$.
By Lemma~\ref{lem:bipartite:partition-connected}, there exists a bipartition $(U_1, U_2)$ of $V(G)$ 
such that the bipartite graph $G'[U_1, U_2]$ is $(m_1+m_2, l_1+l_2)$-partition-connected. 
We decompose $G'[U_1, U_2]$ into two factors $F_1$ and $F_2$ such that $F_i$ is $(m_i, l_i)$-partition-connected. 
Thus by Corollary~\ref{cor:F_1-F_2}, every vertex $v$ can be split into two vertices $v_1$ and $v_2$
such that the resulting graph $H$ is a bipartite graph with the bipartition $(V_1,V_2)$ for which $V_i=\{v_i:v\in V(G)\}$, and
(i1) for all $v\in V_1$, $d_H({v_1})\stackrel{k}{\equiv}p(v)$,
(i2) $H[S_i]$ contains the transformed edges of $F_i[U_1, U_2]=F_i$, where $S_i = \{v_1:v\in U_i\}\cup \{v_2:v\not\in U_i\}$,
and (i3) the edges of $H[S_1, S_2]$ are the same transformed edges of $G[U_1]\cup G[U_2]$.
By item (i2), the graph $H[S_i]$ must be $(m_i, l_i)$-partition-connected. This completes the proof.
}\end{proof}
In the following, we are going to improve the partition-connectivity needed in Theorem~\ref{thm:split:tree-connected} 
with a longer technical proof.
Our proof is based on the following simpler version of Lemma~\ref{lem:bipartite:partition-connected}.
\begin{lem}{\rm (\cite{Modulo-Factors-Bounded})}\label{lem:bipartite}
{Every $2m$-tree-connected graph has a bipartite $m$-tree-connected factor.
}\end{lem}
%
\begin{lem}\label{lem:bipartite-index:decomposition}
{Let $G$ be a graph and let $l$ be a nonnegative integer-valued function on $V(G)$.
If $G$ is $(m_0+ 2m, l)$-partition-connected, then $G$ can be edge-decomposed into three factors $G_0$, $G_1$, $M$
such that $G_0$ is $(m_0, l)$-partition-connected,
$G_1$ is an $m$-tree-connected bipartite graph with a bipartition $(U_1,U_2)$, and $M$ is a factor of $G[U_1]\cup G[U_2]$ consists of at least $\min \{m, e_{G}(U_1)+e_{G}(U_2)\}$ edges.
}\end{lem}
\begin{proof}
{First, we decompose $G$ into three factors $H_0$, $L$, and $H$ such that $H_0$ is
 $m_0$-tree-connected, $L$ is $(0,l)$-partition-connected, and $H$ is $2m$-tree-connected. 
By Lemma~\ref{lem:bipartite}, there exists a bipartition $(U_1, U_2)$ of $V(G)$ such that
 the graph $H[U_1, U_2]$ is $m$-tree-connected.
If $e_{H}(U_1)+e_{H}(U_2)\ge m$, 
then we set $G_0=H_0\cup L$, $G_1=H[U_1,U_2]$, and $M=H[U_1]\cup H[U_2]$.
So, suppose that $e_{H}(U_1)+e_{H}(U_2)< m$. 
Let $T_1,\ldots, T_{2m}$ be a factorization of $H$ such that each of them is connected.
We may assume that each of $T_{m+1}\ldots, T_{2m}$ does not contain any edge of $H[U_1]\cup H[U_2]$ 
and let $G_1$ be the union of them. In other words, $G_1$ is an $m$-tree-connected bipartite graph with a bipartition $(U_1,U_2)$.
On the other hand, $G\setminus E(G_1)$ must be $(m_0+m, l)$-partition-connected, because it contains both of $H_0\cup L$ and
$T_1\cup \cdots \cup T_m$.
Therefore, we can decompose $G\setminus E(G_1)$ into factors $\mathcal{L}$ and $\mathcal{T}_1,\ldots, \mathcal{T}_{m+m_0}$
 such that $\mathcal{L}$ is $(0,l)$-partition-connected and $\mathcal{T}_i$ is connected.
In this case, we may assume that the graph $\mathcal{L}$ contains minimum number of edges of $G[U_1]\cup G[U_2]$.
We consider an orientation for $\mathcal{L}$ such that for each vertex $v$, $d^+_{\mathcal{L}}(v)\ge l(v)$.
We may also assume that $e_{\mathcal{T}_i}(U_1)+e_{\mathcal{T}_i}(U_2) \ge e_{\mathcal{T}_{i+1}}(U_1)+e_{\mathcal{T}_{i+1}}(U_2)$. 
If $e_{\mathcal{T}_m}(U_1)+e_{\mathcal{T}_m}(U_2)\neq 0$, then 
we set $G_0=\mathcal{T}_{m+1}\cup \cdots \cup \mathcal{T}_{m+m_0}\cup \mathcal{L}$, $G_1=H[U_1,U_2]$, and 
$M=\mathcal{T}_{1}\cup \cdots \cup \mathcal{T}_{m}$. Note that $e_{M}(U_1)+e_{M}(U_2)\ge m$.
So, suppose $e_{\mathcal{T}_m}(U_1)+e_{\mathcal{T}_m}(U_2)= 0$.
We claim that $e_{\mathcal{L}}(U_1)+e_{\mathcal{L}}(U_2)= 0$. 
Otherwise, if there is an edge $xy\in E(\mathcal{L}[U_1]\cup \mathcal{L}[U_2])$ which is directed from $x$ to $y$, 
then there is an edge $xz\in \mathcal{T}_m$ such that $\mathcal{T}_m-xz+xy$ is still connected. 
On the other hand, it is easy to see that $\mathcal{L}-xy+xz$ is still $(0, l)$-partition-connected.
Since $xz\in \mathcal{T}_m[U_1, U_2]$, $\mathcal{L}-xy+xz$ contains smaller number of edges of 
$G[U_1]\cup G[U_2]$ which is a contradiction. Hence the claim holds. In this case, we again
 set $G_0=\mathcal{T}_{m+1}\cup \cdots \cup \mathcal{T}_{m+m_0}\cup \mathcal{L}$, $G_1=H[U_1,U_2]$, and 
$M=\mathcal{T}_{1}\cup \cdots \cup \mathcal{T}_{m}$.
Note that $e_{M}(U_1)+e_{M}(U_2)= e_{G}(U_1)+e_{G}(U_2)\ge \min\{bi(G), m\}$.
This completes the proof.
}\end{proof}
Now, we have all ingredient to prove the main result of this section.
\begin{thm}\label{thm:split:partition-connected:3m}
{Let $G$ be a graph with $z\in V(G)$, let $k$ be a positive integer, and let $p:V(G)\rightarrow \mathbb{Z}_k$ a mapping with
$|E(G)| \stackrel{k}{\equiv} \sum_{v\in V(G)}p(v)$. Let $l_i$ be a nonnegative integer-valued function on $V(G)$, where $i=1,2$.
If $G$ is $(2k-2+3m, l_1+l_2)$-partition-connected, then every vertex $v$ can be split into two vertices $v_1$ and $v_2$
such that the resulting graph $H$ satisfies the following properties:
\begin{enumerate}{
\item [$\bullet$]
For all $v\in V(G)$, $d_H({v_1})\stackrel{k}{\equiv}p(v)$.
\item [$\bullet$]
$H$ is a bipartite graph with the bipartition $(V_1,V_2)$ for which $V_i=\{v_i:v\in V(G)\}$.

\item [$\bullet$]
$H$ is $(m, l)$-partition-connected after adding at most $\max\{0, m-bi(G)\}$ new copies of the edge $z_1z_2$, where $l(v_i)=l_i(v)$ for all $v\in V(G)$.
}\end{enumerate}
}\end{thm}
\begin{proof}
{Define $\ell_i(v)=l_1(v)$ for all $v\in U_i$, and define $\ell_i(v)=l_2(v)$ for all $v\in V(G)\setminus U_i$.
Note that for all vertices $v$, $\ell_1(v)+\ell_2(v)=l_1(v)+l_2(v)$.
Thus by applying Lemma~\ref{lem:bipartite-index:decomposition} with setting $m_0=m$, one can conclude that $G$ can be edge-decomposed into 
factors $G_0$, $G_1$, $G_2$, $M$, $L_1$, and $L_2$ such that 
$G_0$ is $(2k-2)$-tree-connected, 
$G_1$ is $m$-tree-connected, 
$G_2$ is an $m$-tree-connected bipartite graph with a bipartition $(U_1,U_2)$, 
$M$ is a factor of $G[U_1]\cup G[U_2]$ consists of $\min \{m,bi(G)\}$ edges, and
$L_i$ is $(0, \ell_i)$-partition-connected. Note that $e_{G}(U_1)+e_{G}(U_2)\ge bi(G)$.
Let $F_1= G_1\cup M\cup L_1$ and $F_2= G_2 \cup L_2$.

Since $G_0=G\setminus E(F_1\cup F_2)$ is $(2k-2)$-tree-connected, by Corollary~\ref{cor:F_1-F_2}, 
every vertex $v$ can be split into two vertices $v_1$ and $v_2$
such that the resulting graph $H$ satisfies the following properties:
(i1)
for all $v\in V(G)$, $d_H({v_1})\stackrel{k}{\equiv}p(v)$, 
(i2)
$H$ is a bipartite graph with the bipartition $(V_1,V_2)$ for which $V_i=\{v_i:v\in V(G)\}$, 
(i3)
$H[S_i]$ contains the transformed edges of $F_i[U_1, U_2]$, where $S_i = \{v_1:v\in U_i\}\cup \{v_2:v\not\in U_i\}$, and
(i4)
the edges of $H[S_1, S_2]$ are the same transformed edges of $G[U_1]\cup G[U_2]$.
In particular, for every edge $xy$ of $F_1\cup F_2$ with both ends in $U_1$ or $U_2$, we can arbitrarily have $x_1y_2 \in E(H)$ or $x_2y_1 \in E(H)$. 

We consider an orientation for $L_i$ such that for each vertex $v$, $d^+_{L_i}(v)\ge \ell_i(v)$.
For every edge $xy$ of $L_i[U_1]\cup L_i[U_2]$ which is directed from $x$ to $y$,
we impose the condition $x_{a}y_{b}\in E(H)$ if $x_{a}\in S_i$ (Note that we have either $x_1\in S_i$ or $x_2\in S_i$).
This can imply that $\mathcal{L}_i$ admits an orientation such that for each $v_t\in S_i$, 
$d^+_{\mathcal{L}_i}(v_t)\ge d^+_{L_i}(v)\ge \ell_i(v)=l_t(v)=l(v_t)$, where $\mathcal{L}_i$ is the factor of $H$ obtained from $L_i$
(the orientation of the every edge of $\mathcal{L}_i[S_i]$ must be induced by the orientation of the corresponding edge in $L_i[U_1,U_2]$).
Consequently, $\mathcal{L}_1\cup \mathcal{L}_2$ is a $(0, l)$-partition-connected factor of $H$.

Let $E(M)=\{e_1,\ldots, e_n\}$. 
Decompose $G_1$ into connected factors $T_1,\ldots, T_m$. We may assume that $T_i$ is a spanning tree. 
We define $e_i$ to be a new copy of the loop $zz$ when $n < i\le m$.
Now, assume that $e_i=rr'$ and $r_a\in S_1$ and we imposed the condition $r_a r'_b\in E(H)$ when $1\le i\le n$.
Let $xy$ be an arbitrary edge in $T_i[U_1]\cup T_i[U_2]$ and assume $x$ lies in the unique path connecting $r$ and $y$ in $T_i$.
For this edge, we impose the condition $x_{a}y_b\in E(H)$ if $y_b\in S_1$.
 According to this construction, $(\mathcal{T}_i+r_{a}r'_b)/S_2$ is connected when $1\le i\le n$, and 
$(\mathcal{T}_i+z_1z_2)/S_2$ is connected when $n< i\le m$, where $\mathcal{T}_i$ is the factor of $H$ obtained from $T_i$.
Let $\mathcal{G}_1$, $\mathcal{G}_2$, and $\mathcal{M}$ be the three factors of $H$ obtained from
 $G_1$, $G_2$, and $M$, respectively.
 Since $\mathcal{G}_2[S_2]$ is $m$-tree-connected and the contracted graph $(\mathcal{G}_1 \cup \mathcal{M})/S_2$ is $m$-tree-connected after adding $m-n$ new copies of $z_1z_2$, the graph $\mathcal{G}_1\cup \mathcal{M} \cup \mathcal{G}_2$ is $m$-tree-connected after adding $m-n$ new copies of $z_1z_2$.
Consequently, $H$ itself must be $(m, l)$-partition-connected after adding $m-n$ new copies of the edge $z_1z_2$,
because $\mathcal{L}_1\cup \mathcal{L}_2$ is $(0, l)$-partition-connected and 
$H=(\mathcal{G}_1\cup \mathcal{M} \cup \mathcal{G}_2)\cup \mathcal{L}_1\cup \mathcal{L}_2$.
This completes the proof.
}\end{proof}
The following corollary plays an important role in this paper.
\begin{cor}\label{cor:split:tree-connected:3m}
{Let $G$ be a graph with $z\in V(G)$, let $k$ be a positive integer, and let $p:V(G)\rightarrow \mathbb{Z}_k$ a mapping with
$|E(G)| \stackrel{k}{\equiv} \sum_{v\in V(G)}p(v)$.
If $G$ is $(2k-2+3m)$-tree-connected, then every vertex $v$ can be split into two vertices $v_1$ and $v_2$
such that the resulting graph $H$ satisfies the following properties:
\begin{enumerate}{
\item [$\bullet$]
For all $v\in V(G)$, $d_H({v_1})\stackrel{k}{\equiv}p(v)$.
\item [$\bullet$]
$H$ is a bipartite graph with the bipartition $(V_1,V_2)$ for which $V_i=\{v_i:v\in V(G)\}$.

\item [$\bullet$]
$d_{H}(X)\ge m$ for all proper subsets $X$ of $V(H)$ satisfying $ \{z_1, z_2\}\subseteq X$.
}\end{enumerate}
}\end{cor}
\begin{proof}
{It is enough to apply Theorem~\ref{thm:split:partition-connected:3m} with $l_i=0$ and use the fact that
 $d_{H'}(X)=d_{H}(X)$ for all subsets $X$ of $V(H)$ satisfying $ \{z_1, z_2\}\subseteq X$, where $H'$ is a graph obtained from $H$ by adding some copies of the edge $z_1z_2$.
}\end{proof}
%
%
%
%
%
%
%
%
%
\section{Equitable factorizations in highly edge-connected graphs}
\label{sec:Factorizations}
\subsection{Degree-ratio factors: bipartite graphs}
As we already observed bipartite graphs admit equitable factorizations.
In this subsection, we are going to investigate the existence of a special type of such factorizations by giving a sufficient edge-connectivity condition. Our proof is based on the following result due to Folkman and Fulkerson (1970).
\begin{thm}{\rm (\cite{Folkman-Fulkerson-1970})}\label{thm:Folkman:Fulkerson--1970}
{Let $G$ be a bipartite graph with bipartition $(X,Y)$ and let $g$ and $f$ be two integer-valued functions on $V(G)$ with $g\le f$. 
Then $G$ has a $(g,f)$-factor if and only if 
for all subsets $A$ and $B$ of $V(G)$ satisfying either $A\subseteq X$ and $B\subseteq Y$ or $A\subseteq Y$ and $B\subseteq X$,
$$d_G(A,B)\le \sum_{v\in A} f(v)+\sum_{v\in B} (d_{G}(v)-g(v)).$$
}\end{thm}
Before stating the main result, let us write the following counterpart of Theorem~\ref{thm:Folkman:Fulkerson--1970}.
\begin{cor}\label{cor:Folkman:Fulkerson--1970}
{Let $G$ be a bipartite graph with bipartition $(X,Y)$ and let $g$ and $f$ be two integer-valued functions on $V(G)$ with $g\le f$. 
Let $m$ and $k$ be two positive integers with $m\le k$.
Then $G$ has a $(g,f)$-factor if and only if 
for all subsets $A$ and $B$ of $V(G)$ satisfying either $A\subseteq X$ and $B\subseteq Y$ or 
$A\subseteq Y$ and $B\subseteq X$,
\begin{equation}\label{eq:m per k:bipartite}
0\le \sum_{v\in A} (f(v)-\frac{m}{k}d_G(v))+\sum_{v\in B} ( \frac{m}{k}d_G(v)-g(v) ) 
+\frac{m}{k} d_{G\setminus B}(A)+\frac{k-m}{k} d_{G\setminus A}(B).
\end{equation}
}\end{cor}
\begin{proof}
{It is enough to apply Theorem~\ref{thm:Folkman:Fulkerson--1970} with the identity $d_G(A,B)=\sum_{v\in A} \frac{m}{k}d_G(v)-\frac{m}{k}d_{G\setminus B}(A)+\sum_{v\in B} \frac{k-m}{k}d_G(v)-\frac{k-m}{k}d_{G\setminus A}(B)$
}\end{proof}
Now, we are in a position to prove the following theorem which plays an essential role in this paper.
\begin{thm}\label{thm:strongly-equitable:bipartite}
{Let $H$ be a bipartite graph with bipartition $(V_1,V_2)$, let $z_i\in V_i$, and let $\sigma=\sum_{i=1,2} d_{H}(z_i)$.
Assume that for all $X\subsetneq V(G)$ with $\{z_1,z_2\}\subseteq X$, $d_G(X)\ge k-1$.
Then $H$ admits a factorization $H_1,\ldots, H_k$ with the following properties:
\begin{enumerate}{
\item [$\bullet$] 
For any graph $H_i$, $||E(H_i)|-\frac{1}{k} |E(H)||<1$.
\item [$\bullet$] 
For all $v\in V(H_i)\setminus \{z_1,z_2\}$, $|d_{H_i}(v)-\frac{1}{k} d_{H}(v)|<1$.

\item [$\bullet$] 
$|\sum_{j=1,2}d_{H_i}(z_j)-\frac{1}{k}(\sum_{j=1,2} d_{H}(z_j))|<1$.

}\end{enumerate}
if there exists a positive integer $m$ satisfying the following properties:
\begin{enumerate}{

\item [$\bullet$] 
For all $v\in V_1 \setminus \{z_1\}$, $[d_H(v)]_k\le k-m$, 
and for all $v\in V_2\setminus \{z_2\}$, $[d_H(v)]_k\ge m$ or $[d_H(v)]_k= 0$.

\item [$\bullet$] 
One of the following conditions holds:

\begin{enumerate}{
\item [$(i)$] 
$[d_H(z_2)]_k= k- m$, $[\sigma]_k\le k-m$, and
$$[\sigma]_k+\sum_{v\in V_1\setminus \{z_1\}}[d_H(v)]_k+
\sum_{v\in V_2\setminus \{z_2\}}[k-d_H(v)]_k \neq 
k-2m.$$

\item [$(ii)$] 
 $[d_H(z_1)]_k= m$, $[\sigma]_k\ge  m$ or $[\sigma]_k=0$,  and 
$$\sum_{v\in V_1\setminus \{z_1\}}[d_H(v)]_k+
\sum_{v\in V_2\setminus \{z_2\}}[k-d_H(v)]_k +[k-\sigma]_k\neq 
k-2m.$$
}\end{enumerate}
}\end{enumerate}
}\end{thm}
\begin{proof}
{Let $Z=\{z_1,z_2\}$.
For all $v\in V(H)\setminus Z$, we define
$g(v)=\max\{0, r(v)-(k-m)\}+\frac{m}{k} (d_H(v)-r(v))$ and $f(v)=\min\{m, r(v)\}+\frac{m}{k} (d_H(v)-r(v))$,
where $r(v)=[d_H(v)]_k$.
It is easy to check that $m\lfloor \frac{1}{k} d_H(v) \rfloor\le g(v)\le \frac{m}{k}d_H(v) \le f(v)\le m\lceil \frac{1}{k} d_H(v) \rceil$. In other words, 
\begin{equation}\label{eq:g-f:v}
\lfloor \frac{1}{k}d_H(v)\rfloor \le \lfloor \frac{1}{m}g(v)\rfloor \le \lceil \frac{1}{m}f(v)\rceil \le \lceil \frac{1}{k}d_H(v)\rceil.
\end{equation}
Likewise,
\begin{equation}\label{eq:g-f:v:complement}
\lfloor \frac{1}{k}d_H(v)\rfloor \le \lfloor \frac{1}{k-m}(d_H(v)-f(v))\rfloor \le \lceil \frac{1}{k-m}(d_H(v)-g(v))\rceil \le \lceil \frac{1}{k}d_H(v)\rceil.
\end{equation}
Let $r=[\sigma]_k$. In case of item (i), 
we define $g(z_1)=\max\{0, r-(k-m)\}+\frac{m}{k} (d_H(z_1)-m-r)+m$, 
$f(z_1)=\min\{m, r\}+\frac{m}{k} (d_H(z_1)-m-r)+m$, 
and $g(z_2)=f(z_2)=\frac{m}{k} (d_H(z_2)-(k-m))$.
Since $d_H(z_1)\stackrel{k}{\equiv} m+r$ and $d_H(z_2)\stackrel{k}{\equiv}  k- m$, both of $f(z_i)$ and $g(z_i)$ are integer.
In case of item (ii), 
we define $g(z_1)=f(z_1)=\frac{m}{k}(d_H(z_1)-m)+m$, $g(z_2)=\max\{0, r-(k-m)\}+\frac{m}{k} (d_H(z_2)-(k-m)-r)$, and 
$f(z_2)=\min\{m, r\}+\frac{m}{k} (d_H(z_2)-(k-m)-r)$.
Since $d_H(z_1)\stackrel{k}{\equiv} m$ and $d_H(z_2)\stackrel{k}{\equiv}  k- m+r$, both of $f(z_i)$ and $g(z_i)$ are integer.
Recall that for item (i), $r\le k-m$, and for item (ii), either $r\ge m$ or $r=0$.
Thus, in both cases, 
we must have $\frac{m}{k}d_H(z_1)\le g(z_1)\le f(z_1)$ and $g(z_2)\le f(z_2)\le \frac{m}{k}d_H(z_2)$.
Moreover, 
\begin{equation}\label{eq:g-f:z}
\lfloor \frac{1}{k}\sum _{i=1,2} d_H(z_i)\rfloor \le 
\sum _{i=1,2} \lfloor \frac{1}{m}g(z_i)\rfloor \le 
\sum _{i=1,2} \lceil \frac{1}{m}f(z_i)\rceil 
\le \lceil \sum _{i=1,2} \frac{1}{k}d_H(z_i)\rceil.
\end{equation}
Likewise, 
\begin{equation}\label{eq:g-f:z:complement}
\lfloor \frac{1}{k}\sum _{i=1,2} d_H(z_i)\rfloor \le 
\sum _{i=1,2}\lfloor \frac{1}{k-m}(d_H(z_i)-g(z_i))\rfloor 
 \le \sum _{i=1,2} \lceil \frac{1}{k-m}(d_H(z_i)-g(z_i))\rceil \le 
\lceil \sum _{i=1,2} \frac{1}{k}d_H(z_i)\rceil.
\end{equation}
We are going to show that $G$ admits a $(g, f)$-factor based on Corollary~\ref{cor:Folkman:Fulkerson--1970}.
Let $A$ and $B$ be two disjoint subsets of $V(G)$ 
for which either $A\subseteq V_1$ and $B\subseteq V_2$ or $A\subseteq V_2$ and $B\subseteq V_1$.
If $z_2\not \in A$, then for all $v\in A$, $f(v)- \frac{m}{k}d_H(v)\ge 0$.
If $z_1\not \in B$, then for all $v\in B$, $\frac{m}{k}d_H(v)-g(v)\ge 0$.
Consequently, Inequation~(\ref{eq:m per k:bipartite}) holds when $z_2 \not \in A$ and $z_1 \not\in B$.
We may therefore assume that $z_2 \in A$ or $z_1 \in B$. 
In particular, $A\subseteq V_2$ and $B\subseteq V_1$.
By the assumption, for all $v\in B\setminus\{z_1\}$, $r(v)=[d_H(v)]_k\le k-m$,
and so 
\begin{equation}\label{eq:g-difference}
\frac{m}{k}d_H(v)-g(v)=\frac{m}{k}r(v)- \max\{0, r(v)-(k-m)\}=\frac{m}{k}[d_H(v)]_k.
\end{equation}
Likewise, for all $v\in A\setminus \{z_2\}$, 
$r(v)\ge m$ or $r(v)=0$, and so 
\begin{equation}\label{eq:f-difference}
f(v)-\frac{m}{k}d_H(v)=\min\{m, r(v)\}-\frac{m}{k}r(v) =\frac{m}{k}[k-d_H(v)]_k.
\end{equation}
Recall that in case of item (i), we have $r\le k-m$. Also, in case of item (ii), we have either $r\ge m$ or $r=0$.
Thus
\begin{equation}\label{eq:g-difference:z1}
\frac{m}{k}d_H(z_1)-g(z_1)=-\frac{m(k-m)}{k}+
\begin{cases}
\frac{m}{k}[\sigma]_k,	&\text{when item (i) holds};\\
0,	&\text{when item (ii) holds},
\end {cases}
\end{equation}
and
\begin{equation}\label{eq:f-difference:z2}
 f(z_2)-\frac{m}{k}d_H(z_2)=-\frac{m(k-m)}{k}+
\begin{cases}
0,	&\text{when item (i) holds};\\
\frac{m}{k}[k-\sigma]_k,	&\text{when item (ii) holds}.
\end {cases}
\end{equation}
For notational simplicity, let us define 
$$\alpha=\sum_{v\in B\setminus \{z_1\}}[d_H(v)]_k+c_{A, B}+
\sum_{v\in A\setminus \{z_2\}}[k-d_H(v)]_k+\lambda_a,$$
where in case of item (i), $c_{A, B}\in \{0,[\sigma]_k\}$ and $c_{A, B}=0$ if and only if $z_1\not \in B$, 
and in case of item (ii), $c_{A, B}\in \{0, [k-\sigma]_k\}$ and $c_{A, B}=0$ if and only if $z_2\not\in A$.
In order to derive Inequation~(\ref{eq:m per k:bipartite}), according to Equations~(\ref{eq:g-difference}),~(\ref{eq:f-difference}),~(\ref{eq:g-difference:z1}), and~(\ref{eq:f-difference:z2}), it is enough to prove
\begin{equation}
\frac{m}{k} \alpha +
\frac{k-m}{k}\lambda_b\ge \begin{cases}
\frac{2m(k-m)}{k},	&\text{when $z_2\in A$ and $z_1\in B$};\\
\frac{m(k-m)}{k},	&\text{when either $z_2\in A$ or $z_1\in B$},
\end {cases}
\end{equation}
where $\lambda_a=d_{H\setminus B}(A)$ and $\lambda_b=d_{H\setminus A}(B)$.
For proving this part, we heavily need the following equations:
\begin{equation}\label{eq:modulo}
\lambda_b\stackrel{k}{\equiv} \sum_{v\in B}[d_H(v)]_k+\sum_{v\in A}[k-d_H(v)]_k +\lambda_a
\stackrel{k}{\equiv} \alpha +
\begin{cases}
2m,	&\text{when $z_2\in A$ and $z_1\in B$};\\
m,	&\text{when either $z_2\in A$ or $z_1\in B$}.
\end {cases}
\end{equation}
The left side of this equation can easily be derived by the identity 
$\sum_{v\in B}d_H(v)-d_{H\setminus A}(B)=d_H(B,A)= \sum_{v\in A}d_H(v)-d_{H\setminus B}(A)$.
Now, we are ready to consider the following three cases:
%
%
\vspace{3mm}
\\
{\bf Case 1. $A\cup B=V(G)$.}

In this case, $\lambda_a=\lambda_b = 0$, $A=V_2$, and $B=V_1$. 
Thus Equation~(\ref{eq:modulo}) implies that
$\alpha
\stackrel{k}{\equiv} k-2m$.
The variable $\alpha$ is obviously nonnegative. 
By the assumption, it is not equal to $k-2m$, and so it must be at least $2k-2m$.
Consequently, Inequation~(\ref{eq:m per k:bipartite}) holds. $\square$
%
%
%
\vspace{3mm}
\\
{\bf Case 2. $z_1 \in B$ and $z_2 \in A$.}

According to Case 1, we may assume that $A\cup B\subsetneq V(G)$.
We deduce that, by the assumption, $\lambda_a+\lambda_b=d_H(A\cup B)\ge k-1$.
On the other hand, by Equation~(\ref{eq:modulo}), we have
$\alpha \stackrel{k}{\equiv} k-2m+\lambda_b$.
Thus if $\lambda_b < m$, then $\alpha \ge \lambda_a \ge k-1-\lambda_b > k-2m+\lambda_b$ which implies that 
$\alpha \ge 2k-2m+\lambda_b\ge 2k-2m$.
Also, if $2m > \lambda_b\ge m$, then $k > k-2m+\lambda_b\ge 0$ which implies that $\alpha \ge k-2m+\lambda_b\ge k-m$.
Therefore, regardless of $ \lambda_b\ge 2m$ or not, one can conclude that
$ \frac{m}{k}\alpha
+\frac{k-m}{k}\lambda_b
\ge
\frac{2m(k-m)}{k}$. 
Consequently, Inequation~(\ref{eq:m per k:bipartite}) holds. $\square$
%
%
\vspace{3mm}
\\
{\bf Case 3. Either $z_1\in B$ and $z_2\not \in A$ or $z_1\not \in B$ and $z_2 \in A$.}

By Equation~(\ref{eq:modulo}), we have
$\alpha \stackrel{k}{\equiv} k-m+\lambda_b$.
Therefore, $\lambda_b\ge m$ or $\alpha\ge k-m+\lambda_b \ge k-m$. In both cases, 
$\frac{m}{k}\alpha+\frac{k-m}{k}\lambda_b\ge \frac{m(k-m)}{k}$.
Consequently, Inequation~(\ref{eq:m per k:bipartite}) again holds. $\square$
Therefore, by Corollary~\ref{cor:Folkman:Fulkerson--1970}, the graph $H$ has a factor $F$
such that for each vertex $v$, $g(v)\le d_{F}(v)\le f(v)$. 
According to Theorem~\ref{thm:Werra}, the graph $F$ has a factorization $H_1, \ldots, H_m$ 
such that for each $v\in V(H_j)$ with $1\le j \le m$,
$|d_{H_j}(v)-d_F(v)/m|<1$.
If $v\in V(H_j)\setminus Z$, then 
by Inequation~(\ref{eq:g-f:v}), we must have
$ \lfloor d_H(v)/k \rfloor \le \lfloor g(v)/m \rfloor \le d_{H_j}(v)\le \lceil f(v)/m \rceil \le \lceil d_H(v)/k \rceil$.
Likewise, 
by Inequation~(\ref{eq:g-f:z}), 
$ \lfloor \sum _{i=1,2} d_H(z_i)/k \rfloor \le \sum _{i=1,2}\lfloor g(z_i)/m \rfloor \le
 \sum _{i=1,2} d_{H_j}(z_i)\le \sum _{i=1,2} \lceil f(z_i)/m \rceil \le \lceil \sum _{i=1,2} d_H(z_i)/k \rceil$.
Let $F_0$ be the complement of $F$ in $H$. For each vertex $v$, we have $d_H(v)-f(v)\le d_{F_0}(v)\le d_H(v)-g(v)$. 
According to Theorem~\ref{thm:Werra}, the graph $F_0$ has a factorization $H_{m+1}, \ldots, H_k$ such that for each vertex $v\in V(H_j)$ with $m< j \le k$,
$|d_{H_j}(v)-d_{F_0}(v)/(k-m)|<1$.
If $v\in V(H_j)\setminus Z$, then again 
by Inequation~(\ref{eq:g-f:v:complement}), we must have
$ \lfloor d_H(v)/k \rfloor \le \lfloor (d_H(v)-f(v))/(k-m) \rfloor 
\le d_{H_j}(v)\le
\lceil (d_H(v)-g(v))/(k-m) \rceil \le \lceil d_H(v)/k \rceil$.
Likewise, 
by Inequation~(\ref{eq:g-f:z:complement}), 
$ \lfloor \sum _{i=1,2} d_H(z_i)/k \rfloor \le\sum _{i=1,2} \lfloor (d_H(z_i)-f(z_i))/(k-m) \rfloor 
\le \sum _{i=1,2} d_{H_j}(z_i)\le
\sum _{i=1,2} \lceil (d_H(z_i)-g(z_i))/(k-m) \rceil \le \lceil \sum _{i=1,2} d_H(z_i)/k \rceil$.
It is not difficult to check that $H_1,\ldots, H_k$ is a factorization satisfying the second and the third desired conditions.

We are going to confirm the remaining part by imposing the first desired condition on sizes.
First consider a $k$-coloring of the edges of $G$ obtained from this factorization.
For every $v\in V(H)\setminus Z$, we split it into vertices $v_0,\ldots, v_n$ such that every vertex $v_i$ with $1\le i\le n$ has degree $k$ and it is incident with all $k$ colors $c_1,\ldots, c_k$. 
Likewise, for the exceptional $v_0$, it has degree at most $k$ and it is not incident with two edges having the same color. 
For the vertex $z_j$, we split it into vertices $z'_j, z_{j,0},\ldots, z_{j,n}$ such that every vertex $z_{j,i}$ with $1\le i\le n$ has degree $k$ and it is incident with all $k$ colors $c_1,\ldots, c_k$. Likewise, for the exceptional $z_{j,0}$, it has degree at most $k$ and it is not incident with two edges having the same color. 
If $d_{H}(z_1)\neq 0$ and $d_{H}(z_2)\neq 0$ (the proof of the cases that $d_{H}(z_1)= 0$ or $d_{H}(z_2)= 0$ are similar), then according to the construction, $d_{H}(z_1)\ge  m$ and $d_{H}(z_2)\ge k-m$ and we can impose that $z'_1$ is incident with all colors $c_1,\ldots, c_m$ and so it has degree $m$, and $z'_2$ is incident with all colors $c_{m+1},\ldots, c_k$ and so it has degree $k-m$. 
Consequently, $z'_1$ and $z'_2$ must not be adjacent. 
Then we shrink them into a new vertex $z$ of degree $k$ which is incident with all colors $c_1,\ldots, c_k$.
Call the resulting loopless graph $H_0$. 
Note that $H_0$ has maximum degree at most $k$, and also it admits a proper edge-coloring with $k$ colors $c_1,\ldots, c_k$. 
Let us consider such a coloring with the minimum $\sum_{1\le i\le k}|m_i-m/k|$, 
where $m_i$ is the number of edges colored with $c_i$ and  $m$ is the size of $H$.
We claim that $|m_i-m/k|<1$ for every $c_i$.
Suppose, to the contrary, that there is a color $c_i$ with $|m_i-m/k|\ge 1$. 
We may assume that $m_i\ge m/k+1$; as the proof of the case $m_i\le m/k-1$ is similar.
Thus there is another color $c_j$ with $m_j<m/k$ so that $m_i> m_j+1$.
Let $H'_0$ be the factor of $H_0$ consisting of the edges that are colored with $c_i$ or $c_j$.
According to the proper edge coloring of $H_0$, the graph $H'_0$ must be the union of some paths and cycles.
 Since all cycles have even size and $m_i> m_j$, it is easy to check that there is a path $P$ 
such that whose edges are colored alternatively by $c_i$ and $c_j$, and also whose end edges are colored with $c_i$. 
Now, it is enough to exchange the colors of the edges of $P$ to find another proper edge-coloring satisfying 
 $m'_i=m_i-1$ and $m'_j=m_j+1$ where $m'_i$ and $m'_j$ are the number of edges having these new colors.
Since $|m'_j-m/k|+|m'_i-m/k|< (|m_j-m/k|+1)+(|m_i-m/k|-1)$, one can easily derive a contradiction.
Finally, we redefine $H_i$ to be the factor of $H$ obtained from all edges colored with $c_i$.
Therefore, $H_1,\ldots, H_k$ are the desired factors that we are looking for.
}\end{proof}
\subsection{Graphs with many vertices whose degrees are not divisible by $k$}
The following lemma establishes a sufficient degree condition for the existence of equitable factorizations in highly edge-connected graphs. 
(This condition is necessary for $k=2,3$, see \cite[Corollary 3.2]{EquitableFactorizations-2022}).
\begin{lem}{\rm (\cite{EquitableFactorizations-2022})}\label{lem:m=0}
{Let $G$ be a graph satisfying $|E(G)|\stackrel{k}{\equiv}\sum_{v\in Z}d_G(v)$ for a vertex subset $Z\subseteq V(G)$.
If $G$ is $(3k-3)$-edge-connected, then it can be edge-decomposed into $k$ factors $G_1,\ldots, G_k$ such that for each factor $G_i$, $||E(G_i)|-|E(G)|/k|\le 1$, and for each $v\in V(G_i)$, $1\le i\le k$, $|d_{G_i}(v)-d_G(v)/k|< 1$.
}\end{lem}
The following theorem introduces another similar sufficient condition for the existence of equitable factorizations based on a combination of Corollary~\ref{cor:split:tree-connected:3m} and Theorem~\ref{thm:strongly-equitable:bipartite}. At first glance, one might think it covers a small ratio of graphs, but we will surprisingly use this result to characterize degree sequences modulo $k$ without equitable factorizations in the subsequent subsection.
\begin{thm}\label{thm:new-sufficient-condition}
{Let $G$ be a $5(k-1)$-tree-connected graph.
Then $G$ can be edge-decomposed into $k$ factors $G_1,\ldots, G_k$ such that for each factor $G_i$, $||E(G_i)|-|E(G)|/k|< 1$,
 and for each $v\in V(G_i)$, $1\le i\le k$, 
$ | d_{G_i}(v)-d_G(v)/k|< 1$, if there exists a subset $Z\subseteq V(G)$ satisfying the following properties:
\begin{enumerate}{
\item [$(1)$]
$m = [|E(G)|-\sum_{v\in Z}d_G(v)]_k$.

\item [$(2)$]
For all $v\in Z$, $[d_G(v)]_k\le  k-m$, and for all $v\in V(G)\setminus Z$, $[d_G(v)]_k\ge m$.

\item [$(3)$]
 $\sum_{v\in Z}[d_G(v)]_k+\sum_{v\in V(G)\setminus Z}[k-d_G(v)]_k\neq k-2m$.
}\end{enumerate}
}\end{thm}
\begin{proof}
{Choose $z\in V(G)$. 
Let $Z_0=V(G)\setminus Z$. 
For all $v\in V(G)\setminus \{z\}$, we define $p(v)=d_G(v)$ when $v\in Z$, 
and we define $p(v)=0$ when $v\in Z_0$.
Likewise, we define $p(z)=d_G(z)+m$ when $z\in Z$,
and we define $p(z)=m$ when $z\in Z_0$.
According to the first item, 
$|E(G)| \stackrel{k}{\equiv}
 \sum_{v\in V(G)}p(v)$.
Since $G$ is $(2k-2+3\lambda)$-tree-connected for which $\lambda = k-1$, by applying Corollary~\ref{cor:split:tree-connected:3m} together the function $p$, 
every vertex $v$ can be split into two vertices $v_1$ and $v_2$
such that the resulting graph $H$ satisfies the following properties:
(i1)
for all $v\in V(G)$, $d_H({v_1})\stackrel{k}{\equiv}p(v)$,
(i2)
$H$ is a bipartite graph with the bipartition $(V_1,V_2)$ for which $V_i=\{v_i:v\in V(G)\}$, and
(i3)
for all $X\subsetneq V(H)$ with $\{z_1,z_2\}\subseteq X$, $d_H(X)\ge \lambda $.
For all $v\in Z\setminus \{z\}$, $[d_H(v_1)]_k=[d_G(v)]_k$ and $[d_H(v_2)]_k=0$.
Note that for all $v\in Z_0\setminus \{z\}$, $[d_H(v_1)]_k=0$ and $[d_H(v_2)]_k=[d_G(v)]_k$.
If $z\in Z$, then 
$[d_H({z_2})]_k= k-m$, $[d_H(z_1)+d_H(z_2)]_k=[d_G(z)]_k\le k-m$, and  
$$[d_H(z_1)+d_H(z_2)]_k+\sum_{v\in V_1\setminus \{z_1\}}[d_H(v_1)]_k+\sum_{v\in V_2\setminus \{z_2\}}[k-d_H(v_2)]_k
=
\sum_{v\in Z}[d_G(v)]_k+\sum_{v\in Z_0}[k-d_G(v)]_k\neq k-2m.$$
If $z\in Z_0$, then  $[d_H({z_1})]_k=m$,  $[d_H(z_1)+d_H(z_2)]_k=[d_G(z)]_k\ge m$,  and 
$$\sum_{v\in V_1\setminus \{z_1\}}[d_H(v_1)]_k+\sum_{v\in V_2\setminus \{z_2\}}[k-d_H(v_2)]_k + [k-d_H(z_1)-d_H(z_2)]_k
=
\sum_{v\in Z}[d_G(v)]_k+\sum_{v\in Z_0}[k-d_G(v)]_k\neq k-2m.$$
Thus by Theorem~\ref{thm:strongly-equitable:bipartite}, the graph $H$ admits a factorization $H_1,\ldots, H_k$ such that 
for each graph $H_i$, $||E(H_i)|-|E(H)|/k|< 1$,
for all $v\in V(H_i)\setminus \{z_1,z_2\}$, 
$|d_{H_i}(v)-d_{H}(v)/k|<1$, and 
$|d_{H_i}(z_1)+d_{H_i}(z_2)-(d_{H}(z_1)+d_{H}(z_2))/k|<1$.
Let $G_i$ be the factor of $G$ obtained from $H$. Obviously, $||E(G_i)|-|E(G)|/k|< 1$ and $|d_{G_i}(z)-d_{G}(z)/k|<1$.
For all $v\in V(G)\setminus \{z\}$, $d_H(v_1)$ or $d_H(v_2)$ is divisible by $k$. This implies that $d_H(v_1)=d_H(v_1)/k$ or 
$d_H(v_2)=d_H(v_2)/k$. Since $d_{G_i}({v})=d_{H_i}({v_1})+d_{H_i}({v_2})$, we must therefore have
$|d_{G_i}(v)-d_{G}(v)/k|<1$.
Hence the proof is completed.
}\end{proof}
For working with Theorem~\ref{thm:new-sufficient-condition}, we need the following lemma to find fruitful vertex sets.
\begin{lem}\label{lem:Z:equitable:minimum}
{Let $G$ be a graph. If $Z$ is a subset of $V(G)$ with the minimum $m = [|E(G)|-\sum_{v\in Z}d_G(v)]_k$ including all vertices $v$ with $[d_G(v)]_k= 0$,
then the following properties hold:
\begin{enumerate}{
\item [$(1)$]
$m\le k/2$.

\item [$(2)$]
For all $v\in Z$, $[d_G(v)]_k\le k-2m$, and for all $v\in V(G)\setminus Z$, $[d_G(v)]_k\ge 2m$.
}\end{enumerate}
}\end{lem}
\begin{proof}
{Let $Z_0=V(G)\setminus Z$. We may assume that $m\neq 0$. 
First, we claim that $m\le k/2$. Otherwise, 
$[|E(G)|-\sum_{v\in Z_0}d_G(v)]_k=[\sum_{v\in Z}d_G(v)-|E(G)|]_k=k-m< m$ which is a contradiction.
Moreover, for all $u\in Z$, we must have $[d_G(u)]_k\le k-2m$.
Otherwise, if $k-2m < [d_G(u)]_k \le k-m$, then 
$[|E(G)|-\sum_{v\in Z_0\cup \{u\}}d_G(v)]_k=[k-m-[d_G(u)]_k]_k< m$, and if $ [d_G(u)]_k\ge k-m$, then
$[|E(G)|-\sum_{v\in Z\setminus \{u\}}d_G(v)]_k=[m+[d_G(u)]_k]_k< m$. 
A contradiction.
Furthermore, for all $v\in Z_0$ with $[d_G(v)]_k\neq 0$, we must have $[d_G(v)]_k\ge 2m$.
Otherwise, if $m \le [d_G(v)]_k< 2m$, then 
$[|E(G)|-\sum_{v\in Z_0\setminus  \{u\}}d_G(v)]_k=[k-m+[d_G(u)]_k]_k< m$, 
and if $0< [d_G(u)]_k\le  m$, then $[|E(G)|-\sum_{v\in Z\cup  \{u\}}d_G(v)]_k=[m-[d_G(u)]_k]_k< m$. 
This is again a contradiction. Hence the lemma holds.
}\end{proof}
The following corollary shows that every highly edge-connected graph having at least $k-1$ vertices with degree not divisible by $k$ admits a $k$-equitable factorization. By a computer search, we observed that for all $k\le 7$, the degree conditions in Corollary~\ref{cor:sum-min-x-k-x} and Theorem~\ref{intro:thm:factorization:Z} are not only sufficient but also necessary for the existence of $k$-equitable factorizations in highly edge-connected graphs, see Table 1.
\begin{cor}\label{cor:sum-min-x-k-x}
{If $G$ is a $5(k-1)$-tree-connected graph satisfying $$\sum_{v\in V(G)}\min \{[d_G(v)]_k, [k-d_G(v)]_k\}\ge k-1,$$
then it can be edge-decomposed into $k$ factors $G_1,\ldots, G_k$ such that for each factor $G_i$, $||E(G_i)|-|E(G)|/k|\le 1$, and for each $v\in V(G_i)$, $1\le i\le k$, $ | d_{G_i}(v)-d_G(v)/k|< 1$. 
}\end{cor}
\begin{proof}
{Let $Z\subseteq V(G)$ with the minimum
$m = [|E(G)|-\sum_{v\in Z}d_G(v)]_k$.
If $m=0$, then the assertion follows from Lemma~\ref{lem:m=0}.
We may assume that $m\neq 0$ and $Z$ contains all vertices $v$ with $[d_G(v)]_k=0$.
By Lemma~\ref{lem:Z:equitable:minimum}, for all $v\in Z$, $[d_G(v)]_k\le k-2m < k-m$, 
and for all $v\in V(G)\setminus Z$, $[d_G(v)]_k\ge 2m$.
Moreover, by the assumption, 
 $$\sum_{v\in Z}[d_G(v)]_k+\sum_{v\in V(G)\setminus Z}[k-d_G(v)]_k\ge 
\sum_{v\in V(G)}\min \{[d_G(v)]_k, [k-d_G(v)]_k\}\ge k-1 > k-2m.$$
Therefore, the assertion follows from Theorem~\ref{thm:new-sufficient-condition}.
}\end{proof}
\begin{table}\center
{\small \label{Table}
\begin{tabular}[ht]{|l| l | l | l | l | l l| l l l |}\hline
&$k=2$& $k=3$ 	& $k=4$ 		& $k=5$		& $k=6$		& 				& $k=7$ 		& 			& 		\\\hline
1&$(0)^*$& $(1)$	& $(0)^*$ 	& $(1)$		& $(0)^*$	& $(4,4)^*$ 		& $(1)$		& $(4,6)$	& $(5,5,6)$	\\
2&& $(2)$		& $(2)^*$ 	& $(2)$		& $(2)^*$	& $(5,5)^*$		& $(2)$		& $(5,6)$ 	& $(6,6,6)$	\\
3&& 			& $(2)$ 		& $(3)$		& $(2)$		& $(1,1,2)^*$	& $(3)$		& $(1,1,1)$	& $(1,1,1,2)$	\\
4&& 			& $(1,1)^*$	& $(4)$		& $(4)^*$	& $(1,1,4)$		& $(4)$		& $(1,1,3)$	& $(1,1,1,5)$	\\
5&& 			& $(1,3)$	& $(1,2)$	& $(4)$		& $(1,2,5)$		& $(5)$		& $(1,1,4)$	& $(1,1,2,6)$	\\
6&& 			& $(3,3)^*$ 	& $(1,3)$	& $(1,1)^*$	& $(1,4,5)^*$	& $(6)$		& $(1,1,6)$	& $(1,1,5,6)$	 \\
7&& 			& 			& $(2,4)$	& $(1,3)^*$	& $(2,5,5)^*$	& $(1,2)$	& $(1,2,2)$	& $(1,2,6,6)$ \\
8&& 			& 			& $(3,4)$	& $(1,3)$	& $(4,5, 5)$		& $(1,3)$	& $(1,2,5)$	& $(1,5,6,6)$	\\
9&& 			& 			& $(1,1,1)$	& $(1,5)$	& $(1,1, 1,1)^*$	& $(1,4)$	& $(1,3,6)$	& $(2,6,6,6)$ \\
10&& 			& 			& $(1,1,4)$	& $(2,2)^*$	& $(1,1, 1,5)$		& $(1,5)$	& $(1,4,6)$ 	& $(5,6,6,6)$	\\
11&& 			& 			& $(1,4,4)$	& $(2,4)$	& $(1,1, 5,5)^*$	& $(2,3)$	& $(1,5,5)$	& $(1,1,1,1,1)$\\
12&& 			& 			& $(4,4,4)$	& $(3,5)^*$	& $(1,5, 5,5)$		& $(2,4)$ 	& $(1,6,6)$	& $(1,1,1,1,6)$ \\
13&& 			& 			&			& $(3,5)$	& $(5,5, 5,5)^*$	& $(2,6)$	& $(2,2,6)$	& $(1,1,1,6,6)$ \\
14&& 			& 			&			&			&				& $(3,5)$ 	& $(2,5,6)$	& $(1,1,6,6,6)$\\
15&& 			& 			&			&			&				& $(3,6)$ 	& $(3,6,6)$	& $(1,6,6,6,6)$\\
16&& 			& 			&			&			&				& $(4,5)$ 	& $(4,6,6)$ & $(6,6,6,6,6)$\\
 \hline
\multicolumn{10}{|c|}{$(x_1,\ldots, x_n)^*$ means that $|E(G)|\stackrel{k}{\equiv}\sum_{1\le i\le n}x_i+k/2$.}
 \\\hline
\end{tabular}
}
\caption{All degree sequences modulo $k$, $k\le 7$, for which there is no $k$-equitable factorizations in highly edge-connected graphs $G$ (by ignoring zero elements).}
\end{table}
%
%
%
%
%
%
%
\subsection{A complete characterizations: binary sequences $x_i(v)$ and vertex sets $Z$}
\label{sec:characterizations}
The following lemma shows that if a vertex $Z$ obtained in Lemma~\ref{lem:Z:equitable:minimum} does not meet the third assumption of Theorem~\ref{thm:new-sufficient-condition}, then it satisfies a simple but important parity property and consequently this vertex set is unique.
\begin{lem}\label{lem:unique:Z}
{Let $G$ be a graph and let $Z\subseteq V(G)$ including all vertices $v$ with $[d_G(v)]_k=0$. 
Assume that $\sum_{v\in Z}[d_G(v)]_k+\sum_{v\in V(G)\setminus Z}[k-d_G(v)]_k= k-2m$.
Then $m=[|E(G)|-\sum_{v\in Z}d_G(v)]_k$ if and only if
$$|V(G)\setminus Z| \not \stackrel{2}{\equiv} \sum_{v\in V(G)}\lfloor d_G(v)/k\rfloor.$$
Consequently, there is at most one vertex set $Z_0$ including all vertices $v$ with $[d_G(v)]_k=0$ satisfying 
 $m_0=[|E(G)|-\sum_{v\in Z_0}d_G(v)]_k>0$ and 
$\sum_{v\in Z_0}[d_G(v)]_k+\sum_{v\in V(G)\setminus Z_0}[k-d_G(v)]_k= k-2m_0$.
}\end{lem}
\begin{proof}
{Let $s= \sum_{v\in V(G)}\lfloor d_G(v)/k\rfloor + |V(G)\setminus Z| $.
Since $Z$ includes all vertices $v$ with $[d_G(v)]_k=0$, it is not difficult to check that
$2|E(G)|=\sum_{v\in V(G)} d_G(v)=
\sum_{v\in V(G)}\big(k \lfloor \frac{d_G(v)}{k}\rfloor +[d_G(v)]_k\big)=
ks+
\sum_{v\in Z}[d_G(v)]_k-\sum_{v\in V(G)\setminus Z}[k-d_G(v)]_k$.
This implies that 
$|E(G)|\stackrel{k}{\equiv} \frac{k}{2}[s]_2+
\frac{1}{2}\big(\sum_{v\in Z}[d_G(v)]_k-\sum_{v\in V(G)\setminus Z}[k-d_G(v)]_k\big)$.
Since $\sum_{v\in Z}[d_G(v)]_k+\sum_{v\in V(G)\setminus Z}[k-d_G(v)]_k= k-2m$, one can deduce that
$|E(G)|+\frac{k}{2}-m\stackrel{k}{\equiv} \frac{k}{2}[s]_2 +\sum_{v\in Z}[d_G(v)]_k$.
This means that $m\stackrel{k}{\equiv}|E(G)|-\sum_{v\in Z}d_G(v)$ 
 if and only if $s$ is odd if and only if 
$|V(G)\setminus Z| \not \stackrel{2}{\equiv} \sum_{v\in V(G)}\lfloor d_G(v)/k\rfloor$.
Hence the first claim is proved.
Suppose, to the contrary, that there is another vertex set $Z_0$ including all vertices $v$ with $[d_G(v)]_k=0$ satisfying $\sum_{v\in Z_0}[d_G(v)]_k+\sum_{v\in V(G)\setminus Z_0}[k-d_G(v)]_k< k$ and $|V(G)\setminus Z_0| \not \stackrel{2}{\equiv} \sum_{v\in V(G)}\lfloor d_G(v)/k\rfloor$.
Since
$Z\neq Z_0$ and $|Z|+| Z_0|$ is even, we clearly have $|Z\Delta Z_0|\ge 2$, where $Z\Delta Z_0$ denotes the symmetric difference of $Z$ and $Z_0$. 
Therefore, $2k-2m>\sum_{v\in Z}[d_G(v)]_k+\sum_{v\in V(G)\setminus Z}[k-d_G(v)]_k+\sum_{v\in Z_0}[d_G(v)]_k+\sum_{v\in V(G)\setminus Z_0}[k-d_G(v)]_k\ge \sum_{v\in Z\Delta Z_0} ([d_G(v)]_k+[k-d_G(v)]_k)\ge k|Z\Delta Z_0|\ge 2k$, which is a contradiction.
Hence the proof is completed.
}\end{proof}
In the following theorem, we give a complete characterization of degree sequences modulo $k$ 
 for the existence of $k$-equitable factorizations in highly edge-connected graphs.
According to the following result, highly edge-connected graphs $G$ satisfying $|E(G)|\stackrel{8}{\equiv}3$
with degree sequence $(2,2,2)$ modulo $8$ admits a $8$-equitable factorization
 while neither of Theorem~\ref{intro:thm:factorization:Z} and Corollary~\ref{cor:sum-min-x-k-x} covers them.
\begin{thm}\label{thm:characterization}
{Let $G$ be a $5(k-1)$-tree-connected graph. Then $G$ admits a $k$-equitable factorization if and only if 
 for all $i\in \{1,\ldots, k\}$ and $v\in V(G)$, there are integers $x_{i}(v)\in \{0, 1\}$ satisfying the following properties:
\begin{enumerate}{
\item [$\bullet$] 
For all $v\in V(G)$, $\sum_{1\le i\le k}x_{i}(v)=[d_G(v)]_k$.
\item [$\bullet$] 
For all $i\in \{1,\ldots,k\}$, $\sum_{v\in V(G)}x_{i}(v)\stackrel{2}{\equiv} \sum_{v\in V(G)}\lfloor d_G(v)/k \rfloor $,
}\end{enumerate}
if and only if 
 for all $i\in \{1,\ldots, k\}$ and $v\in V(G)$, there are integers $y_{i}(v)\in \{0, 1\}$ satisfying the following properties:
\begin{enumerate}{
\item [$\bullet$] 
For all $v\in V(G)$, $\sum_{1\le i\le k}y_{i}(v)=[k-d_G(v)]_k$
\item [$\bullet$] 
For all $i\in \{1,\ldots,k\}$, $\sum_{v\in V(G)}y_{i}(v)\stackrel{2}{\equiv} \sum_{v\in V(G)}\lceil d_G(v)/k \rceil $.
}\end{enumerate}
}\end{thm}
\begin{proof}
{We only prove the first part, since for checking the second part, it is enough to set $y_i(v)=1-x_i(v)$ when $[d_G(v)]_k\neq 0$, and set $y_i(v)=0$ when $[d_G(v)]_k= 0$.
If $G$ admits a $k$-equitable factorization $G_1,\ldots, G_k$, then we define $x_i(v)=d_{G_i}(v)-\lfloor d_G(v)/k \rfloor$.
Since $\sum_{v\in V(G)}d_{G_i}(v)$ is even, we must have 
$\sum_{v\in V(G)}x_{i}(v)= \sum_{v\in V(G)} d_{G_i}(v) -\sum_{v\in V(G)}\lfloor d_G(v)/k \rfloor
 \stackrel{2}{\equiv} \sum_{v\in V(G)}\lfloor d_G(v)/k \rfloor $. This completes the proof of the necessity.
Conversely, assume that there are integers $x_{i}(v)\in \{0, 1\}$ satisfying the theorem.
Suppose, to the contrary, that $G$ does not have $k$-equitable factorizations. 
Let $Z\subseteq V(G)$ and let
$m = [|E(G)|-\sum_{v\in Z}d_G(v)]_k$.
Assume that $Z$ has the minimum $m$. 
According to Lemma~\ref{lem:m=0}, 
we may assume that $m\neq 0$ and $Z$ contains all vertices $v$ with $[d_G(v)]_k=0$.
By Lemma~\ref{lem:Z:equitable:minimum}, we must have $m\le k/2$,
for all $v\in Z$, $[d_G(v)]_k\le k-2m < k-m$, and for all $v\in V(G)\setminus Z$, $[d_G(v)]_k\ge 2m>m$.
Thus by Theorem~\ref{thm:new-sufficient-condition}, we must have 
$\sum_{v\in Z}[d_G(v)]_k+\sum_{v\in V(G)\setminus Z}[k-d_G(v)]_k= k-2m,$
and so
 $\sum_{v\in Z}\sum_{1\le i\le k}x_{i}(v)+\sum_{v\in V(G)\setminus Z}\sum_{1\le i\le k}(1-x_{i}(v))= k-2m$.
In other words,
 $\sum_{1\le i\le k}\big(\sum_{v\in Z}x_{i}(v)+\sum_{v\in V(G)\setminus Z}(1-x_{i}(v))\big)= k-2m$.
Thus for at least one index $i$, 
$0= \sum_{v\in Z}x_{i}(v)+\sum_{v\in V(G)\setminus Z}(1-x_{i}(v))\stackrel{2}{\equiv}
\sum_{v\in V(G)}\lfloor d_G(v)/k \rfloor+ |V(G)\setminus Z|.$
Therefore, by Lemma~\ref{lem:unique:Z}, we derive a contradiction, as desired.
}\end{proof}
It seems that the tree-connectivity needed in Theorem~\ref{thm:characterization} is not yet sharp and it can be reduced by a constant factor and so we propose the following conjecture for further investigation.
As far as we know, this upper bound cannot be reduced to $k-1$ by considering the graph obtained from $\lceil \frac{2}{3}k\rceil-1$ copies of a triangle provided that $k\stackrel{3}{\equiv}1$ and $k\ge 4$. 
However, this graph is $(k-1)$-tree-connected, it is not $k$-tree-connected 
(more generally, it is not the union of $k$ edge-disjoint factors with size at least two).
\begin{conj}\label{conj:main}
{Let $G$ be a $k$-tree-connected graph. Then $G$ admits a $k$-equitable factorization if and only if 
 for all $i\in \{1,\ldots, k\}$ and $v\in V(G)$, there are integers $x_{i}(v)\in \{0, 1\}$ satisfying the following properties:
\begin{enumerate}{
\item [$\bullet$] 
For all $v\in V(G)$, $\sum_{1\le i\le k}x_{i}(v)=[d_G(v)]_k$
\item [$\bullet$] 
For all $i\in \{1,\ldots,k\}$, $\sum_{v\in V(G)}x_{i}(v)\stackrel{2}{\equiv} \sum_{v\in V(G)}\lfloor d_G(v)/k \rfloor $.
}\end{enumerate}
}\end{conj}
We shall below introduce another method to produce either a two-dimensional sequence $x_i(v)$ satisfying Theorem~\ref{thm:characterization} or a special vertex set $Z$ (satisfying Theorem~\ref{thm:characterization:Z}).
\begin{thm}\label{thm:minimal-obstacle-sequence}
Let $G$ be a graph. 
For any $i\in \{1,\ldots, k\}$ and $v\in V(G)$, let $x_{i}(v)\in \{0, 1\}$ satisfying 
$ \sum_{1\le i\le k}x_{i}(v)=[d_G(v)]_k$ for all $v\in V(G)$.
Let $$I=\big\{i\in \{1,\ldots, k\}:\sum_{v\in V(G)}x_{i}(v)\not \stackrel{2}{\equiv} \sum_{v\in V(G)}\lfloor d_G(v)/k \rfloor\big\}.$$
If $I$ is empty, then there is not a vertex set $Z$ satisfying $m=[|E(G)|-\sum_{v\in Z}d_G(v)]_k> 0$ and 
$\sum_{v\in Z}[d_G(v)]_k+\sum_{v\in V(G)\setminus Z}[k-d_G(v)]_k= k-2m$.
Otherwise, if $I$ is nonempty and the two-dimensional sequence $x_i(v)$ produces the minimum size of $I$, then there exists a unique vertex $Z$ including all vertices $v$ with $[d_G(v)]_k=0$ satisfying the following properties:
\begin{enumerate}{
\item [$\bullet$] 
$|I|=2m$, $m=[|E(G)|-\sum_{v\in Z}d_G(v)]_k$, and $\sum_{v\in Z}[d_G(v)]_k+\sum_{v\in V(G)\setminus Z}[k-d_G(v)]_k=k-2m$.

\item [$\bullet$] 
 $Z=\{v\in V(G): x_i(v)=0\}$ for any $i\in I$, and 
$\big| \{v\in Z: x_i(v)= 1\} \cup \{v\in V(G)\setminus Z: x_i(v)= 0\}\big|= 1$ for any $i\in \{1,\ldots, k\}\setminus I$.
}\end{enumerate}
\end{thm}
\begin{proof}
{If $I$ is empty, then by the proof of Theorem~\ref{thm:characterization}, one can conclude that 
$|V(G)\setminus Z| \stackrel{2}{\equiv} \sum_{v\in V(G)}\lfloor d_G(v)/k\rfloor$. Thus by Lemma~\ref{lem:unique:Z}, there is not a vertex set $Z$ satisfying $m=[|E(G)|-\sum_{v\in Z}d_G(v)]_k> 0$ and 
$\sum_{v\in Z}[d_G(v)]_k+\sum_{v\in V(G)\setminus Z}[k-d_G(v)]_k=k-2m$.
If $I$ is not empty, then there is an index $j$ satisfying 
$\sum_{v\in V(G)}x_{j}(v)\not \stackrel{2}{\equiv} \sum_{v\in V(G)}\lfloor d_G(v)/k \rfloor $.
Let $Z=\{v\in V(G):x_j(v)=0\}$. Note that $Z$ includes all vertices $v$ with $[d_G(v)]_k=0$.
If for an index $i\in I$, there is a vertex $v$ with $x_i(v)\neq x_j(v)$, then we can replace 
$x_i(v)$ and $x_j(v)$ by $1-x_i(v)$ and $1-x_j(v)$ to obtain a new sequence with smaller set $I$ which is a contradiction. 
Note that $x_i(v)+x_j(v)=1-x_i(v)+1-x_j(v)$.
Obviously, for all vertices $v$, $k\lfloor d_{G}(v)/k\rfloor+[d_{G}(v)]_k$. Since $\sum_{v\in V(G)}d_G(v)$ is even, one can conclude that 
$\sum_{1\le i\le k} \sum_{v\in V(G)}\lfloor d_{G}(v)/k\rfloor
=
 \sum_{v\in V(G)}k \lfloor d_{G}(v)/k\rfloor
\stackrel{2}{\equiv}
 \sum_{v\in V(G)}[d_G(v)]_k
\stackrel{2}{\equiv} \sum_{v\in V(G)}\sum_{1\le i\le k} x_i(v)= \sum_{1\le i\le k}\sum_{v\in V(G)} x_i(v)$. 
This implies that $|I|$ is even and so $|I|=2m\ge 2$.
Let $j_0\in I\setminus \{j\}$.
If for an index $i\in \{1,\ldots, k\}\setminus I$, there are two vertices $v$ and $w$ satisfying $x_i(v)\neq x_j(v)$ and $x_i(w)\neq x_j(w)$, then we can replace $x_i(v)$, $x_i(w)$, $x_j(v)$, and $x_{j_0}(w)$, by $1-x_i(v)$, $1-x_i(w)$, $1-x_j(v)$, and $1-x_{j_0}(w)$
 to obtain a new sequence with smaller set $I$ which is a contradiction. This implies that 
$\big| \{v\in Z: x_i(v)= 1\} \cup \{v\in V(G)\setminus Z: x_i(v)= 0\}\big|= 1$.
Let $I_0$ be the set of all indices $i\not \in I$ with $| \{v\in Z: x_i(v)= 1\}|=1$ and let $I_1$ be the set of all indices $i\not \in I$ with 
$|\{v\in V(G)\setminus Z: x_i(v)= 0\}|=1$. 
Therefore $|I_1| =\sum_{v\in Z} \sum_{1\le i \le k} x_i(v)= \sum_{v\in Z}[d_G(v)]_k$ and 
$|I_0|=\sum_{v\in V(G)\setminus Z} \sum_{1\le i \le k} (1-x_i(v)) =\sum_{v\in V(G)\setminus Z}[k-d_G(v)]_k$.
Thus $\sum_{v\in Z}[d_G(v)]_k+\sum_{v\in V(G)\setminus Z}[k-d_G(v)]_k=|I_1|+|I_0|=k-|I|=k-2m$.
By Lemma~\ref{lem:unique:Z}, the vertex set $Z$ must be unique.
This completes the proof.
}\end{proof}
%
%
%
%
%
%
%
%
%
\section{Graphs without equitable factorizations}
When a graph does not admit $k$-equitable factorizations, one may ask how may edge-disjoint factors we can find whose degrees are close enough $1/k$ of the corresponding degrees in the main graph. In this section, we answer this question for highly edge-connected graphs (see Corollary~\ref{cor:how-many}).  We shall begin with the following theorem which gives another characterization in terms of vertex sets $Z$ for recognizing highly-edge-connected graphs without equitable factorizations.
\begin{thm}\label{thm:characterization:Z}
{Let $G$ be a $5(k-1)$-tree-connected graph. Then $G$ does not admit a $k$-equitable factorization if and only if 
there is a vertex set $Z$ satisfying the following properties:
\begin{enumerate}{
\item [$\bullet$]
$m = [|E(G)|-\sum_{v\in Z}d_G(v)]_k > 0$.

\item [$\bullet$]
 $\sum_{v\in Z}[d_G(v)]_k+\sum_{v\in V(G)\setminus Z}[k-d_G(v)]_k= k-2m$.
}\end{enumerate}
Moreover, the union of $Z$ and $\{v\in V(G): d_G(v)]_k \stackrel{k}{\equiv} 0\}$ is unique.
}\end{thm}
\begin{proof}
{First assume that $G$ does not admit a $k$-equitable factorization. Then, there is a vertex set $Z$ satisfying Lemma~\ref{lem:Z:equitable:minimum}. 
By Lemma~\ref{lem:m=0}, we must have $m\neq 0$. Thus by Theorem~\ref{thm:new-sufficient-condition}, $\sum_{v\in Z}[d_G(v)]_k+\sum_{v\in V(G)\setminus Z}[k-d_G(v)]_k= k-2m$.  
Note that $Z\cup \{v\in V(G): d_G(v)]_k \stackrel{k}{\equiv} 0\}$ is unique according to Lemma~\ref{lem:unique:Z}.
Conversely, assume that $G$ admit a $k$-equitable factorization. 
Thus there is a two-dimensional sequence $x_i(v)$ satisfying Theorem~\ref{thm:characterization}. 
If there is a vertex set $Z$ including all vertices $v$ with $[d_G(v)]_k=0$ satisfying $m=[|E(G)|-\sum_{v\in Z}d_G(v)]_k> 0$ and 
$\sum_{v\in Z}[d_G(v)]_k+\sum_{v\in V(G)\setminus Z}[k-d_G(v)]_k=k-2m$,
then by the proof of Theorem~\ref{thm:characterization}, one can conclude that 
$|V(G)\setminus Z| \stackrel{2}{\equiv} \sum_{v\in V(G)}\lfloor d_G(v)/k\rfloor$. 
Hence by Lemma~\ref{lem:unique:Z}, we derive a contradiction.
This completes the proof.
}\end{proof}
This vertex $Z$ produces a good structure for finding almost equitable factorizations.
For this purpose, we only need the following lemma that provides a relationship between orientations and equitable factorizations of graphs.
\begin{lem}{\rm (\cite{EquitableFactorizations-2022})}\label{lem:factorization:main:directed:original}
{Every directed graph $G$ can be edge-decomposed into $k$ factors $G_1,\ldots, G_k$
 such that for each $G_i$, $||E(G_i)|-|E(G)|/k|<1$ and for each $v\in V(G_i)$, there is a tripartition $I_0(v), I_1(v), I_2(v)$ of $\{1,\ldots,k\}$
in which for all $i\in I_p(v)$, 
\begin{enumerate}{
\item [$\bullet$] 
 If $[d^+_G(v)]_k \in \{0, [d_G(v)]_k\}$,
then $\lfloor d_G(v)/k \rfloor \le d_{G_i}(v) \le \lceil d_G(v)/k \rceil$.

\item [$\bullet$]
 If $[d^+_G(v)]_k \le [d_G(v)]_k$, then $d_{G_i}(v)=\lfloor d_G(v)/k \rfloor+p$ and $|I_2(v)|= |I_0(v)|+[d_G(v)]_k-k$. 
\item [$\bullet$] 
 If $[d^+_G(v)]_k > [d_G(v)]_k$, then $d_{G_i}(v)=\lfloor d_G(v)/k \rfloor-1+p$ and $|I_2(v)|= |I_0(v)|+[d_G(v)]_k$.

}\end{enumerate}
 In particular, 
$|I_0(v)|\le \min\{k- [d^+_G(v)]_k, k-[d^-_G(v)]_k\}$, 
$|I_1(v)|\ge  | [d^+_G(v)]_k- [d^-_G(v)]_k|$, and
$|I_2(v)|\le \min\{ [d^+_G(v)]_k, [d^-_G(v)]_k\}$.
}\end{lem}
\begin{proof}
{By Theorem 2.1 (i) in \cite{EquitableFactorizations-2022}, the graph $G$ can be edge-decomposed into $k$ factors $G_1,\ldots, G_k$
 such that for each $G_i$, $||E(G_i)|-|E(G)|/k|<1$, and for each $v\in V(G_i)$, 
$|d^+_{G_i}(v)-d^+_{G}(v)/k|<1$ and $|d^-_{G_i}(v)-d^-_{G}(v)/k|<1$.
We define
$I_0(v)=\{i: d^+_{G_i}(v)=\lfloor d^+_{G}(v)/k\rfloor \text{ and } d^-_{G_i}(v)=\lfloor d^-_{G}(v)/k\rfloor \}$ and 
$I_2(v)=\{i: d^+_{G_i}(v)=\lfloor d^+_{G}(v)/k\rfloor+1\text{ and } d^-_{G_i}(v)=\lfloor d^-_{G}(v)/k\rfloor +1 \}$.
Since $\sum_{1\le i\le k}(d^+_{G_i}(v)-\lfloor d^+_{G}(v)/k\rfloor) = [d^+_G(v)]_k$, we must have $|I_0(v)|\le k- [d^+_G(v)]_k$.
Similarly, $|I_0(v)|\le k- [d^-_G(v)]_k$.
Since $\sum_{1\le i\le k}(\lfloor d^+_{G}(v)/k\rfloor+1- d^+_{G_i}(v)) = k-[d^+_G(v)]_k$, we must have $|I_2(v)|\le [d^+_G(v)]_k$.
Similarly, $|I_2(v)|\le [d^-_G(v)]_k$.
Therefore, $|I_1(v)|=k-|I_0(v)|-|I_2(v)|\ge  | [d^+_G(v)]_k- [d^-_G(v)]_k|$.
It is easy to check that 
$d_G(v)-k (\lfloor d^+_{G}(v)/k\rfloor+\lfloor d^-_{G}(v)/k\rfloor)=
\sum_{1\le i\le k}(d^+_{G_i}(v)-\lfloor d^+_{G}(v)/k\rfloor+d^-_{G_i}(v)-\lfloor d^-_{G}(v)/k\rfloor)=
|I_1(v)|+2|I_2(v)|=k+|I_2(v)|-|I_0(v)|$.
If $[d^+_G(v)]_k \le [d_G(v)]_k$, then
$ \lfloor d_{G}(v)/k\rfloor=\lfloor d^+_{G}(v)/k\rfloor+\lfloor d^-_{G}(v)/k\rfloor$ and so $|I_2(v)|= |I_0(v)|+[d_G(v)]_k-k$. 
If $[d^+_G(v)]_k > [d_G(v)]_k$, then 
$\lfloor d_{G}(v)/k\rfloor-1=\lfloor d^+_{G}(v)/k\rfloor+\lfloor d^-_{G}(v)/k\rfloor$ and so $|I_2(v)|= |I_0(v)|+[d_G(v)]_k$.
These complete the proof.
}\end{proof}
Now, we are ready to prove the main result of this section motivated by Theorem~\ref{thm:characterization:Z}. We say below that a graph $G$ is {\it $u$-almost $f$-factor} if for all vertices $v\in V(G)\setminus \{u\}$, $d_{G_i}(v)=f(v)$, and $d_{G_i}(u)=f(u)\pm 1$.
\begin{thm}\label{thm:not-equitable}
{Let $G$ be a $(3k-3)$-edge-connected or $(2k-2)$-tree-connected graph. If there exists a set $Z\subseteq V(G)$
satisfying $m = [|E(G)|-\sum_{v\in Z}d_G(v)]_k > 0$ and $\sum_{v\in Z}[d_G(v)]_k+\sum_{v\in V(G)\setminus Z}[k-d_G(v)]_k= k-2m$, then $G$ can be edge-decomposed into almost $f$-factors $G_1,\ldots, G_k$, where for all vertices $v$,
$$f(v)=\begin{cases}
\lfloor \frac{1}{k}d_G(v) \rfloor,	&\text{when $v\in Z$};\\
\lceil \frac{1}{k}d_G(v)\rceil ,	&\text{when $v\in V(G)\setminus Z$}.
\end {cases}$$
More precisely, $G$ can be edge-decomposed into $k$ factors $G_1,\ldots, G_k$ such that exactly $s(v)$ of them are $v$-almost $f$-factor for all vertices $v$ if and only if for each $v\in Z$, $s(v)-[d_G(v)]_k $ is a nonnegative even integer, for each $v\in V(G)\setminus Z$, $s(v)-[k-d_G(v)]_k $ is a nonnegative even integer, and $\sum_{v\in V(G)}s(v)=k$.
}\end{thm}
\begin{proof}
{First assume that there is a factorization $G_1,\ldots, G_k$ satisfying the theorem.
For all vertices $v$, we set $I_q(v)=\{i:d_{G_i}(v)=f(v)+q\}$, where $q\in \{-1, 1\}$.
If $v\in Z$, then $s(v)=|I_1(v)|+|I_{-1}(v)|\ge |I_1(v)|-|I_{-1}(v)|=\sum_{1\le i\le k} (d_{G_i}(v)-f(v))=[d_G(v)]_k$ which implies that $s(v)\stackrel{2}{\equiv}[d_G(v)]_k$ as well.
If $v\in V(G)\setminus Z$, then $s(v)=|I_1(v)|+|I_{-1}(v)|\ge -|I_1(v)|+|I_{-1}(v)|=\sum_{1\le i\le k} (f(v)-d_{G_i}(v))=[k-d_G(v)]_k$ which implies that $s(v)\stackrel{2}{\equiv}[k-d_G(v)]_k$ as well.

Conversely, assume that for each $v\in Z$, $s(v)-[d_G(v)]_k $ is a nonnegative even integer and for each $v\in V(G)\setminus Z$, $s(v)-[k-d_G(v)]_k $ is a nonnegative even integer.
We may assume that $Z$ contains all vertices with degree divisible by $k$. 
For each $v\in Z$, we define $\eta(v)=\frac{1}{2}(s(v)-[d_G(v)]_k)$, and 
for each $v\in V(G)\setminus Z$, we define $\eta(v)=\frac{1}{2}(s(v)-[k-d_G(v)]_k)$.
By the assumption, $\sum_{v\in V(G)}\eta(v)=m$.
According to Theorem~\ref{thm:modulo:2k-2:3k-3}, the graph $G$ admits an orientation such that
for each $v\in Z$, $d_G^+(v)\stackrel{k}{\equiv}d_G(v)+\eta(v)$ 
and for each $v\in V(G)\setminus Z$, $d_G^+(v)\stackrel{k}{\equiv}\eta(v)$. 
If $v\in Z$, then $[d_G(v)]_k+\eta(v)\le k-2m+m<k$ and so $[d^+_G(v)]_k\ge [d_G(v)]_k$.
If $v\in V(G)\setminus Z$, then $[d_G(v)]_k \ge 2m$ and so $0\le [d^+_G(v)]_k\le m < [d_G(v)]_k$.

By Lemma~\ref{lem:factorization:main:directed:original}, the graph has a factorization $G_1,\ldots, G_k$
such that for each vertex $v$, $d_{G_i}(v)-f(v)\in \{-1,0,1\}$.
In particular, for each $v\in Z$, the number of factors $G_i$ with $d_{G_i}(v)-f(v)\in \{-1,1\}$ is at most $[d_G(v)]_k+2\eta (v)$, and for each $v\in V(G)\setminus Z$, the number of factors $G_i$ with $d_{G_i}(v)-f(v)\in \{-1,1\}$ is at most $[k-d_G(v)]_k+2\eta (v)$. 
On the other hand, by Lemma~\ref{lem:unique:Z}, $\sum_{v\in V(G)}f(v)$ must be odd.
Thus for any factor $G_i$, there is a vertex $v$ with $d_{G_i}(v)-f(v)\in \{-1,1\}$.
 Since $\sum_{v\in Z}([d_G(v)]_k+2\eta (v))+\sum_{v\in V(G)\setminus Z}([k-d_G(v)]_k+2\eta (v))= k$, one can deduce that
for each $v\in Z$, the number of factors $G_i$ with $d_{G_i}(v)-f(v)\in \{-1,1\}$ is precisely $[d_G(v)]_k+2\eta(v)$, for each $v\in V(G)\setminus Z$, the number of factors $G_i$ with $d_{G_i}(v)-f(v)\in \{-1,1\}$ is precisely $[k-d_G(v)]_k+2\eta(v)$, and for any factor $G_i$, there is exactly one vertex $v$ with $d_{G_i}(v)-f(v)\in \{-1,1\}$. Hence the assertion is proved.
}\end{proof}
In the following corollary, we consider a family of graphs satisfying $k-2m=0$ which develops Theorem 2 in \cite{Thomassen-2020} to regular graphs of odd order.
\begin{cor}\label{cor:z1-zk}
{Let $G$ be a $(3k-3)$-edge-connected  graph  with degrees divisible by $k$.
If $\sum_{v\in V(G)}d_G(v)/k$ is odd, then $G$ can be edge-decomposed into factors $G_1,\ldots, G_k$ such that for each $v\in V(G_i)\setminus \{z_i\}$, $d_{G_i}(v)=d_G(v)/k$, and 
$d_{G_i}(z_i)\in \{d_G(v)/k-1,d_G(v)/k+1\}$, where  $z_1,\ldots, z_k\in V(G)$ and  $z_{2i-1}=z_{2i}$ (repetition of vertices is permitted). 
}\end{cor}
\begin{proof}
{Apply Theorem~\ref{thm:not-equitable} with setting $Z=\emptyset$, $m=k/2$, $f(v)=d_G(v)/k$, and $s(v)=|\{i:z_i=v\}|$.
}\end{proof}
\begin{cor}\label{cor:how-many}
{Let $G$ be a $(3k-3)$-edge-connected graph. If there exists a set $Z\subseteq V(G)$
satisfying $m = [|E(G)|-\sum_{v\in Z}d_G(v)]_k >0 $ and $\sum_{v\in Z}[d_G(v)]_k+\sum_{v\in V(G)\setminus Z}[k-d_G(v)]_k= k-2m > 0$, then $G$ contains $k-m$ edge-disjoint factors $G_1,\ldots, G_{k-m}$ satisfying  $|d_{G_i}(v)-d_G(v)/k|<1$ for all $v\in V(G_i)$.
}\end{cor}
\begin{proof}
{Since $k-2m>0$, the graph $G$ contains a vertex $z$ with degree not divisible by $k$. 
For all $v\in Z$, we define $s(v)=[d_G(v)]_k$, and for all  $v\in V(G)\setminus Z$, we define $s(z)=[k-d_G(v)]_k$.
Finally, we refine $s(z)$ to be $s(z)+2m$.
Thus there is a factorization $G_1,\ldots, G_{k}$ of $G$ satisfying Theorem~\ref{thm:not-equitable}.
Choose $u\in V(G)$. 
Let $I^+_u$ (resp. $I^-_u$) be the set of all indices $i$ that $G_i$ is $u$-almost $f$-factor and $d_{G_i}(u)=f(u)+1$
 (resp. $d_{G_i}(u)=f(u)-1$).
If $u\in Z$ and $u\neq z$, then $|I^-_u|-|I^+_u| =\sum_{1\le i\le k}(d_{G_i}(u)-f(u))=[d_G(u)]_k=s(u)= |I^-_u\cup I^+_u|$ 
which implies that $|I^+_u|=\emptyset$ and $|I^-_u|=[d_G(u)]_k$ and so $|d_{G_i}(u)-d_G(u)/k|<1$ for all $i\in I^-_u\cup I^+_u$, 
regardless of $d_G(v)$ is divisible by $k$ or not.
If $u\in V(G)\setminus Z$  and $u\neq z$, then $|I^+_u|-|I^-_u| =\sum_{1\le i\le k}(f(u)- d_{G_i}(u))=[k-d_G(u)]_k=s(u)= |I^+_u\cup I^-_u|$ which implies that $|I^-_u|=\emptyset$ and $|I^+_u|=[k-d_G(u)]_k$ and so $|d_{G_i}(u)-d_G(u)/k|<1$ for all $i\in I^+_u\cup I^-_u$, 
regardless of $d_G(v)$ is divisible by $k$ or not.
If $u=z\in Z$, then similarly we have $|I^+_u|=m$ and $|I^-_u|=m+[d_G(u)]_k$.  In this case, we set $I=I^+_u$.
If $u=z\in V(G)\setminus Z$, then $|I^-_u|=m$ and $|I^-_u|=m+[k-d_G(u)]_k$. In this case, we set $I=I^-_u$.
Since $d_G(z)$ is not divisible by $k$,  it is not difficult to check that for all vertices $v$, $|d_{G_i}(v)-d_G(v)/k|<1$ provided that $i\in \{1,\ldots, k\}\setminus I$. Now, by permuting indices the proof can be completed.
}\end{proof}
When all $z_i$ are equal, Corollary~\ref{cor:z1-zk} is also a consequence of Theorem 3.2 in~\cite{EquitableFactorizations-2022}.
Let us recall this theorem by generalizing it to the following useful version.
\begin{thm}\label{thm:Z0}
{Let $G$ be a graph and let $V_0$ be a proper subset of $ V(G)$.
If $G$ is $V_0$-partially $(3k-3)$-edge-connected or $(2k-2)$-tree-connected, then $G$ can be edge-decomposed into $k$ factors $G_1,\ldots, G_k$ such that for each $G_i$, $||E(G_i)|-|E(G)|/k|<1$, for all $v\in V_0$, $|d_{G_i}(v)- d_G(v)/k|<1$, and for all $v\in V(G_i)\setminus V_0$,
$$\lfloor \frac{1}{k}(d_{G}(v)+1)\rfloor-1 \le d_{G_i}(v) \le \lceil \frac{1}{k}(d_{G}(v)-1)\rceil+1.$$
}\end{thm}
\begin{proof}
{By applying Theorem~\ref{thm:modulo:2k-2:3k-3} to the contracted graph $G/V_1$ with $V_1=V(G)\setminus V_0$, this graph admits an orientation such that 
 the out-degree of the vertex corresponding to $V_1$ is congruent to the size of $G/V_1$ (modulo $k$) and out-degrees of all other vertices are divisible by $k$. By orienting the remaining edges of $G$ arbitrarily, one can induce an orientation for $G$ such that
for all $v\in V_0$, $d_G^+(v)\stackrel{k}{\equiv}0$.
Note that if $d_G(v)\stackrel{k}{\equiv}k-1$, then $[d^+_G(v)]_k \le [d_G(v)]_k$,
and if $d_G(v)\stackrel{k}{\equiv}1$, then $[d^+_G(v)]_k =0$ or $[d^+_G(v)]_k\ge [d_G(v)]_k$.
Now, it is enough to apply Lemma~\ref{lem:factorization:main:directed:original}. 
}\end{proof}
\begin{cor}{\rm (\cite{Liu-2015:factorizations})}
{Every graph $G$ can be edge-decomposed into $k$ factors $G_1,\ldots, G_k$ such for each $v\in V(G_i)$ with $1\le i\le k$, 
$\lfloor \frac{1}{k}(d_{G}(v)+1)\rfloor-1 \le d_{G_i}(v) \le \lceil \frac{1}{k}(d_{G}(v)-1)\rceil+1$.
}\end{cor}
\begin{proof}
{Apply Theorem~\ref{thm:Z0} with $V_0=\emptyset$.
}\end{proof}
\section{Nearly equitable partial parity factorizations with restricted degrees}
\label{sec:nearly-equitable-factorizations}
In the following theorem, we determine the size of $|I_q(v)|$ in Lemma~\ref{lem:factorization:main:directed:original} but by losing the equitable restriction on sizes.
This result is a strengthened version of Theorem 2.1 (ii) in \cite{EquitableFactorizations-2022}.
\begin{thm}\label{thm:factorization:main:directed}
{Every directed graph $G$ can be edge-decomposed into $k$ factors $G_1,\ldots, G_k$
 such that for each $v\in V(G)$, there is a tripartition $I_0(v), I_1(v), I_2(v)$ of $\{1,\ldots,k\}$
in which for all $i\in I_p(v)$, 
\begin{enumerate}{
\item [$\bullet$] 
 If $[d^+_G(v)]_k \in \{0, [d_G(v)]_k\}$,
then $\lfloor d_G(v)/k \rfloor \le d_{G_i}(v) \le \lceil d_G(v)/k \rceil$.

\item [$\bullet$]
 If $[d^+_G(v)]_k \le [d_G(v)]_k$, then $d_{G_i}(v)=\lfloor d_G(v)/k \rfloor+p$. 
\item [$\bullet$] 
 If $[d^+_G(v)]_k > [d_G(v)]_k$, then $d_{G_i}(v)=\lfloor d_G(v)/k \rfloor+p-1$.

}\end{enumerate}
 In particular, 
$|I_0(v)|= \min\{k- [d^+_G(v)]_k, k-[d^-_G(v)]_k\}$, $|I_1(v)|= | [d^+_G(v)]_k- [d^-_G(v)]_k|$,
and $|I_2(v)|= \min\{ [d^+_G(v)]_k, [d^-_G(v)]_k\}$.
}\end{thm}
\begin{proof}
{First we split every vertex $v$ into two vertices $v^+$ and $v^-$ such that
 the incident edges of $v^+$ are directed away from $v$ in $G$ and 
the incident edges of $v^-$ are directed toward $v$ in $G$.
In this construction, every loop in $G$ incident with $v$ is transformed into an edge between $v^+$ and $v^-$.
Next, we split every vertex in the new graph into vertices with degrees divisible by $k$, 
except one vertex with degree less than $k$. We denote these exceptional vertices by $v^+_0$ and $v^-_0$ 
(they may have degree zero). 
Call the resulting (loopless) bipartite graph $G_0$.
Let $G'_0$ be the graph obtained from $G_0$ by inserting $k-\max\{d_{G_0}(v^+_0),d_{G_0}(v^-_0)\}$ new artificial edges 
between $v^+_0$ and $v^-_0$ provided that $d_{G_0}(v^+_0)$ and $d_{G_0}(v^-_0)$ are not zero.
Since $G'_0$ has maximum degree at most $k$,
 it admits a proper edge-coloring with 
$k$ colors $c_1\ldots, c_k$
(to prove this, it is enough to consider a $k$-regular supergraph of it and apply K{\"o}nig's Theorem~\cite{Konig}).
Let $G_i$ be the factor of $G$ consisting of the edges of $G$ corresponding to the edges with color $c_i$ in $G_0$.
Note that for every vertex $v$, at least $\lfloor d^+_G(v)/k\rfloor$ (resp. $\lfloor d^-_G(v)/k\rfloor$) edges with color $c_i$ are incident with the vertex $v^+$ (resp. $v^-$) in $G_0$. 
Moreover, at most $\lceil d^+_G(v)/k\rceil$ (resp. $\lceil d^-_G(v)/k\rceil$ edges with color $c_i$ are incident with the vertex $v^+$ (resp. $v^-$) in $G_0$. 
According to the construction, if 
the degree of $v^+_0$ and $v^-_0$ are not zero in $G_0$, then 
 the degree of $v^+_0$ or $v^-_0$ is $k$ in $G'_0$, and so one of these vertices is incident with all colors.
This implies that there are $\min\{d_{G_0}(v^+_0),d_{G_0}(v^-_0)\}$ 
colors $c_i$ incident with both of $v^+_0$ and $v^-_0$ appearing on edges of $G_i$ (not artificial edges). 
Therefore, $d_{G_i}(v)=\lfloor d^+_G(v)/k\rfloor+ \lfloor d^-_G(v)/k\rfloor+2$ 
when the color $c_i$ is incident with both of $v^+$ and $v^-$ in $G_0$ (we denote by $I_2(v)$ the set of these indices $i$),
and $d_{G_i}(v)=\lfloor d^+_G(v)/k\rfloor+ \lfloor d^-_G(v)/k\rfloor$ 
when the color $c_i$ appears on an artificial edge between $v^+_0$ and $v^-_0$ (we denote by $I_0(v)$ the set of these indices $i$), 
and also $d_{G_i}(v)=\lfloor d^+_G(v)/k\rfloor+ \lfloor d^-_G(v)/k\rfloor+1$ otherwise. Note that 
$\lfloor d^+_G(v)/k\rfloor+ \lfloor d^-_G(v)/k\rfloor=\lfloor (d^+_G(v)+ d^-_G(v))/k\rfloor$ if and only if 
$[d^+_G(v)]_k\le [d_G(v)]_k$.
Thus it is not difficult to check that $G_1,\ldots, G_k$ are the desired factors that we are looking for.
}\end{proof}
\begin{cor}\label{cor:parity-factorizations:odd-k}
{Let $k$ be an odd positive integer. Let $G$ be a $(3k-3)$-edge-connected or $(2k-2)$-tree-connected graph with $V_0\subseteq V(G)$.
If there exists a subset $Z$ of $V(G)\setminus V_0$ satisfying $|E(G)|\stackrel{k}{\equiv}
\sum_{v\in Z}d_G(v)+
\sum_{v\in V_0}\frac{d_G(v)}{2} 
+\frac{k}{2}[\sum_{v\in V_0}d_G(v)]_2$, then $G$ can be edge-decomposed into $k$ factors $G_1,\ldots, G_k$ such that for every $i$ with $1\le i\le k$,
\begin{enumerate}{
\item [$\bullet$] For all $v\in V(G)\setminus V_0$, $|d_{G_i}(v)- d_G(v)/k|<1$.
\item [$\bullet$] For all $v\in V_0$, $|d_{G_i}(v)-d_G(v)/k|<2$, and $ d_{G_i}(v)\stackrel{2}{\equiv} d_G(v)$.
}\end{enumerate}
}\end{cor}
\begin{proof}
{According to Theorem~\ref{thm:modulo:2k-2:3k-3}, the graph $G$ admits an orientation such that
for each $v\in V_0$, $d_G^+(v)\stackrel{k}{\equiv}\frac{d_G(v)}{2}+\frac{k}{2}[d_G(v)]_2\stackrel{k}{\equiv} d_G^-(v)$, for each $v\in Z$, $d_G^+(v)\stackrel{k}{\equiv}d_G(v)$, 
and for each $v\in V(G)\setminus (V_0\cup Z)$, $d_G^+(v)\stackrel{k}{\equiv}0$.
Hence it is enough to apply Theorem~\ref{thm:factorization:main:directed}.
}\end{proof}

\begin{cor}\label{cor:parity-factorizations:even-k}
{Let $k$ be an even positive integer. 
Let $G$ be a $(3k-3)$-edge-connected or $(2k-2)$-tree-connected graph and let $V_0$ be a set of some vertices with even degrees.
Let $f:V_0\rightarrow \mathbb{Z}_2$ be a mapping.
If there exists a subset $Z$ of $V(G)\setminus V_0$ satisfying $|E(G)|\stackrel{k}{\equiv}\sum_{v\in Z}d_G(v)+
\sum_{v\in V_0}\frac{d_G(v)}{2}+\sum_{v\in V_0}\frac{k}{2}f(v)$,
then $G$ can be edge-decomposed into $k$ factors $G_1,\ldots, G_k$ such that for every $i$ with $1\le i\le k$,
\begin{enumerate}{
\item [$\bullet$] For all $v\in V(G)\setminus V_0$, $|d_{G_i}(v)- d_G(v)/k|<1$.
\item [$\bullet$] For all $v\in V_0$, $|d_{G_i}(v)-d_G(v)/k|<2$, and $ d_{G_i}(v)\stackrel{2}{\equiv} f(v)$.
}\end{enumerate}
}\end{cor}
\begin{proof}
{According to Theorem~\ref{thm:modulo:2k-2:3k-3}, the graph $G$ admits an orientation such that
for each $v\in V_0$, $d_G^+(v)\stackrel{k}{\equiv}d_G(v)/2+f(v)k/2$, and for each $v\in Z$, $d_G^+(v)\stackrel{k}{\equiv}d_G(v)$, 
and for each $v\in V(G)\setminus (V_0\cup Z)$, $d_G^+(v)\stackrel{k}{\equiv}0$.
Hence it is enough to apply Theorem~\ref{thm:factorization:main:directed}.
}\end{proof}
\begin{remark}
{Note that Corollaries~\ref{cor:parity-factorizations:odd-k} and~\ref{cor:parity-factorizations:even-k} 
can be developed to graphs having at least $k-1$ vertices with degree not divisible by $k$ in $V(G)\setminus V_0$ using a similar idea of the proofs of Theorems~\ref{thm:new-sufficient-condition} and~\ref{thm:factorization:main:directed}.
We will present details in a forthcoming paper.
}\end{remark}
The following corollary implies Theorem 8 in \cite{Hilton1982} and improves Theorem 3.1 in~\cite{Anstee-1996}.
\begin{cor}\label{cor:Hilton:generalized}
{Every graph $G$ can be edge-decomposed into $k$ factors $G_1,\ldots, G_k$ satisfying the following properties:
\begin{enumerate}{
\item [$\bullet$] For all $v\in V(G_i)$, $| d_{G_i}(v) -d_{G}(v)/k|< 2$.
\item [$\bullet$] For all $v\in V(G_i)$ with $d_G(v)$ even, $d_{G_i}(v)$ is even.
\item [$\bullet$] For each $v\in V(G)$ with $d_G(v)$ odd, there is an index $j_v$ with $d_{G_{j_v}}(v)$ odd such that 
 $d_{G_i}(v)=d_{G_{j_v}}(v)\pm 1$ for all $i\in \{1,\ldots, k\}\setminus \{j_v\}$.
}\end{enumerate}
In addition, $\lfloor d_{G}(v) /k\rfloor \le d_{G_i}(v) \le \lceil d_{G}(v) /k\rceil$ when $d_G(v)\stackrel{2k}{\equiv}\pm 1$ or $d_G(v)\stackrel{2k}{\equiv} 0$.
}\end{cor}
\begin{proof}
{Consider an orientation for $G$ such that for each vertex $v$, 
 $|d^+_G(v)-d^-_G(v)|\le 1$. Let $G_1,\ldots, G_k$ be a factorization of $G$ satisfying Theorem~\ref{thm:factorization:main:directed}. Thus for all $v\in V(G_i)$, $| d_{G_i}(v) -d_{G}(v)/k|< 2$.
Note that if $d_G(v)\stackrel{2k}{\equiv}\pm 1$ or $d_G(v)\stackrel{2k}{\equiv} 0$, then $[d^+_G(v)]_k\in \{0, [d_G(v)]_k\}$ 
and so $| d_{G_i}(v) -d_{G}(v)/k|< 1$.
Assume that $d_G(v)\not \stackrel{2k}{\equiv}\pm 1$ and $d_G(v)\not\stackrel{2k}{\equiv} 0$.
Since $d_G(v)=k\lfloor d_G(v)/k\rfloor+[d_G(v)]_k $, we must have 
$[d^+_G(v)]_k\le [d_G(v)]_k$ if and only if $\lfloor d_G(v)/k\rfloor$ is even.
If $d_G(v)$ is odd, then $|I_1(v)|=1$ and so we can set $j_v$ to be the unique integer in $I_1(v)$. 
In this case, $d_{G_{j_v}}(v)$ must be odd.
If $d_G(v)$ is even, then $I_1(v)=\emptyset$ and so all $d_{G_i}(v)$ must be even.
}\end{proof}
\begin{cor}{\rm (\cite{Kano-1985})}\label{cor:[2a,2b]-factorization}
{Let $a$ and $b$ be two integer with $0\le a \le b$. A graph $G$ satisfies $2ka\le \delta(G) \le \Delta(G)\le 2kb$ if and only if it can be edge-decomposed into $k$ factors $G_1,\ldots, G_k$ 
satisfying $2a\le \delta(G_i) \le \Delta(G_i)\le 2b$.
}\end{cor}
\begin{proof}
{ Let $G_1,\ldots, G_k$ be a factorization of $G$ satisfying Corollary~\ref{cor:Hilton:generalized} so that for all vertices $v$, $2a-1 \le d_{G_i}(v) \le 2b+1$ .
If $d_G(v)$ is even, then $d_{G_i}(v) $ is even which implies that $2a \le d_{G_i}(v) \le 2b$.
If $d_G(v)$ is odd, then $d_{G_i}(v) $ is even for all indices $i$, except exactly one index $j_v$, and again $2a \le d_{G_i}(v) \le 2b$.
Since $2ka \le \sum_{1\le i\le k} d_{G_i}(v) \le 2kb$ and $d_{G_i}(v)=d_{G_{j_v}}(v)\pm 1$, we must also have $ d_{G_{j_v}}(v) \not\in \{2a-1, 2b+1\}$.
}\end{proof}
%
%
%
%
%
%
%
%
\section{Factorizations with bounded minimum or maximum degrees}
\label{sec:lower-upper-bound-factorizations}

Our aim in this section is to present simpler criterion for existence of factorizations with given lower or upper bounds on vertex degrees. 
Moreover, we improve the edge-connectivity needed in Theorems~\ref{intro:thm:delta} and~\ref{intro:thm:Delta} 
when $\delta_i$ and $\Delta_i$ are larger.
%
%
%
%
%
%
%
%
%
\subsection{Factorizations with given lower or upper bounds on degrees}
The following lemma refines Theorem 3.2 in~\cite{EquitableFactorizations-2022} by giving a restriction on the degrees of~$z$.
\begin{lem}\label{lem:equitable-z}
{Let $G$ be a graph with $Z\subseteq V(G)$ and let $z\in V(G)\setminus Z$. 
Let $m=[|E(G)|-\sum_{v\in Z}d_G(v)]_k$.
If $G$ is $(3k-3)$-edge-connected or $(2k-2)$-tree-connected, then it can be edge-decomposed into $k$ factors $G_1,\ldots, G_k$ such that for each $G_i$, $||E(G_i)|-|E(G)|/k|<1$, for all $v\in V(G_i)\setminus \{z\}$, $|d_{G_i}(v)- d_G(v)/k|<1$, and $|d_{G_i}(z)- d_G(z)/k|<2$. In particular, 
$$ d_{G_i}(z)\neq 
\begin{cases} 
\lceil \frac{1}{k} d_G(z)\rceil +1,	&\text{if $m\le  [d_G(z)]_k\neq 0$};\\
\lfloor \frac{1}{k}d_G(z)\rfloor-1,	&\text{if $ m \ge [d_G(z)]_k\neq 0$}.
\end {cases}$$
}\end{lem}
\begin{proof}
{According to Theorem~\ref{thm:modulo:2k-2:3k-3}, the graph $G$ admits an orientation such that
for each $v\in V(G)\setminus (Z\cup \{z\})$, $d_G^+(v)\stackrel{k}{\equiv}0$, 
for each $v\in Z$, $d_G^+(v)\stackrel{k}{\equiv}d_G(v)$, 
and $d_G^+(z)\stackrel{k}{\equiv}m$.
Hence it is enough to apply Lemma~\ref{lem:factorization:main:directed:original}. 
}\end{proof}
The following theorem provides a necessary and sufficient condition for the existence of factorizations with given lower or upper bounds on vertex degrees. 
\begin{thm}\label{thm:factorization:main}
{Let $G$ be a $(3k-3)$-edge-connected or $(2k-2)$-tree-connected graph and let $V_0\subseteq V(G)$.
Then $G$ can be edge-decomposed into $k$ factors $G_1,\ldots, G_k$ such that 
for each $v\in V(G_i)$ with $1\le i\le k$, 
$$d_{G_i}(v) \in\big\{\lfloor \frac{1}{k}d_G(v) \rfloor, \lceil \frac{1}{k} d_G(v)\rceil \big\}\cup \begin{cases}
\{ \lceil \frac{d_G(v)}{k}\rceil +1\},	&\text{when $v\in V_{0}$};\\
\{\lfloor\frac{d_G(v)}{k} \rfloor-1\},	&\text{when $v\in V(G)\setminus V_0$},
\end {cases}$$
 if and only if one of the following conditions holds:
\begin{enumerate}{
\item [$\bullet$] $\sum_{v\in V_0} \lfloor d_G(v)/k\rfloor +\sum_{v\in V(G)\setminus V_0}\lceil d_G(v)/k\rceil$ is even.
\item [$\bullet$] 
$\sum_{v\in V_0}[d_G(v)]_k +\sum_{v\in V(G)\setminus V_0}[k-d_G(v)]_k\ge k-1$.
}\end{enumerate}
In addition, under this assumption we can have $||E(G_i)|-|E(G)|/k|<1$ for all factors $G_i$.
}\end{thm}
\begin{proof}
{For notational simplicity, let us define $V_1=V(G)\setminus V_0$. We first prove the necessity. Let $G_1,\ldots, G_k$ be a factorization of $G$ satisfying the theorem.
If $\sum_{v\in V_0} \lfloor d_G(v)/k\rfloor +\sum_{v\in V_1}\lceil d_G(v)/k\rceil$ is odd, then 
since $\sum_{v\in V(G)}d_{G_i}(v)$ is even, for a vertex $v\in V_0$, 
$d_{G_i}(v) \neq \lfloor d_G(v)/k\rfloor$ or for a vertex $v\in V_1$, $d_{G_i}(v) \neq \lceil d_G(v)/k\rceil$.
According to the assumption, for the first case, we must have $d_{G_i}(v) \ge \lfloor d_G(v)/k\rfloor +1$, and for the second case,
we must have $d_{G_i}(v) \le \lceil d_G(v)/k\rceil-1$.
On the other hand, it is easy to see that for all vertices~$v$,
\begin{equation}\label{eq:essential:identity}
k\lfloor d_{G}(v)/k\rfloor+[d_{G}(v)]_k=
d_{G}(v)=k\lceil d_{G}(v)/k\rceil - [k-d_{G}(v)]_k,
\end{equation}
which implies that
$$\sum_{v\in V_0}[d_G(v)]_k+\sum_{v\in V_1}[k-d_G(v)]_k=
\sum_{1\le i\le k} \Big(\sum_{v\in V_0}(d_{G_i}(v)-\lfloor d_G(v)/k\rfloor)+\sum_{v\in V_1}(\lceil d_G(v)/k\rceil-d_{G_i}(v))\Big)\ge k.$$
Now, we prove the sufficiency. We may assume that $[|E(G)|]_k\neq 0$ according to Lemma~\ref{lem:m=0}
(by setting $Z=\emptyset$).
Let $Z_0$ be a subset of $V_0$ satisfying 
 $\sum_{v\in Z_0}[d_G(v)]_k< [|E(G)|]_k$ with the maximum $\sum_{v\in Z_0}[d_G(v)]_k$.
If there exists a vertex $z\in V_0\setminus Z_0$, then by the maximality property of $Z_0$,
 $[|E(G)|]_k \le \sum_{v\in Z_0}[d_G(v)]_k+[d_G(z)]_k$. 
Thus $[|E(G)| -\sum_{v\in Z_0}d_G(v)]_k=[|E(G)|]_k -\sum_{v\in Z_0}[d_G(v)]_k\le [d_G(z)]_k$.
In this case, the assertion follows from Lemma~\ref{lem:equitable-z} (by setting $Z=Z_0$).
Hence $Z_0=V_0$.
Let $Z_1$ be a subset of $V_1$ satisfying 
 $\sum_{v\in Z_1}[k-d_G(v)]_k< [k-|E(G)|]_k$ with the maximum $\sum_{v\in Z_1}[k-d_G(v)]_k$.
If there exists a vertex $z\in V_1\setminus Z_1$, then by the maximality property of $Z_1$,
 $[k-|E(G)|]_k \le \sum_{v\in Z_1}[k-d_G(v)]_k+[k-d_G(z)]_k$. 
Thus $[\sum_{v\in Z_0}(k-d_G(v))-|E(G)|]_k=[k-|E(G)|]_k -\sum_{v\in Z_0}[k-d_G(v)]_k\le [k-d_G(z)]_k$.
In this case, the assertion follows from Lemma~\ref{lem:equitable-z} (by setting $Z=Z_1$).
Hence $Z_1=V_1$.
Therefore, 
$$\sum_{v\in V_0}[d_G(v)]_k+\sum_{v\in V_1}[k-d_G(v)]_k \le [|E(G)|]_k-1+[k-|E(G)|]_k-1= k-2.$$
Thus by the assumption, $\sum_{v\in V_0} \lfloor d_G(v)/k\rfloor +\sum_{v\in V_1}\lceil d_G(v)/k\rceil$ must be even.
Then, by Equation~(\ref{eq:essential:identity}), one can conclude that
$$|E(G)|=
\sum_{v\in V_0}\frac{d_G(v)}{2}+\sum_{v\in V_1}\frac{d_G(v)}{2}
\stackrel{k}{\equiv}
\frac{1}{2}(\sum_{v\in V_0}[d_G(v)]_k-\sum_{v\in V_1}[k-d_G(v)]_k).$$ 
If $\sum_{v\in V_0}[d_G(v)]_k\ge \sum_{v\in V_1}[k-d_G(v)]_k$, then $[|E(G)|]_k\le \frac{1}{2}\sum_{v\in V_0}[d_G(v)]_k\le 
\sum_{v\in V_0}[d_G(v)]_k<[|E(G)|]_k $, which is a contradiction. Otherwise, $[k-|E(G)|]_k\le \frac{1}{2}\sum_{v\in V_0}[k-d_G(v)]_k\le \sum_{v\in V_0}[k-d_G(v)]_k<[k-|E(G)|]_k $, which is again a contradiction.
Consequently, the proof is completed.
}\end{proof}
When $V_0$ is equal to $V(G)$ or the empty set, the theorem becomes simpler as the following corollary.
\begin{cor}\label{cor:factorization:main}
{Let $G$ be a $(3k-3)$-edge-connected or $(2k-2)$-tree-connected graph.
Then $G$ can be edge-decomposed into $k$ factors $G_1,\ldots, G_k$ such that 
for each $v\in V(G_i)$ with $1\le i\le k$, $d_{G_i}(v)\ge \lfloor \frac{1}{k} d_G(v)\rfloor$, if and only if one of the following conditions holds:
\begin{enumerate}{
\item [$\bullet$] $\sum_{v\in V(G)}\lfloor d_G(v)/k\rfloor$ is even.
\item [$\bullet$] 
$\sum_{v\in V(G)}[d_G(v)]_k\ge k-1$.
}\end{enumerate}
Moreover, $G$ can be edge-decomposed into $k$ factors $G_1,\ldots, G_k$ such that 
for each $v\in V(G_i)$ with $1\le i\le k$, $d_{G_i}(v)\le \lceil \frac{1}{k} d_G(v)\rceil$, if and only if one of the following conditions holds:
\begin{enumerate}{
\item [$\bullet$] $\sum_{v\in V(G)}\lceil d_G(v)/k\rceil$ is even.
\item [$\bullet$] 
$\sum_{v\in V(G)}[k-d_G(v)]_k\ge k-1$.
}\end{enumerate}
}\end{cor}
\begin{proof}
{For proving the necessity, it is enough to repeat the first part of the proof of Theorem~\ref{thm:factorization:main}.
For prove the sufficiency, it is enough to apply Theorem~\ref{thm:factorization:main} with setting $V_0=V(G)$ or $V_0=\emptyset$, respectively.
}\end{proof}
\subsection{Factorizations with bounded minimum degrees}
The following lemma shows an application of factorizations with the same lower bound on minimum degrees.
\begin{lem}\label{lem:factorization:minimum-degree}
{Let $p$ be an odd positive integer and let $\delta_1,\ldots,\delta_m$ be positive integers with $\delta_i\ge p-1$ and $kp=\delta_1+\cdots+ \delta_m$. 
Then $G$ contains a factorization $G_1,\ldots, G_m$ satisfying $\delta(G_i)\ge \delta_i$ for all $i\in \{1,\ldots, m\}$, 
if $G$ contains a factorization $H_1,\ldots, H_k$ satisfying 
 $\delta(H_j)\ge p$ for all $j\in J$, and
$$\sum_{ j\in \{1,\ldots, k\}\setminus J}\delta(H_{j})\ge (k-|J|) p,$$
where $J\subseteq \{1.\ldots, k\}$ and $|J|\ge |\{i: 1\le i\le m \text{ and }\delta_i \text{ is odd}\}|$.
}\end{lem}
\begin{proof}
{Without loss of generality, we may assume that $J= \{i:\delta_i \text{ is odd}\}\subseteq \{1,\ldots, m\}$ 
by deleting some indices from $J$ (if necessary). Let $n= \frac{1}{2}(k-|J|)p$.
Since $p$ is odd, $k \stackrel{2}{\equiv}
 \delta_1+\cdots+ \delta_m
\stackrel{2}{\equiv}\sum_{j\in J}\delta_j \stackrel{2}{\equiv}|J|$ which implies that $n$ is a positive integer. 
By the assumption, 
$\delta(G')\ge \sum_{ j\in \{1,\ldots, k\}\setminus J}\delta(H_{j})\ge 2n$, 
where $G'=\cup_{ j\in \{1,\ldots, k\}\setminus J} H_{j}$.
Thus by Corollary~\ref{cor:[2a,2b]-factorization}, the graph $G'$ 
can be edge-decomposed into factors $F_1,\ldots, F_{n}$ satisfying $\delta(F_i)\ge 2$, where $1\le i\le n$.
Let $\mathcal{J}=\{1,\ldots, m\}\setminus J$.
Notice that $\delta_j-p$ is an even nonnegative integer for all $j\in J$ and $\delta_j$ is an even nonnegative integer for all $j\in \mathcal{J}$.
Since $\sum_{j\in J}(\delta_j-p)+\sum_{j\in \mathcal{J}}\delta_j = kp-|J|p=2n$,
there is another factorization $F'_1,\ldots, F'_m$ of $G'$ such that 
$F'_j$ is the union of $\frac{1}{2}(\delta_j-p)$ factors $F_i$ for all $j\in J$,
and $F'_j$ is the union of $\frac{1}{2}\delta_j$ factors $F_i$ for all $j \in \mathcal{J}$. 
If $j\in J$, we set $G_j=H_j\cup F'_j$ so that 
$\delta(G_j)\ge p+(\delta_j-p)= \delta_j$, and 
if $j\in \mathcal{J}$, we set $G_j=F'_j$ so that $\delta(G_j)\ge \delta_j$.
This completes the proof.
}\end{proof}
A generalization of Theorems~\ref{intro:thm:delta} is given in the following result.
\begin{thm}\label{thm:delta1-deltam}
{Let $k$ be a positive integer, let $p$ be an odd integer with $p\ge 3$, and let $\delta_1,\ldots,\delta_m$ be positive integers with $\delta_i\ge p-1$ and $kp=\delta_1+\cdots+ \delta_m$.
Let $G$ be a $V_0$-partially $(3k-3)$-edge-connected or $(2k-2)$-tree-connected graph satisfying $\delta(G)\ge \delta_1+\cdots+ \delta_m$, where $V_0=\{v\in V(G): d_G(v)\le kp+k-2\}$.
Then $G$ can be edge-decomposed into factors $G_1,\ldots, G_m$ satisfying $\delta(G_i)\ge \delta_i$ for all $i$ with $1\le i \le m$, 
if and only if one of the following conditions holds:
\begin{enumerate}{
\item [$\bullet$] 
$|V(G)|$ is even.
\item [$\bullet$] 
 $\sum_{v\in V(G)}(d_G(v)-(\delta_1+\cdots+ \delta_m))\ge |\{i:\delta_i \text{ is odd}\}|-1$ (for example, $V_0\neq V(G)$).
}\end{enumerate}
}\end{thm}
\begin{proof}
{We first prove the necessity.
Assume that $G_1,\ldots, G_m$ is a factorization of $G$ satisfying $\delta(G_i)\ge \delta_i$ for all $i$ with $1\le i \le m$. If $|V(G)|$ is odd, then $\sum_{v\in V(G_i)}d_{G_i}(v)\ge \delta_i|V(G)|+1$ when $\delta_i$ is odd. 
This implies that $\sum_{v\in V(G)}d_{G}(v)=\sum_{1\le i\le m} \sum_{v\in V(G_i)}d_{G_i}(v)\ge 
(\delta_1+\cdots + \delta_m)|V(G)|+|\{i:\delta_i \text{ is odd}\}|$.
Now, we prove the sufficiency. 
If $V_0\neq V(G)$, then by Theorem~\ref{thm:Z0},
 there is a factorization $H_1,\ldots, H_k$ of $G$ such that 
 for all $v\in V(H_j)\setminus V_0$, $ d_{H_j}(v) \ge \lfloor \frac{1}{k} (d_G(v)+1)\rfloor-1\ge p$, and 
for all $v\in V_0$, $ d_{H_j}(v) \ge \lfloor \frac{1}{k} d_G(v)\rfloor \ge p$
In this case, the assertion follows from Lemma~\ref{lem:factorization:minimum-degree} immediately.
So, suppose $V_0= V(G)$.
If $|V(G)|$ is even or $\sum_{v\in V(G)} [d_G(v)]_k \ge k$, then 
 a combination of Corollary~\ref{cor:factorization:main} and Lemma~\ref{lem:factorization:minimum-degree} can similarly complete the proof.

We may therefore assume that 
for all vertices $v$, $\lfloor d_G(v)/k\rfloor=p$ and $[d_G(v)]_k=d_G(v)-kp$,  $|V(G)|$ is odd, and $\sum_{v\in V(G)} [d_G(v)]_k < k$.
Since $d_{G}(v)=[d_{G}(v)]_k +k\lfloor d_{G}(v)/k\rfloor$, one can conclude that
$\sum_{v\in V(G)}[d_G(v)]_k$ and $k$ have the same parity and
 $|E(G)|=\frac{1}{2}d_G(v)\stackrel{k}{\equiv }\frac{1}{2}(k+\sum_{v\in V(G)}[d_G(v)]_k)$.
Pick $z\in V(G)$.
Let $n=\frac{1}{2}(k-\sum_{v\in V(G)}[d_G(v)]_k)\le \frac{1}{2}(k-[d_G(z)]_k)$ so that $n>0$. 
According to Theorem~\ref{thm:modulo:2k-2:3k-3}, the graph $G$ admits an orientation such that for each
 $v\in V(G)\setminus \{z\}$, $d^+_G(v)\stackrel{k}{\equiv}d_G(v)$, 
and $d^+_G(z)\stackrel{k}{\equiv}d_G(z)+n$.
Thus by Theorem~\ref{thm:factorization:main:directed},
 there is a factorization $H_1,\ldots, H_k$ of $G$ such that 
for all $v\in V(H_j)\setminus \{z\}$, $ d_{H_j}(v) \ge \lfloor \frac{1}{k} d_G(v)\rfloor \ge p$.
In addition, there is a tripartition $J_0,J_1,J_2$ of $\{1,\ldots, k\}$ such that
$ d_{H_j}(z) \ge \lfloor \frac{1}{k} d_G(z)\rfloor-1\ge p-1$ when $j\in J_0$,
 $ d_{H_j}(z) \ge \lfloor \frac{1}{k} d_G(z)\rfloor\ge p$ when $j\in J_1$,
and $ d_{H_j}(z) \ge \lfloor \frac{1}{k} d_G(z)\rfloor+1\ge p+1$ when $j\in J_2$. 
 In particular, $|J_0|\le 
k-[d^-_G(z)]_k=k-[k-n]_k
= n$ and $|J_0| \le |J_2|$. Note that $[d^+_G(z)]_k=[d_G(z)]_k+n> [d_G(z)]_k$.
Let $J'_2$ be a subset of $J_2$ with $|J'_2|=|J_0|$.
Obviously, $\sum_{j\in J_0\cup J'_2}\delta(H_j)\ge (p-1)|J_0|+(p+1)|J'_2|= p|J_0\cup J'_2|= (k-|J|)p$, 
where $J=\{1,\ldots, k\}\setminus (J_0\cup J'_2)$.
Since $\sum_{v\in V(G)} (d_G(v)-kp)\stackrel{2}{\equiv}\sum_{1\le i\le k}\delta_i \stackrel{2}{\equiv}|\{i:\delta_i \text{ is odd}\}|$, 
by the assumption, we must have
$|J|\ge k-2n= \sum_{v\in V(G)}[d_G(v)]_k=\sum_{v\in V(G)} (d_G(v)-kp)\ge |\{i:\delta_i \text{ is odd}\}|$.
Therefore, by applying Lemma~\ref{lem:factorization:minimum-degree}, the proof can be completed.
}\end{proof}
\subsection{Factorizations with bounded maximum degrees}
The following lemma shows an application of factorizations with the same upper bound on maximum degrees.
\begin{lem}\label{lem:factorization:maximum-degree}
{Let $p$ be an odd positive integer and let $\Delta_1,\ldots,\Delta_m$ be positive integers with $\Delta_i\ge p-1$ and $kp=\Delta_1+\cdots+ \Delta_m$. 
Then $G$ contains a factorization $G_1,\ldots, G_m$ satisfying $\Delta(G_i)\le \Delta_i$ for all $i\in \{1,\ldots, m\}$, 
if $G$ contains a factorization $H_1,\ldots, H_k$ satisfying 
 $\Delta(H_j)\le p$ for all $j\in J$, and
$$\sum_{ j\in \{1,\ldots, k\}\setminus J}\Delta(H_{j})\le (k-|J|) p,$$
where $J\subseteq \{1.\ldots, k\}$ and $|J|\ge |\{i: 1\le i\le m \text{ and }\Delta_i \text{ is odd}\}|$.
}\end{lem}
\begin{proof}
{Without loss of generality, we may assume that $J= \{i:\Delta_i \text{ is odd}\}\subseteq \{1,\ldots, m\}$ 
by deleting some indices from $J$ (if necessary). Let $n= \frac{1}{2}(k-|J|)p$.
Since $p$ is odd, $k \stackrel{2}{\equiv}
 \Delta_1+\cdots+ \Delta_m
\stackrel{2}{\equiv}\sum_{j\in J}\Delta_j \stackrel{2}{\equiv}|J|$ which implies that $n$ is a positive integer. 
By the assumption, 
$\Delta(G')\le \sum_{ j\in \{1,\ldots, k\}\setminus J}\Delta(H_{j})\le 2n$, 
where $G'=\cup_{ j\in \{1,\ldots, k\}\setminus J} H_{j}$.
Thus by Corollary~\ref{cor:[2a,2b]-factorization}, the graph $G'$ 
can be edge-decomposed into factors $F_1,\ldots, F_{n}$ satisfying $\Delta(F_i)\le 2$, where $1\le i\le n$.
Let $\mathcal{J}=\{1,\ldots, m\}\setminus J$.
Notice that $\Delta_j-p$ is an even nonnegative integer for all $j\in J$ and $\Delta_j$ is an even nonnegative integer for all $j\in \mathcal{J}$.
Since $\sum_{j\in J}(\Delta_j-p)+\sum_{j\in \mathcal{J}}\Delta_j = kp-|J|p=2n$,
there is another factorization $F'_1,\ldots, F'_m$ of $G'$ such that 
$F'_j$ is the union of $\frac{1}{2}(\Delta_j-p)$ factors $F_i$ for all $j\in J$,
and $F'_j$ is the union of $\frac{1}{2}\Delta_j$ factors $F_i$ for all $j \in \mathcal{J}$. 
If $j\in J$, we set $G_j=H_j\cup F'_j$ so that 
$\Delta(G_j)\le p+(\Delta_j-p)= \Delta_j$, and 
if $j\in \mathcal{J}$, we set $G_j=F'_j$ so that $\Delta(G_j)\le \Delta_j$.
This completes the proof.
}\end{proof}
A generalization of Theorems~\ref{intro:thm:Delta} is given in the following result.
\begin{thm}\label{thm:Delta1-Deltam}
{Let $k$ be a positive integer, let $p$ be an odd integer with $p\ge 3$, and let $\Delta_1,\ldots,\Delta_m$ be positive integers with $\Delta_i\ge p-1$ and $kp=\Delta_1+\cdots+ \Delta_m$.
Let $G$ be a $V_0$-partially $(3k-3)$-edge-connected or $(2k-2)$-tree-connected graph satisfying $\Delta(G)\le \Delta_1+\cdots+ \Delta_m$, where $V_0=\{v\in V(G): d_G(v)\ge kp-k+2\}$.
Then $G$ can be edge-decomposed into factors $G_1,\ldots, G_m$ satisfying $\Delta(G_i)\le \Delta_i$ for all $i$ with $1\le i \le m$, 
if and only if one of the following conditions holds:
\begin{enumerate}{
\item [$\bullet$] 
$|V(G)|$ is even.
\item [$\bullet$] 
 $\sum_{v\in V(G)}(\Delta_1+\cdots+ \Delta_m-d_G(v))\ge |\{i:\Delta_i \text{ is odd}\}|-1$ (for example, $V_0\neq V(G)$).
}\end{enumerate}
}\end{thm}
\begin{proof}
{We first prove the necessity.
Assume that $G_1,\ldots, G_m$ is a factorization of $G$ satisfying $\Delta(G_i)\le \Delta_i$ for all $i$ with $1\le i \le m$.
 If $|V(G)|$ is odd, then $\sum_{v\in V(G_i)}d_{G_i}(v)\le \Delta_i|V(G)|-1$ when $\Delta_i$ is odd. 
This implies that $\sum_{v\in V(G)}d_{G}(v)=\sum_{1\le i\le m} \sum_{v\in V(G_i)}d_{G_i}(v)\le 
(\Delta_1+\cdots + \Delta_m)|V(G)|-|\{i:\Delta_i \text{ is odd}\}|$.
Now, we prove the sufficiency. 
If $V_0\neq V(G)$, then by Theorem~\ref{thm:Z0},
 there is a factorization $H_1,\ldots, H_k$ of $G$ such that 
 for all $v\in V(H_j)\setminus Z_0$, $ d_{H_j}(v) \le \lceil \frac{1}{k} (d_G(v)-1)\rceil+1\le p$, and 
for all $v\in V_0$, $ d_{H_j}(v) \le \lceil \frac{1}{k} d_G(v)\rceil \le p$.
In this case, the assertion follows from Lemma~\ref{lem:factorization:maximum-degree} immediately.
So, suppose $V_0= V(G)$.
If $|V(G)|$ is even or $\sum_{v\in V(G)} [k-d_G(v)]_k \ge k$, then 
 a combination of Corollary~\ref{cor:factorization:main} and Lemma~\ref{lem:factorization:maximum-degree} can similarly complete the proof.

We may therefore assume that 
for all vertices $v$, $\lceil d_G(v)/k\rceil=p$ and $[k-d_G(v)]_k=kp-d_G(v)$, $|V(G)|$ is odd, and
 $\sum_{v\in V(G)} [k-d_G(v)]_k <k$.
Since $[k-d_{G}(v)]_k=k\lceil d_{G}(v)/k\rceil - d_{G}(v)$, one can conclude that
$\sum_{v\in V(G)}[k-d_G(v)]_k$ and $k$ have the same parity and
 $|E(G)|=\frac{1}{2}d_G(v)\stackrel{k}{\equiv }\frac{1}{2}(k-\sum_{v\in V(G)}[k-d_G(v)]_k)$.
Pick $z\in V(G)$.
Let $n=\frac{1}{2}(k-\sum_{v\in V(G)}[k-d_G(v)]_k)\le \frac{1}{2}(k-[k-d_G(z)]_k)$ so that $n>0$.
According to Theorem~\ref{thm:modulo:2k-2:3k-3}, the graph $G$ admits an orientation such that for each
 $v\in V(G)\setminus \{z\}$, $d^+_G(v)\stackrel{k}{\equiv}d_G(v)$, 
and $d^+_G(z)\stackrel{k}{\equiv}d_G(z)-n$.
Thus by Theorem~\ref{thm:factorization:main:directed},
 there is a factorization $H_1,\ldots, H_k$ of $G$ such that 
for each $v\in V(H_j)\setminus \{z\}$ with $1\le j\le k$, $ d_{H_j}(v) \le \lceil \frac{1}{k} d_G(v)\rceil \le p$.
In addition, there is a tripartition $J_0,J_1,J_2$ of $\{1,\ldots, k\}$ such that
$ d_{H_j}(z) \le \lceil \frac{1}{k} d_G(z)\rceil-1\le p-1$ when $j\in J_0$,
 $ d_{H_j}(z) \le \lceil \frac{1}{k} d_G(z)\rceil\le p$ when $j\in J_1$,
and $ d_{H_j}(z) \le \lceil \frac{1}{k} d_G(z)\rceil+1\le p+1$ when $j\in J_2$. 
 In particular, $|J_2|\le 
[d^-_G(z)]_k= n$ and $ |J_2| \le |J_0| $. Note that $[d^+_G(z)]_k\le [d_G(z)]_k$ or $[d_G(z)]_k=0$.
Let $J'_0$ be a subset of $J_0$ with $|J'_0|=|J_0|$.
Obviously, $\sum_{j\in J'_0\cup J_2}\Delta(H_j)\le (p-1)|J'_0|+(p+1)|J_2|= p|J'_0\cup J_2|= (k-|J|)p$, 
where $J=\{1,\ldots, k\}\setminus (J'_0\cup J_2)$.
Since $\sum_{v\in V(G)} (kp-d_G(v))\stackrel{2}{\equiv}\sum_{1\le i\le k}\Delta_i \stackrel{2}{\equiv}|\{i:\Delta_i \text{ is odd}\}|$, 
by the assumption, we must have $|J|\ge k-2n= \sum_{v\in V(G)}[k-d_G(v)]_k= \sum_{v\in V(G)}(d_G(v)-kp)\ge |\{i:\Delta_i \text{ is odd}\}$.
Therefore, by applying Lemma~\ref{lem:factorization:maximum-degree}, the proof can be completed.
}\end{proof}
\subsection{Graphs of even order: odd-edge-connected graphs}
Our aim in this subsection to improve the edge-connectivity need in Theorems~\ref{thm:delta1-deltam} and~\ref{thm:Delta1-Deltam} for graphs of even order which is motivated by Theorem 5 in~\cite{Thomassen-1980} and Corollary 5.4 in~\cite{EquitableFactorizations-2022}. Our proof is based on the following refined version of Theorem~\ref{thm:modulo:2k-2:3k-3} for odd and even $k$. 
\begin{lem}{\rm (\cite{Lovasz-Thomassen-Wu-Zhang-2013})}\label{lem:balanced-modulok:odd}
{Let $k$ be an odd positive integer. 
If $G$ is an odd-$(3k-3)$-edge-connected graph, 
then it admits an orientation such that for each vertex $v$, $d_G^+(v)\stackrel{k}{\equiv}d_G(v)/2$ or 
$d_G^+(v)\stackrel{k}{\equiv}d_G(v)/2+k/2$.
}\end{lem}

\begin{lem}{\rm (\cite{ModuloBounded})}\label{lem:balanced-modulok:even}
{Let $k$ be an even positive integer, let $G$ be an Eulerian graph, and let $Q\subseteq V(G)$ with $|Q|$ even.
If $d_G(X)\ge 3k-3$ for every $X\subseteq V(G)$ with $|X\cap Q|$ odd,
then $G$ admits an orientation such that for each $v\in V(G)\setminus Q$, 
$d_G^+(v)\stackrel{k}{\equiv}d_G(v)/2$, and 
for each $v\in Q$, $d_G^+(v) \stackrel{k}{\equiv}d_G(v)/2+k/2$.
}\end{lem}
These two lemmas are also helpful to refine the edge-connectivity needed in Theorem~\ref{thm:factorization:main} as the following version.
\begin{thm}\label{thm:factorization:main:odd-edge-connectivity}
{Let $G$ be a graph and let $V_0$ and $V_1$ be two disjoint subsets of $V(G)$ excluding vertices $v$ satisfying
 $d_G(v)\not \stackrel{k}{\equiv}k-1 $ or $d_G(v)\not \stackrel{k}{\equiv}1 $, respectively.
 Assume that $V_0\cup V_1\neq V(G)$ or $|Q|$ is even, where 
$$Q=\{v\in V_0: \lfloor d_{G}(v)/k \rfloor \text{ is odd}\}\cup
 \{v\in V_1: \lceil d_{G}(v)/k \rceil \text{ is odd}\}.$$ 
If $G$ is $(V_0\cup V_1)$-partially odd-$(3k-3, Q)$-edge-connected,
then $G$ can be edge-decomposed into $k$ factors $G_1,\ldots, G_k$ such that 
for each graph $G_i$, $||E(G_i)|-|E(G)|/k|<1$, and for all vertices $v$, $|d_{G_i}(v)-d_G(v)/k|<2$. 
In particular,
$$d_{G_i}(v) \in\big\{\lfloor \frac{1}{k}d_G(v) \rfloor, \lceil \frac{1}{k} d_G(v)\rceil \big\}\cup \begin{cases}
\{ \lceil \frac{d_G(v)}{k}\rceil +1\},	&\text{when $v\in V_{0}$ or $d_G(v) \stackrel{k}{\equiv}k-1$};\\
\{\lfloor\frac{d_G(v)}{k} \rfloor-1\},	&\text{when $v\in V_1$ or $d_G(v) \stackrel{k}{\equiv}1$}.
\end {cases}$$
}\end{thm}
\begin{proof}
{Let $Z_0=V(G)\setminus (V_0\cup V_1)$. 
First, assume that $k$ is odd. Let 
$S_0=\{v\in V_0: d_G(v)\not \stackrel{2}{\equiv} \lfloor d_G(v)/k\rfloor \}$ and
 $S_1=\{v\in V_1:  d_G(v)\not \stackrel{2}{\equiv} \lceil d_{G}(v)/k \rceil\}$.
If $|S_0\cup S_1|$ is odd, we set $S=S_0\cup S_1\cup \{Z_0\}$. 
Since $\sum_{v\in V(G)}d_G(v)$ is even, by the assumption, if $Z_0= \emptyset$, then $Q$ is even and so is $|S_0\cup S_1|$.
If $|S_0\cup S_1|$ is even, we set $S=S_0\cup S_1$.
Consequently, for both cases, $|S|$ is even.
Let $G_0$ be a graph obtained from $G$ by inserting a new matching $M$ paring vertices of $S$. 
According to the construction, 
$ V_0=\{v\in V_0\setminus S_0: d_{G_0}(v)=d_G(v) \stackrel{2}{\equiv} \lfloor d_G(v)/k\rfloor= \lfloor d_{G_0}(v)/k\rfloor\}\cup
 \{v\in S_0: d_{G_0}(v)=d_G(v)+1 \stackrel{2}{\equiv} \lfloor d_G(v)/k\rfloor =\lfloor d_{G_0}(v)/k\rfloor\}=
\{v\in V_0: d_{G_0}(v)\stackrel{2}{\equiv} \lfloor d_{G_0}(v)/k\rfloor\}$.
In addition, 
$V_1=\{v\in V_1\setminus S_1: d_{G_0}(v)=d_G(v) \stackrel{2}{\equiv} \lceil d_G(v)/k\rceil= \lceil d_{G_0}(v)/k\rceil\}\cup
 \{v\in S_1: d_{G_0}(v)=d_G(v)+1 \stackrel{2}{\equiv} \lceil d_G(v)/k\rceil= \lceil d_{G_0}(v)/k\rceil\}\}=
\{v\in V_1: d_{G_0}(v)\stackrel{2}{\equiv} \lceil d_{G_0}(v)/k\rceil\}$.
Thus $Q=\{v\in V_0\cup V_1: d_{G_0}(v) \text{ is odd}\}$.
This implies that for every subset $X\subseteq V_0\cup V_1$, $|X\cap Q|$ is odd if and only if $d_{G_0}(X)$ is odd, and so $G_0/Z_0$ is odd-$(3k-3)$-edge-connected.
Thus by applying to Lemma~\ref{lem:balanced-modulok:odd} to the contracted graph $G_0/Z_0$, similar to the proof of Theorem~\ref{thm:Z0}, one can conclude that $G_0$ admits an orientation such that for each vertex $v$, $d_{G_0}^+(v)\stackrel{k}{\equiv}d_{G_0}(v)/2$ or 
$d_{G_0}^+(v)\stackrel{k}{\equiv}d_{G_0}(v)/2+k/2$. 
This implies that 
$d_{G_0}^+(v)\stackrel{k}{\equiv}\frac{k}{2}[d_{G_0}(v)+\lfloor \frac{d_{G_0(v)}}{k}\rfloor ]_2+\frac{1}{2}[d_{G_0}(v)]_k$.
If $v\in V_0$, then $d^+_{G_0}(v)\stackrel{k}{\equiv}\frac{1}{2}[d_{G_0}(v)]_k\le [d_{G_0}(v)]_k$ regardless of $v \in Q$ or not.
If $v\in V_1$, then $d^+_{G_0}(v)\stackrel{k}{\equiv}\frac{k}{2}+\frac{1}{2}[d_{G_0}(v)]_k\ge 
 \frac{1}{2}[d_{G_0}(v)]_k$ regardless of $v \in Q$ or not.
We consider the orientation of $G$ obtained from this orientation.
Obviously, for each $v\in V(G)\setminus S$, $d^+_G(v)=d^+_{G_0}(v)$ and $d_G(v)=d_{G_0}(v)$, and 
for each $v\in S$, $d^+_G(v)\in \{d^+_{G_0}(v), d^+_{G_0}(v)- 1\}$ and $d_G(v)=d_{G_0}(v)-1$.
Consequently, if $v\in V_0$ or $d_G(v) \stackrel{k}{\equiv}k-1$, then $[d^+_{G}(v)]_k\le [d_{G}(v)]_k$.
In addition, if $v\in V_1$ or $d_G(v) \stackrel{k}{\equiv}1$, then 
$[d^+_G(v)]_k=0$ or $[d^+_G(v)]_k\ge [d_G(v)]_k$.
Hence the assertion follows from Lemma~\ref{lem:factorization:main:directed:original}.

Now, assume that $k$ is even. 
Let $G_0$ be an even graph obtained from $G$ by inserting a new matching $M$ paring odd vertices of $G$. 
If $|Q|$ is odd, we set $Q_0=Q\cup \{Z_0\}$. 
By the assumption, if $Z_0= \emptyset$, then $Q$ is even.
If $|Q|$ is even, we set $Q_0=Q$.
Consequently, for both cases, $|Q_0|$ is even.
Thus by applying Lemma~\ref{lem:balanced-modulok:even} to the contracted graph $G_0/Z_0$ (with the vertex set $Q_0$), one can conclude that the graph $G_0$ has an orientation such that for each $v\in Q$, $d_{G_0}^+(v) \stackrel{k}{\equiv}d_{G_0}(v)/2+k/2$ and
for each $v\in V(G)\setminus Q$, 
$d_{G_0}^+(v)\stackrel{k}{\equiv}d_G(v)/2$. 
If $v\in V_0$, then $d^+_{G_0}(v)\stackrel{k}{\equiv}\frac{1}{2}[d_{G_0}(v)]_k$ regardless of $v\in Q$ or not.
If $v\in V_1$, then $[d^+_{G_0}(v)]_k=0$ or $d^+_{G_0}(v)\stackrel{k}{\equiv}\frac{k}{2}+\frac{1}{2}[d_{G_0}(v)]_k> k/2$ regardless of $v \in Q$ or not.
We consider the orientation of $G$ obtained from this orientation.
Obviously, for each vertex $v$ with $d_G(v)$ even, $d^+_G(v)=d^+_{G_0}(v)$ and $d_G(v)=d_{G_0}(v)$, and 
for each vertex $v$ with $d_G(v)$ odd, $d^+_G(v)\in \{d^+_{G_0}(v), d^+_{G_0}(v)- 1\}$ and $d_G(v)=d_{G_0}(v)-1$.
Consequently, if $v\in V_0$ or $d_G(v) \stackrel{k}{\equiv}k-1$, then $[d^+_{G}(v)]_k\le [d_{G}(v)]_k$.
In addition, if $v\in V_1$ or $d_G(v) \stackrel{k}{\equiv}1$, then 
$[d^+_G(v)]_k=0$ or $[d^+_G(v)]_k\ge [d_G(v)]_k$.
Hence the assertion follows from Lemma~\ref{lem:factorization:main:directed:original}.
}\end{proof}
The following corollary slightly improves the edge-connectivity needed in Theorem 5 in~\cite{Thomassen-1980} and generalizes it to graphs of even order.
\begin{cor}
{Let $k$ be a positive integer, let $p$ be an odd integer with $p\ge 3$, and let $\delta_1,\ldots,\delta_m$ be positive integers with $\delta_i\ge p-1$ and $kp=\delta_1+\cdots+ \delta_m$.
Let $G$ be a graph of even order satisfying $\delta(G)\ge \delta_1+\cdots+ \delta_m$ and let $V_0=\{v\in V(G): d_G(v)\le kp+k-2\}$.
If $G$ is $V_0$-partially odd-$(3k-3, V_0)$-edge-connected, then $G$ can be edge-decomposed into factors $G_1,\ldots, G_m$ satisfying $\delta(G_i)\ge \delta_i$ for all $i$ with $1\le i \le m$.
}\end{cor}
\begin{proof}
{By applying Theorem~\ref{thm:factorization:main} with setting $V_1=\emptyset$ and $Q=V_0$, the graph $G$ admits a factorization $H_1,\ldots, H_k$ satisfying $\delta(H_i)\ge p$. Now, it is enough to apply Lemma~\ref{lem:factorization:minimum-degree}. 
}\end{proof}
The following corollary slightly improves the edge-connectivity needed in Corollary 5.4 in~\cite{EquitableFactorizations-2022}.
\begin{cor}
{Let $k$ be a positive integer, let $p$ be an odd integer with $p\ge 3$, and let $\Delta_1,\ldots,\Delta_m$ be positive integers with $\Delta_i\ge p-1$ and $kp=\Delta_1+\cdots+ \Delta_m$.
Let $G$ be a graph of even order satisfying $\Delta(G)\le \Delta_1+\cdots+ \Delta_m$ and let $V_1=\{v\in V(G): d_G(v)\ge kp-k+2\}$.
If $G$ is $V_1$-partially odd-$(3k-3, V_1)$-edge-connected, then $G$ can be edge-decomposed into factors $G_1,\ldots, G_m$ satisfying $\Delta(G_i)\le \Delta_i$ for all $i$ with $1\le i \le m$.
}\end{cor}
\begin{proof}
{By applying Theorem~\ref{thm:factorization:main} with setting $V_0=\emptyset$ and $Q=V_1$, the graph $G$ admits a factorization $H_1,\ldots, H_k$ satisfying $\Delta(H_i)\le p$. Now, it is enough to apply Lemma~\ref{lem:factorization:maximum-degree}. 
}\end{proof}
%
%
%
%
%
%
%
\section{A sufficient edge-connectivity condition for the existence of a partial parity factor}
\label{sec:epsilon-parity-factor}
In this section, we completely confirm Conjecture~\ref{intro:conj:new} for the special case $k=2$ (see Corollary~\ref{cor:even-factor:atleast2/3}). In addition, we improve the edge-connectivity needed in Theorems~\ref{intro:thm:delta} and \ref{intro:thm:Delta} for the special case $m=2$. Our proof is based on the following lemma due to Kano and Matsuda (2001) which is a generalization of two well-known results due to Lov{\'a}sz (1970, 1972) \cite{Lovasz-1970, Lovasz-1972} for the existence of $(g,f)$-factors and parity $(g,f)$-factors.
\begin{lem}{\rm (\cite{Kano-Matsuda-2001})}\label{lem:Kano-Matsuda-2001}
{Let $G$ be a connected graph with $V_0\subseteq V(G)$ and let $g$ and $f$ be two integer-valued functions on $V(G)$ with $g\le f$ satisfying $g(v)\stackrel{2}{\equiv}f(v)$ for all $v\in V_0$ and $g(v)<f(v)$ for all $v\in V(G)\setminus V_0$. 
Then $G$ has a $V_0$-partial parity $(g,f)$-factor if and only if 
for all disjoint subsets $A$ and $B$ of $V(G)$,
$$\omega_{f, V_0}(G, A,B)<1+ \sum_{v\in A} f(v)+\sum_{v\in B} (d_{G}(v)-g(v))-d_G(A,B),$$
where $\omega_{f, V_0}(G, A,B)$ denotes the number of components $G[X]$ 
of $G\setminus (A\cup B)$ 
satisfying $X\subseteq V_0$ and $\sum_{v\in X}f(v)\stackrel{2}{\not\equiv}d_G(X,B)$ 
(except those sets that $G[X]$ is bipartite and $g(v)=f(v)$ for all $v\in X$).
}\end{lem}
The following theorem provides a partial parity version for Theorem 6.2 in \cite{EquitableFactorizations-2022} on partially edge-connected graphs. Note that some edge-connected versions of this theorem were formerly studied in~\cite{Anstee-Nam, Egawa-Kano, Kano-1985}.
\begin{thm}\label{thm:parity:version}
{Let $G$ be a connected graph, let $V_0, V_1, V'_1$ be three disjoint vertex subsets of $V(G)$, 
let $f:V_0\rightarrow \mathbb{Z}_2$ be a mapping, and let $\varepsilon$ be a real number with
 $0< \varepsilon < 1$. Assume that
$ \sum_{v\in V_0}f(v)\stackrel{2}{\equiv}0$ or $V_0\neq V(G)$.
If for every nonempty proper subset $X$ of $V_0$ (except those sets that $G[X]$ is disconnected or bipartite and $\varepsilon d_G(v)$ is an integer with the same parity of $f(v)$ for all $v\in X$), 
 $$d_G(X)\ge \begin{cases}
1/\varepsilon,	&\text{when $\sum_{v\in X}f(v)$ is odd};\\
1/(1-\varepsilon),	&\text{when $\sum_{v\in X}(d_G(v)-f(v))$ is odd},
\end {cases}$$
then $G$ has a $V_0$-partial $f$-parity factor $F$ with the following properties:
\begin{enumerate}{

\item [$\bullet$] 
For all $v\in V_0$, 
$ \lfloor \varepsilon d_G(v)\rfloor -1 \le d_{F}(v) \le \lceil \varepsilon d_G(v)\rceil +1$ and $d_F(v)\stackrel{2}{\equiv}f(v)$.

\item [$\bullet$] 
For all $v\in V_1$, 
$ \lceil \varepsilon d_G(v)\rceil -1\le d_{F}(v) \le \lceil \varepsilon d_G(v)\rceil$, and
for all $v\in V'_1$, 
$\lfloor \varepsilon d_G(v)\rfloor \le d_{F}(v) \le \lfloor \varepsilon d_G(v)\rfloor +1$.
}\end{enumerate}
Furthermore, for an arbitrary given vertex $z\in V_0$, 
we can have $d_{F}(z) \ge \lceil \varepsilon d_G(z)\rceil-1$
or $d_{F}(z) \le \lfloor \varepsilon d_G(z)\rfloor +1$.
}\end{thm}
\begin{proof}{
For each $v\in V_0$, let us define $g_0(v)\in
 \{\lfloor \varepsilon d_G(v)\rfloor-1,
 \lfloor \varepsilon d_G(v)\rfloor\}$ 
and $f_0(v)\in \{\lceil \varepsilon d_G(v)\rceil,\lceil \varepsilon d_G(v)\rceil+1 \}$
 such that 
$g_0(v)\stackrel{2}{\equiv}f_0(v)\stackrel{2}{\equiv} f(v)$.
For the vertex $z$ in which $\varepsilon d_G(z)$ is not integer, we can 
arbitrarily have $g_0(z)\in
 \{ \lfloor \varepsilon d_G(z)\rfloor,\lfloor \varepsilon d_G(z)\rfloor+1\}$ 
or $f_0(z)\in \{\lceil \varepsilon d_G(z)\rceil-1,\lceil \varepsilon d_G(z)\rceil\}$.
For each vertex $v\in V_1$, we define $g_0(v)=\lceil \varepsilon d_G(v)\rceil -1$
and $f_0(v)=\lceil \varepsilon d_G(v)\rceil$,
and for each vertex $v\in V'_1$, we define $g_0(v)=\lfloor \varepsilon d_G(v)\rfloor$
and $f_0(v)=\lfloor \varepsilon d_G(v)\rfloor+1$.
Note that for each $v\in V_1\cup V'_1$, we have $g_0(v)< f_0(v)$.
Let $A$ and $B$ be two disjoint vertex subsets of $V(G)$ with $A\cup B\neq \emptyset$. 
By the definition of $g_0$ and $f_0$, we must have 
\begin{equation}\label{eq:1:thm:epsilon:parity}
{\sum_{v\in A} \varepsilon d_G(v)+\sum_{v\in B} (1-\varepsilon) d_G(v)
< 1+ \sum_{v\in A} f_0(v)+\sum_{v\in B} (d_{G}(v)-g_0(v)),
}\end{equation}
whether $z\in A\cup B$ or not.
Take $P$ to be the collection of all vertex subsets $X\subseteq V_0$ such that $G[X]$
 is a component of $G\setminus (A\cup B)$ satisfying $\sum_{v\in X}f(v)\stackrel{2}{\not\equiv} d_G(X,B)$ (except those sets that $G[X]$ is bipartite and $g_0(v)=f_0(v)$ for all $v\in X$).
It is easy to check that
\begin{equation}\label{eq:2:thm:epsilon:parity}
{\sum_{X\in P}(\varepsilon d_G(X,A)+(1-\varepsilon) d_G(X,B))
\le \sum_{v\in A} \varepsilon d_G(v)+\sum_{v\in B} (1-\varepsilon) d_G(v)\, -d_G(A,B).
}\end{equation}
Define 
$P_c=\{X\in P: d_G(X,A)>0 \text{ and }d_G(X,B)>0\}$. Obviously, 
$$|P_c|= \sum_{X\in P_c}(\varepsilon+(1-\varepsilon))\le 
\sum_{X\in P_c}(\varepsilon d_G(X,A)+(1-\varepsilon) d_G(X,B)).$$
Set $P_a=\{X\in P: d_G(X,A)=0\}$ and 
$P_b=\{X\in P: d_G(X,B)=0\}$.
According to the definition, for every $X\in P_a$, we have 
$$\sum_{v\in X}f(v)\stackrel{2}{\not\equiv} 
d_G(X,B)=d_G(X)\stackrel{2}{\equiv}\sum_{v\in X}d_G(v),$$
which implies that $\sum_{v\in X}(d_G(v)-f(v))$ is odd.
Similarly, for every $X\in P_b$, $\sum_{v\in X}f(v)$ must be odd.
Thus by the assumption, we must have
$$|P_a|\le \sum_{X\in P_a}(1-\varepsilon) d_G(X) =
 \sum_{X\in P_a}(\varepsilon d_G(X,A)+(1-\varepsilon) d_G(X,B)),$$
and also
$$|P_b|\le \sum_{X\in P_b}\varepsilon d_G(X) =
 \sum_{X\in P_b}(\varepsilon d_G(X,A)+(1-\varepsilon) d_G(X,B)).$$
Therefore, 
\begin{equation}\label{eq:3:thm:epsilon:parity}
{|P|=|P_a|+|P_b|+|P_c|
\le \sum_{X\in P}(\varepsilon d_G(X,A)+(1-\varepsilon) d_G(X,B)).
}\end{equation}
According to Inequations~(\ref{eq:1:thm:epsilon:parity}), (\ref{eq:2:thm:epsilon:parity}), and (\ref{eq:3:thm:epsilon:parity}), 
one can conclude that 
$$\omega_{f,V_0}(G, A,B)=|P|<1+ \sum_{v\in A} f_0(v)+\sum_{v\in B} (d_{G}(v)-g_0(v))-d_G(A,B).$$
When both of the sets $A$ and $B$ are empty, the above-mentioned inequality must automatically hold, 
because $\omega_{f, V_0}(G, \emptyset,\emptyset)=0$ when $V_0=V(G)$. Thus the assertion follows from Lemma~\ref{lem:Kano-Matsuda-2001}.
}\end{proof}
\begin{remark}
{Note that the special case $V_0=V(G)$ of Theorem~\ref{thm:parity:version} alrerady studied in \cite[Theorem 6.2]{EquitableFactorizations-2022} and it can conclude Theorem 10 in \cite{Furuya-Kano-2024} with a stronger version. In fact, if $\varepsilon d_G(v)$ is an integer with the same parity of $f(v)$, then we must have $d_F(v)=\varepsilon d_G(v)$.
The special case $V_0=\emptyset$ of Theorem~\ref{thm:parity:version} can also conclude Theorem 5 in \cite{Furuya-Kano-2024}. 
}\end{remark}
The following corollary improves the edge-connectivity needed in Theorem~\ref{thm:delta1-deltam} for the special case $m=2$. Moreover, this is an improvement a result in \cite[Theorem 1.5]{Qin-Wu-2022} due to Qin and Wu (2022) who proved this corollary for $(\delta_1+\delta_2-1)$-edge-connected simple graphs. The special case $V_0=\emptyset$ of this result was already proved in \cite{Gupta-1978, Thomassen-1980}. 
\begin{cor}\label{parity:delta:k=2}
{Let $\delta_1$ and $\delta_2$ be two positive integers.
Let $G$ be a connected graph and let $V_0=\{v\in V(G): d_G(v)=\delta_1+\delta_2\}$.
If $\delta(G)\ge \delta_1+\delta_2$, then $G$ can be edge-decomposed into two factors $G_1$ and $G_2$ satisfying $\delta(G_1)\ge \delta_1$ and $\delta(G_2)\ge \delta_2$
provided that one of the following conditions holds:
\begin{enumerate}{

\item [$\bullet$] 
$\delta_1$ and $\delta_2$ are even.

\item [$\bullet$] 
$\delta_1$ is odd and $\delta_2$ is even, and $G$ is $V_0$-partially odd-$\lceil \frac{\delta_1+\delta_2}{\delta_1 }\rceil$-edge-connected.

\item [$\bullet$] 
$\delta_1$ and $\delta_2$ are odd, and $|V(G)|$ is even or $V_0\neq V(G)$, and $G$ is 
$V_0$-partially odd-$( \lceil \frac{\delta_1+\delta_2}{\min \{\delta_1,\delta_2\} }\rceil, V_0)$-edge-connected.
}\end{enumerate}
}\end{cor}
\begin{proof}
{Let $\varepsilon=\delta_1/(\delta_1+\delta_2)$ and define $f(v)=\delta_1$ for all $v\in V_0$.
Let $X$ be a subset of $V_0$. If $\sum_{v\in X}f(v)$ is odd, then $\delta_1$ and $|X|$ are odd, and so $d_G(X)\ge \lceil 1/\varepsilon \rceil=\lceil \frac{\delta_1+\delta_2}{\delta_1 }\rceil$. Notice that $d_G(X)$ is also odd when $\delta_1+\delta_2$ is odd. 
Likewise, if $\sum_{v\in X}(d_G(v)-f(v))$ is odd, then $\delta_2$ and $|X|$ are odd,
and so $d_G(X)\ge \lceil 1/(1-\varepsilon) \rceil=\lceil \frac{\delta_1+\delta_2}{\delta_2 }\rceil$.
Moreover, by the assumption, $\sum_{v\in V_0}f(v)$ is even or $V_0\neq V(G)$, regardless of both of $\delta_1$ and $\delta_2$ are odd or not.
Thus by applying Theorem~\ref{thm:parity:version},
the graph $G$ has a factor $G_1$ such that for each $v\in V_0$, $ d_{G_1}(v)= \varepsilon d_G(v) =\delta_1$, and 
for each $v\in V(G)\setminus V_0$, 
$\delta_1 \le \lfloor \varepsilon d_G(v) \rfloor \le d_{G_1}(v)\le \lfloor \varepsilon d_G(v) \rfloor +1$.
If we put $G_2=G \setminus E(G_1)$, then 
 for each $v\in V_0$, $ d_{G_2}(v)= (1-\varepsilon) d_G(v) =\delta_2$, and 
for each $v\in V(G)\setminus V_0$, 
$\delta_2 \le \lceil (1-\varepsilon) d_G(v)) \rceil-1 \le d_{G_2}(v)\le \lceil (1-\varepsilon) d_G(v) \rceil$ using the inequality $d_G(v)> \delta_1+\delta_2$. Hence the proof is completed.
}\end{proof}
The following corollary improves the edge-connectivity needed in Theorem~\ref{thm:Delta1-Deltam} for the special case $m=2$.
 The special case $V_0=\emptyset$ was already proved in \cite{Lovasz-1970}; see \cite{Correa-Matamala}. For simple graphs, it
can easily be derived using Vizing's Theorem, see \cite[Section 1]{Qin-Wu-2022}.
\begin{cor}\label{parity:Delta:k=2}
{Let $\Delta_1$ and $\Delta_2$ be two positive integers.
Let $G$ be a connected graph and let $V_0=\{v\in V(G): d_G(v)=\Delta_1+\Delta_2\}$.
If $\Delta(G)\ge \Delta_1+\Delta_2$, then $G$ can be edge-decomposed into two factors $G_1$ and $G_2$ satisfying $\Delta(G_1)\ge \Delta_1$ and $\Delta(G_2)\ge \Delta_2$
provided that one of the following conditions holds:
\begin{enumerate}{

\item [$\bullet$] 
$\Delta_1$ and $\Delta_2$ are even.

\item [$\bullet$] 
$\Delta_1$ is odd and $\Delta_2$ is even, and $G$ is $V_0$-partially odd-$\lceil \frac{\Delta_1+\Delta_2}{\Delta_1 }\rceil$-edge-connected.

\item [$\bullet$] 
$\Delta_1$ and $\Delta_2$ are odd, and $|V(G)|$ is even or $V_0\neq V(G)$, and $G$ is 
$V_0$-partially odd-$( \lceil \frac{\Delta_1+\Delta_2}{\min \{\Delta_1,\Delta_2\} }\rceil, V_0)$-edge-connected.
}\end{enumerate}
}\end{cor}
\begin{proof}
{Let $\varepsilon=\Delta_1/(\Delta_1+\Delta_2)$ and define $f(v)=\Delta_1$ for all $v\in V_0$.
Let $X$ be a subset of $V_0$. If $\sum_{v\in X}f(v)$ is odd, then $\Delta_1$ and $|X|$ are odd, and so $d_G(X)\ge \lceil 1/\varepsilon \rceil=\lceil \frac{\Delta_1+\Delta_2}{\Delta_1 }\rceil$. Notice that $d_G(X)$ is also odd when $\Delta_1+\Delta_2$ is odd. 
Likewise, if $\sum_{v\in X}(d_G(v)-f(v))$ is odd, then $\Delta_2$ and $|X|$ are odd,
and so $d_G(X)\ge \lceil 1/(1-\varepsilon) \rceil=\lceil \frac{\Delta_1+\Delta_2}{\Delta_2 }\rceil$.
Moreover, by the assumption, $\sum_{v\in V_0}f(v)$ is even or $V_0\neq V(G)$, regardless of both of $\Delta_1$ and $\Delta_2$ are odd or not.
Thus by applying Theorem~\ref{thm:parity:version},
the graph $G$ has a factor $G_1$ such that 
and for each $v\in V_0$, $ d_{G_1}(v)= \varepsilon d_G(v) =\Delta_1$, and
for each $v\in V(G)\setminus V_0$, 
$ \lceil \varepsilon d_G(v) \rceil -1 \le d_{G_1}(v)\le \lceil \varepsilon d_G(v) \rceil \le \Delta_1$.
If we set $G_2=G \setminus E(G_1)$, then 
 for each $v\in V_0$, $ d_{G_2}(v)= (1-\varepsilon) d_G(v) =\Delta_2$, and 
for each $v\in V(G)\setminus V_0$, 
$ \lfloor (1-\varepsilon) d_G(v)) \rfloor \le d_{G_2}(v)\le \lfloor (1-\varepsilon) d_G(v) \rfloor +1\le \Delta_2$
using the inequality $d_G(v)<\Delta_1+\Delta_2$. Hence the proof is completed.
}\end{proof}
\begin{cor}{\rm (\cite{Hoffman})}
{Let $\varepsilon $ be a real number with $0\le \varepsilon\le 1$. If $G$ is a bipartite graph,
then it has a factor $F$ such that for each vertex $v$, 
 $ | d_{F}(v)-\varepsilon d_G(v)|< 1.$
}\end{cor}
\begin{proof}
{Since there is no vertex set $X$ that $G[X]$ is non-bipartite, 
by applying Theorem~\ref{thm:parity:version} with setting $f(v)=\varepsilon d_G(v)$ for all $v\in V_0=\{v\in V(G): \varepsilon d_G(v)\text{ is an integer}\}$, the graph $G$ has a $V_0$-partial $f$-parity factor $F$ such that for each $v\in V_0$, $| d_{F}(v)-\varepsilon d_G(v)|<2$ which implies that $d_{F}(v)=\varepsilon d_G(v)$,
and for each $v\in V(G)\setminus V_0$, $| d_{F}(v)-\varepsilon d_G(v)|<1$. Hence the assertion holds.
}\end{proof}
\begin{cor}{\rm (\cite{Kano-Saito})}
{Let $\varepsilon $ be a real number with $0\le \varepsilon\le 1$. 
If $G$ is a graph, then it has a factor $F$ such that for each vertex $v$, 
 $| d_{F}(v)-\varepsilon d_G(v)|\le 1$.
}\end{cor}
\begin{proof}
{Apply Theorem~\ref{thm:parity:version} with $V_0=\emptyset$.
}\end{proof}
 The following corollary was studied in \cite{Modulo-Factors-Bounded} for the special case $V_0=V(G)$.
\begin{cor}\label{cor:1/2:modulo:V0}
{Let $G$ be a connected graph with $V_0\subseteq V(G)$ and let $f:V_0\rightarrow \mathbb{Z}_2$ be a mapping. Assume that
$ \sum_{v\in V_0}f(v)\stackrel{2}{\equiv}0$ when $V_0=V(G)$. If $G$ is $V_0$-partially $2$-edge-connected, 
then it has a factor $F$ such that for each vertex $v$, 
\begin{enumerate}{
\item [$\bullet$] For all $v\in V(G)\setminus V_0$, $\lfloor d_G(v)/2 \rfloor \le d_{F}(v)\le \lfloor d_G(v)/2\rfloor + 1$.
\item [$\bullet$] For all $v\in V_0$, $d_{F}(v)\stackrel{2}{\equiv}f(v)$ and $ | d_{F}(v)- d_G(v)/2| < 2$.
}\end{enumerate}
}\end{cor}
\begin{proof}
{Apply Theorem~\ref{thm:parity:version} with $\varepsilon=1/2$.
}\end{proof}
A generalized version of Corollaries 6.8 and 6.9 in \cite{EquitableFactorizations-2022} is given in the following corollary which completely confirms Conjecture~\ref{intro:conj:new} for the case $k=2$.
\begin{cor}\label{cor:even-factor:atleast2/3}
{Let $\varepsilon $ be a real number with $0\le \varepsilon< 1$.
Let $G$ be a graph with $V_0\subseteq V(G)$.
If $G$ is $V_0$-partially odd-$\lceil 1/(1-\varepsilon)\rceil $-edge-connected, then it has a factor $F$ such that
\begin{enumerate}{
\item [$\bullet$] For all $v\in V(G)\setminus V_0$, $\lfloor \varepsilon d_G(v) \rfloor \le d_{F}(v)\le \lfloor \varepsilon d_G(v) \rfloor + 1$.
\item [$\bullet$] For all $v\in V_0$, $d_{F}(v)$ is even and $ | d_{F}(v)-\varepsilon d_G(v)| < 2$.
}\end{enumerate}
}\end{cor}
\begin{proof}
{We may assume that $G$ is connected. Hence it is enough to apply Theorem~\ref{thm:parity:version} with setting $f=0$. Note that for all $X\subseteq V_0$, $\sum_{v\in X}f(v)$ is even, and $\sum_{v\in X}(d_G(v)-f(v))$ and $d_G(X)$ have the same parity. 
}\end{proof}
When the main graph has no odd edge-cuts in $V_0$, one can derive the following simpler version of 
Corollary~\ref{cor:even-factor:atleast2/3}.
\begin{cor}\label{cor:Eulerian}
{Let $\varepsilon $ be a real number with $0< \varepsilon< 1$.
If $G$ is a graph and $V_0$ is a set of some with even vertices, then it admits a factor $F$ such that for each vertex $v$, 
\begin{enumerate}{
\item [$\bullet$] For all $v\in V(G)\setminus V_0$, $\lfloor \varepsilon d_G(v) \rfloor \le d_{F}(v)\le \lfloor \varepsilon d_G(v) \rfloor + 1$.
\item [$\bullet$] For all $v\in V_0$, $d_{F}(v)$ is even and $ | d_{F}(v)-\varepsilon d_G(v)| < 2$.
}\end{enumerate}
}\end{cor}
\begin{proof}
{It is enough to apply Corollary~\ref{cor:even-factor:atleast2/3}, and use the fact that for any $X\subseteq V_0$, $d_G(X)$ is not odd.
}\end{proof}

%
%
%
%
%
%
%
\section{Applications to almost even factorizations}
\label{thm:almost even factorizations}
In this section, we are going to partially confirm weaker versions of Conjecture~\ref{intro:conj:new} by replacing another bound on vertex degrees (see Corollaries~\ref{cor:V1:V2:epsiloni} and~\ref{cor:G0:edge-decomposition:Eulerian}).
We begin with the following theorem which is inspired by Theorem 4 in \cite{Correa-Goemans}.
\begin{thm}\label{thm:6}
{Let $\varepsilon_0,\ldots,\varepsilon_k$ be $k+1$ nonnegative real numbers with $\varepsilon_0+\cdots +\varepsilon_k=1$.
Let $G$ be a graph and let $V_1,V_2$ be a bipartition of $V(G)$. 
If $G_0$ is a factor of $G$ satisfying $|d_{G_0}(v)-\varepsilon_0 d_G(v)|\le q$ for all $v\in V_q$, then $G\setminus E(G_0)$ can be edge-decomposed into $k$ factors $G_1,\ldots, G_k$ satisfying the following properties:
\begin{enumerate}{
\item [$\bullet$] For all $v\in V_q$, $|d_{G_i}(v)-\varepsilon_i d_G(v)|<3q$, where $1\le i\le k$.
\item [$\bullet$] For any $v\in V_2$, there exists at most one index $j_v\in \{1,\ldots, k\}$ with $d_{G_{j_v}}(v)$ odd.
}\end{enumerate}
}\end{thm}
\begin{proof}
{We repeat the arguments stated in the proof of Theorem 4 in \cite{Correa-Goemans} and improve it by applying Corollary~\ref{cor:Eulerian} in its process. We may assume that all $\varepsilon_i$ with $1\le i\le k$ are positive (by applying an induction on $k$). 
Let $L_0 = \{0,\ldots,k\}$. We construct a binary tree $T$ with $k $ internal nodes
and $k+1$ leaves, each node being labelled by a subset of $L_0$. The root is labelled with $L$,
and the $k$ leaves are labelled by a distinct singleton subset of $L_0$.
The root is labelled with $L_0$,
and the $k$ leaves are labelled by a distinct singleton subset of $L_0$. If an internal node is
labelled with $N$, then its two children are labelled with $I$ and
and $N \setminus I$, where $I$, $N \setminus I$ is
the most balanced number partition of $N$; that means, $I$ is such that 
$| \varepsilon (I )-\varepsilon (N \setminus I )|$
is minimized (for a set $S$, $\varepsilon (S )$ denotes $\sum_{ i \in S} \varepsilon_i $).

With every node with label $I$, we also associate a factor $G_I$ of $G$. We first set
$G_{L_0}= G$ and $G_L= G\setminus E(G_0)$. 
Given $G_N$ for an internal node $N$, we obtain $G_I$ and its complement $G_{N \setminus I}$ in $G_N$
corresponding to both children of $G_N$ by applying Corollary~\ref{cor:Eulerian} to the graph $G_N$ with setting
$\varepsilon = \varepsilon (I )/\varepsilon (N )$.
Note that for each $v\in V_q$, $|d_{G_I}(v)-\frac{\varepsilon (I )}{\varepsilon (N )} d_{G_N}(v)|\le q$ and 
$|d_{G_{N \setminus I}}(v)-\frac{\varepsilon (N \setminus I)}{\varepsilon (N )} d_{G_N}(v)|\le q$.
In addition, if $v\in V_2$, then $d_{G_I}(v)$ and $d_{G\setminus I}(v)$ are even when $d_{G_N}(v)$ is even,
and exactly one of them is odd when $d_{G_N}(v)$ is odd.
The leaves are thus associated with subgraphs
$G_{\{i\}}$ which make a factorization of $G$.
 We claim that $G_{\{i\}}$ satisfies the required properties for $G_i $.

Fix a vertex $v \in V_q$ 
and an index $i \in L_0\setminus \{0\}$. Let $\{ i\} = A_t \subsetneq A_{t-1} \subsetneq \cdots \subsetneq A_0=L_0$ be the labels on the path from
the leaf $\{i\}$ to the root. We now derive an upper bound on $d_{G_i}(v )$ (and we could
proceed similarly for the lower bound). From Corollary~\ref{cor:Eulerian}, we have that
\begin{eqnarray*}
d_{G_i}(v ) = d_{G_{A_t}}(v) & \le & 
\frac{\varepsilon (A_t)}{\varepsilon (A_{t-1})} d_{G_{A_{t-1}}}(v) +q,\\
& \le & 
\frac{\varepsilon (A_t)}{\varepsilon (A_{t-1})} \big(\frac{\varepsilon (A_{t-1})}{\varepsilon (A_{t-2})} d_{G_{A_{t-2}}}(v) +q\big)+q,\\
& \le & 
\frac{\varepsilon (A_t)}{\varepsilon (A_{t-1})} \Big(\frac{\varepsilon (A_{t-1})}{\varepsilon (A_{t-2})}
 \Big(\cdots \Big( \frac{\varepsilon (A_1)}{\varepsilon (A_{0})} d_{G }(v) +q\Big) \cdots \Big)+q\Big)+q,\\
& = & 
\frac{\varepsilon (A_t)}{\varepsilon (A_{0})} d_{G }(v) +
\big(\frac{\varepsilon (A_t)}{\varepsilon (A_1)}+\frac{\varepsilon (A_t)}{\varepsilon (A_2)}+\cdots + 
\frac{\varepsilon (A_t)}{\varepsilon (A_{t-1})}+1\big)q,\\
& = & 
\varepsilon_i d_{G }(v) +c_i (v )q,
\end{eqnarray*}
where $c_i (v )=\frac{\varepsilon (A_t)}{\varepsilon (A_1)}+\frac{\varepsilon (A_t)}{\varepsilon (A_2)}+\cdots + 
\frac{\varepsilon (A_t)}{\varepsilon (A_{t-1})}+1$. 
Let $\varepsilon_j=\min _{i' \in A_{t-1}} \varepsilon_{i'}$ that $j\in A_{t-1}$.
Note that $\varepsilon_i= \varepsilon(A_t) \ge \varepsilon_j $.
Thus we have $\varepsilon (A_{t - 1}) \ge \varepsilon_i + \varepsilon_j$. 
In general, when considering $A_s$ with $0 < s< t$, we split it into $A_{s+1}$ and $A_s \setminus A_{s+1}$,
 while we could have split it into $A_{s+1} \setminus\{ j \}$ and the rest.
This implies that $\varepsilon (A_s ) \ge 2\varepsilon (A_{s+1}) - \varepsilon_j$; otherwise, we have
$|\varepsilon (A_s \setminus A_{s+1})-\varepsilon (A_{s+1})|=|\varepsilon (A_s)-2\varepsilon (A_{s+1})| >
|\varepsilon (A_s)-2\varepsilon (A_{s+1}) +2\varepsilon_j| =|\varepsilon ((A_s \setminus A_{s+1}) \cup \{j\})-\varepsilon (A_{s+1}\setminus \{j\})|$, which is a contradiction.
Using this repeatedly, we get $\varepsilon (A_{t-2}) \ge 2\varepsilon_i + \varepsilon_j$, $\varepsilon (A_{t-3}) \ge 4\varepsilon_i+ \varepsilon_j$,
and, generally, $\varepsilon (A_{s}) \ge 2^{t-s-1} \varepsilon_i +\varepsilon_j > 2^{t-s-1} \varepsilon_i $ when $s>0$. 
Note also that 
$\varepsilon (A_{1})\ge \varepsilon (A_{2})\ge 2^{t-s-1} \varepsilon_i +\varepsilon_j > 2^{t-3} \varepsilon_i $.
Thus the bound becomes
$$c_i (v ) <\frac{\varepsilon_i}{ 2^{t-3}\varepsilon_i}q+
(\frac{\varepsilon_i}{2^{t-3}\varepsilon_i}+
\cdots+\frac{\varepsilon_i}{4\varepsilon_i}+\frac{\varepsilon_i}{2\varepsilon_i}+\frac{\varepsilon_i}{\varepsilon_i}+1)q = 3q.$$
A proof of the lower bound on $d_{G_i}(v )$ is identical. This completes the proof.
}\end{proof}
Now, we are in a position to partially confirm Conjecture~\ref{intro:conj:new} for highly odd-edge-connected graphs based on Corollary~\ref{cor:even-factor:atleast2/3}.
\begin{cor}\label{cor:V1:V2:epsiloni}
{Let $\varepsilon_1,\ldots,\varepsilon_k$ be $k$ nonnegative real numbers with $\varepsilon_1+\cdots +\varepsilon_k=1$.
If $G$ is a graph and $V_1,V_2$ is a bipartition of $V(G)$, then $G$ can be edge-decomposed into $k$ factors $G_1,\ldots, G_k$ satisfying the following properties:
\begin{enumerate}{
\item [$\bullet$] For all $v\in V_q$, $|d_{G_i}(v)-\varepsilon_i d_G(v)|<3q$, where $1\le i\le k$.
\item [$\bullet$] For any $v\in V_2$, there exists at most one index $j_v\in \{1,\ldots, k\}$ with $d_{G_{j_v}}(v)$ odd.
}\end{enumerate}
Furthermore, if $G$ is odd-$\lceil 1/\varepsilon_1\rceil $-edge-connected, then $j_v=1$ for all odd-degree vertices $v$.
}\end{cor}
\begin{proof}
{The first assertion is the same Theorem~\ref{thm:6} by setting $\varepsilon_0=0$ and $E(G_0)=\emptyset$. If $G$ is odd-$\lceil 1/\varepsilon_1\rceil $-edge-connected, then by applying Corollary~\ref{cor:even-factor:atleast2/3} with setting $V_0=V(G)$, the graph $G$ admits a factor $G_0$ (the complement of $F$) such that for all $v\in V(G)$, $d_{G_0}(v)\stackrel{2}{\equiv}d_G(v)$ and $ | d_{G_0}(v)-\varepsilon_1 d_G(v)|< 2$. It is enough to apply Theorem~\ref{thm:6} with replacing $k-1$ instead of $k$.
}\end{proof}
\begin{remark}
{Note that if $G$ is $\max\{\lceil 1/\varepsilon_1\rceil, \lceil 1/(1-\varepsilon_1)\rceil $-edge-connected, then Corollary~\ref{cor:V1:V2:epsiloni} (or Corollary~\ref{cor:G0:edge-decomposition:Eulerian}) can be improved by imposing the condition $j_v=1$ for all $v\in Q$, and $j_v\neq 1$ for all $v\in Q'$, where $Q$ and $Q'$ are two disjoint sets of odd-degree vertices of $V_2$ with even size. For this purpose, we only need to apply Theorem~\ref{thm:parity:version} to find a factor $G_0$ such that for each vertex $v\in V_2$, $d_{G_0}(v)$ is odd if and only if $v\in Q$.
}\end{remark}
\begin{cor}{\rm (\cite{EquitableFactorizations-2022})}\label{intro:even-graph:epsilon:2}
{Let $\varepsilon_1,\ldots,\varepsilon_k$ be $k$ positive real numbers with $\varepsilon_1+\cdots +\varepsilon_k=1$.
If $G$ is an even graph, then it can be edge-decomposed into $k$ even factors $G_1,\ldots, G_k$ such that for each $v\in V(G_i)$,
$1\le i\le k$, $|d_{G_i}(v)-\varepsilon_i d_G(v)|<6$.
}\end{cor}
\begin{proof}
{Apply Corollary~\ref{cor:V1:V2:epsiloni} with setting $V_1=\emptyset$ and $V_2=V(G)$.
}\end{proof}
\begin{cor}{\rm (\cite{Correa-Matamala, Correa-Goemans})}\label{cor:Correa-Goemans}
{Let $\varepsilon_1,\ldots,\varepsilon_k$ be $k$ nonnegative real numbers with $\varepsilon_1+\cdots +\varepsilon_k=1$.
If $G$ is a graph, then it can be edge-decomposed into $k$ factors $G_1,\ldots, G_k$ such that for each $v\in V(G_i)$, $1\le i\le k$, $|d_{G_i}(v)-\varepsilon_i d_G(v)|<3$.
}\end{cor}
\begin{proof}
{Apply Corollary~\ref{cor:V1:V2:epsiloni} with setting $V_1=V(G)$ and $V_2=\emptyset$ .
}\end{proof}
We shall below show that the degree restriction in Corollary~\ref{cor:Correa-Goemans} cannot replaced by $|d_{G_i}(v)-\varepsilon_i d_G(v)|\le 1$. More precisely, we give negative answers to the following question raised by Correa and Matamala (2008).
In fact, these examples contain multiple edges and small edge-cuts, so
 it remains to decide whether the answer is true for simple graphs or highly edge-connected graphs together with an order restriction.
\begin{que}{\rm (\cite[Section~4]{Correa-Matamala})}\label{intro:que}
{Let $\varepsilon_1,\ldots,\varepsilon_k$ be $k$ nonnegative real numbers with $\varepsilon_1+\cdots +\varepsilon_k=1$.
 Can a given graph $G$ be edge-decomposed into $k$ factors $G_1,\ldots, G_k$ satisfying $|d_{G_i}(v)-\varepsilon_i d_G(v)|\le 1$ for all $v\in V(G_i)$ with $1\le i\le k$?
}\end{que}
We would like to introduce a more general structure to avoid the existence of such factorizations in the following observation.
To see some primary examples, one can consider the general graph consisting of a single vertex incident with $\lceil (k+1)/2\rceil$ loops or graphs having three vertices with degree $k+1$ for which there are at most two edges with exactly one end in them. 
\begin{observ}
{Let $G$ be a graph and let $X$ be a set of vertices with degree not divisible by $k$. 
If there exists a bipartition $V_0, V_1$ of $X$ with 
$\sum_{v\in V_0}\lfloor d_G(v)/k\rfloor +\sum_{v\in V_1}\lceil d_G(v)/k\rceil$ odd, and 
$$\sum_{v\in V_0}[d_G(v)]_k+\sum_{v\in V_1}[k-d_G(v)]_k\le k-1-d_G(X),$$
then $G$ cannot be edge-decomposed into $k$ factors $G_1,\ldots, G_k$ satisfying $|d_{G_i}(v)- d_G(v)/k|\le 1$ for all $v\in X$ 
with $1\le i\le k$. 
}\end{observ}
\begin{proof}
{Suppose, to the contrary, that $G$ can be edge-decomposed into $k$ factors $G_1,\ldots, G_k$ satisfying $|d_{G_i}(v)- d_G(v)/k|\le 1$ for all $v\in X$ with $1\le i\le k$. Since $\sum_{v\in V_0}\lfloor d_G(v)/k\rfloor +\sum_{v\in V_1}\lceil d_G(v)/k\rceil$ is odd, 
if $d_{G_i}(X)=0$ for a factor $G_i$, then there exists a vertex $v\in V_0$ with $d_{G_i}(v)\neq \lfloor d_G(v)/k\rfloor $ or
 there is a vertex $v\in V_1$ with $d_{G_i}(v)\neq \lceil d_G(v)/k\rceil $.
Therefore, $\sum_{v\in V_0}[d_G(v)]_k+\sum_{v\in V_1}[k-d_G(v)]_k=\sum_{1\le i\le k} (\sum_{v\in V_0}(d_{G_i}(v)- \lfloor d_G(v)/k\rfloor)+\sum_{v\in V_1}(\lceil d_G(v)/k\rceil -d_{G_i}(v)))\ge k-d_G(X)$. This is a contradiction.
}\end{proof}
In the following theorem, we invigorate Theorem~\ref{thm:6} when almost all of $\varepsilon_i$ are the same number.
This result enables us to improve Corollary~\ref{cor:Hilton:generalized} for odd-edge-connected graphs based on Corollary~\ref{cor:even-factor:atleast2/3} or even inductively prove it based on Corollary~\ref{cor:Eulerian}. 
\begin{thm}\label{thm:epsilon:k}
{Let $G$ be a graph and let $\varepsilon$ be a real number with $0< \varepsilon < 1$. 
If $G_0$ is a factor of $G$ satisfying $|d_{G_0}(v)-\varepsilon d_G(v)|< 2$ for all $v\in V(G)$, then $G\setminus E(G_0)$ can be edge-decomposed into $k$ factors $G_1,\ldots, G_k$ satisfying the following properties:
\begin{enumerate}{
\item [$\bullet$] For all $v\in V(G)$, $|d_{G_i}(v)-\frac{1-\varepsilon} {k}d_G(v)|<2$, where $1\le i\le k$.
\item [$\bullet$] For any $v\in V(G)$, there exists at most one index $j_v\in \{1,\ldots, k\}$ with $d_{G_{j_v}}(v)$ odd.
}\end{enumerate}
}\end{thm}
\begin{proof}
{By Corollary~\ref{cor:Hilton:generalized}, there is a factorization $G_1,\ldots, G_k$ of $G\setminus E(G_0)$
 such that for each vertex $v$, $ | d_{G_i}(v) - d_{G_j}(v)| \le 2$, where $i, j\in \{1,\ldots, k\}$.
Moreover, there is at most one index $j_v\in \{1,\ldots, k\}$ with $d_{G_{j_v}}(v)$ odd.
If there exists an index $j\in \{1,\ldots, k\}$ such that
$d_{G_j}(v) \ge \frac{1-\varepsilon}{k}d_G(v)+2$,
then for all graphs $G_i$ with $i\in \{1,\ldots, k\}$, $d_{G_i}(v)\ge d_{G_j}(v)-2\ge \frac{1-\varepsilon}{k}d_G(v)$.
Therefore, we must have
$ d_G(v)= (1-\varepsilon) d_G(v)+2+\varepsilon d_G(v)-2
< \sum_{i=1,\ldots, k}d_{G_{i}}(v)+d_{G_0}(v)= d_G(v)$,
which is impossible.
If there exists an index $j\in \{1,\ldots, k\}$ such that
$d_{G_j}(v) \le \frac{1-\varepsilon}{k}d_G(v)-2$,
then for all graphs $G_i$ with $i\in \{1,\ldots, k\}$, $d_{G_i}(v)\le d_{G_j}(v)+2\le \frac{1-\varepsilon}{k}d_G(v)$.
Therefore, we must have
$ d_G(v)= (1-\varepsilon) d_G(v)-2+\varepsilon d_G(v)+2
> \sum_{i=1,\ldots, k}d_{G_{i}}(v)+d_{G_0}(v)= d_G(v)$,
which is again impossible.
Hence the assertion holds.
}\end{proof}
The following corollary completely confirm Conjecture~\ref{intro:conj:new} when almost all of $\varepsilon_i$ are the same number.
\begin{cor}\label{cor:G0:edge-decomposition:Eulerian}
{Let $\varepsilon_1,\ldots,\varepsilon_k$ be $k$ nonnegative real numbers with $\varepsilon_1+\cdots +\varepsilon_k=1$.
If $G$ is a graph and $\varepsilon_2=\cdots =\varepsilon_k$, then 
$G$ can be edge-decomposed into factors $G_1,\ldots, G_k$ such that
\begin{enumerate}{
\item [$\bullet$] For all $v\in V(G)$, $|d_{G_i}(v)-\varepsilon_i d_G(v)|<2$, where $1\le i\le k$.

\item [$\bullet$] 
For each $v\in V(G)$, there is at most one index $j_v$ with $d_{G_{j_v}}(v)$ odd.
}\end{enumerate}
Furthermore, if $G$ is odd-$\lceil 1/\varepsilon_1\rceil $-edge-connected, then $j_v=1$ for all odd-degree vertices $v$.
}\end{cor}
\begin{proof}
{For the first assertion, by applying Corollary~\ref{cor:Eulerian} with setting $V_0=\{v\in V(G): d_G(v) \text { is even}\}$, the graph $G$ admits a factor $G_0$ such that for all $v\in V(G)$, $ | d_{G_0}(v)-\varepsilon_1 d_G(v)|< 2$, and for all $v\in V_0$, $d_{G_0}(v) $  is even.
Now, it is enough to apply Theorem~\ref{thm:epsilon:k} with replacing $k-1$ instead of $k$. If $G$ is odd-$\lceil 1/\varepsilon_1\rceil $-edge-connected, by applying Corollary~\ref{cor:even-factor:atleast2/3} with setting $V_0=V(G)$, the graph $G$ admits a factor $G_0$ (the complement of $F$) such that for all $v\in V(G)$, $d_{G_0}(v)\stackrel{2}{\equiv}d_G(v)$ and $ | d_{G_0}(v)-\varepsilon_1 d_G(v)|< 2$.
Again, it is enough to apply Theorem~\ref{thm:epsilon:k} with replacing $k-1$ instead of $k$.
}\end{proof}
The next corollary improves Corollary~\ref{cor:Hilton:generalized} for highly odd-edge-connected graphs by fixing the index $j_v$.

\begin{cor}\label{intro:thm:Hilton:generalized}
{Every graph $G$ can be edge-decomposed into $k$ factors $G_1,\ldots, G_k$ satisfying the following properties:
\begin{enumerate}{
\item [$\bullet$] For each $v\in V(G_i)$, $|d_{G_i}(v)-d_{G}(v)/k|<2$.

\item [$\bullet$] For each $v\in V(G)$, there is at most one index $j_v$ with $d_{G_{j_v}}(v)$ odd (in particular, there is not such index when $d_G(v)$ is even). 
}\end{enumerate}
Furthermore, if $G$ is odd-$k$-edge-connected, then $j_v=1$ for all odd-degree vertices $v$.
}\end{cor}
\begin{proof}
{Apply Corollary~\ref{cor:G0:edge-decomposition:Eulerian} by setting $\varepsilon_i=1/k$.
}\end{proof}
%
%
%
%
%
%
%
%
%


\end{document}